\documentclass{compositio}
\title{Infinitesimal structure of the pluricanonical double ramification locus}

\usepackage{amsmath, amssymb, mathrsfs, amsthm, shorttoc, stmaryrd}
\usepackage{mathtools} 
\usepackage{hyperref}
\usepackage{xcolor}
\hypersetup{
colorlinks,
linkcolor={black},
citecolor={blue!50!black},
urlcolor={blue!80!black}
}

\usepackage[all]{xy}

\usepackage{tabularx}
\usepackage{longtable}
\numberwithin{equation}{subsection}

\usepackage{enumitem}

\usepackage{cleveref}
\crefformat{equation}{(#2#1#3)}
\let\oref\ref
\AtBeginDocument{\renewcommand{\ref}[1]{\cref{#1}}}

\usepackage{pgf,tikz}
\usetikzlibrary{matrix, calc, arrows}
\DeclareMathOperator{\coker}{coker}

\usepackage{tikz-cd} 

\makeatletter
\newcommand*{\doublerightarrow}[2]{\mathrel{
  \settowidth{\@tempdima}{$\scriptstyle#1$}
  \settowidth{\@tempdimb}{$\scriptstyle#2$}
  \ifdim\@tempdimb>\@tempdima \@tempdima=\@tempdimb\fi
  \mathop{\vcenter{
    \offinterlineskip\ialign{\hbox to\dimexpr\@tempdima+1em{##}\cr
    \rightarrowfill\cr\noalign{\kern.5ex}
    \rightarrowfill\cr}}}\limits^{\!#1}_{\!#2}}}
\newcommand*{\triplerightarrow}[1]{\mathrel{
  \settowidth{\@tempdima}{$\scriptstyle#1$}
  \mathop{\vcenter{
    \offinterlineskip\ialign{\hbox to\dimexpr\@tempdima+1em{##}\cr
    \rightarrowfill\cr\noalign{\kern.5ex}
    \rightarrowfill\cr\noalign{\kern.5ex}
    \rightarrowfill\cr}}}\limits^{\!#1}}}
\makeatother

\newcommand{\on}[1]{\operatorname{#1}}
\newcommand{\bb}[1]{{\mathbb{#1}}}

\newcommand{\ca}[1]{{\mathcal{#1}}}
\newcommand{\bd}[1]{{\mathbf{#1}}}

\newcommand{\ul}[1]{{\underline{#1}}}

\newcommand{\hra}{\hookrightarrow}
\newcommand{\sub}{\subseteq}
\newcommand{\tra}{\rightarrowtail}

\theoremstyle{definition}
\newtheorem{definition}{Definition}[section]

\newtheorem{situation}[definition]{Situation}

\theoremstyle{plain}

\newtheorem{proposition}[definition]{Proposition}
\newtheorem{lemma}[definition]{Lemma}
\newtheorem{theorem}[definition]{Theorem}
\newtheorem{corollary}[definition]{Corollary}

\theoremstyle{remark}

\newtheorem{remark}[definition]{Remark}

\newtheorem{example}[definition]{Example}

\usepackage{letltxmacro}
\LetLtxMacro{\phiorig}{\phi}
\renewcommand{\phi}{\varphi}

\newcommand{\cecH}{\check{\on H}}

\newcommand{\loz}{\lozenge}

\newcommand{\barsigma}{\bar{\sigma}}

\newcommand{\tanmap}{\psi}

\author{David Holmes}
\email{holmesdst@math.leidenuniv.nl}
\address{Mathematisch Instituut, Universiteit Leiden, Postbus 9512, 2300 RA Leiden, Netherlands}

\author{Johannes Schmitt }
\email{schmitt@math.uni-bonn.de}
\address{Mathematical Institute, University of Bonn, Endenicher Allee 60, 53115 Bonn, Germany}

\thanks{The first-named author was partially supported by NWO grant 613.009.103/2380. The second author was
supported by the grant SNF-200020162928 and has received funding from the European Research Council (ERC)
under the European Union Horizon 2020 research and innovation programme
(grant agreement No 786580). During the last phase of the project, the second author profited from the SNF Early Postdoc.Mobility grant 184245 and also wants to thank the Max Planck Institute for Mathematics in Bonn for its hospitality. }

\date{\today}

\newcounter{nootje}
\newcommand{\beq}{\begin{equation}}
\newcommand{\eeq}{\end{equation}}
\newcommand{\beqs}{\begin{equation*}}
\newcommand{\eeqs}{\end{equation*}}

\renewcommand{\k}{k}

\tikzset{
  symbol/.style={
    draw=none,
    every to/.append style={
      edge node={node [sloped, allow upside down, auto=false]{$#1$}}}
  }
}

\begin{document}

\begin{abstract} 
We prove that a formula for the `pluricanonical' double ramification cycle proposed by Janda, Pandharipande, Pixton, Zvonkine, and the second-named author is in fact the class of a cycle constructed geometrically by the first-named author. Our proof proceeds by a detailed explicit analysis of the deformation theory of the double ramification cycle, both to first and to higher order. 
\end{abstract}

\maketitle


\tableofcontents

\newcommand{\Mtildes}{ \widetilde{\ca M}^\Sigma}
\newcommand{\sch}[1]{\textcolor{blue}{#1}}

\newcommand{\Mbar}{\overline{\ca M}}
\newcommand{\MD}{\ca M^\blacklozenge}
\newcommand{\Md}{\ca M^\lozenge}
\newcommand{\DRL}{\operatorname{DRL}}
\newcommand{\DR}{\operatorname{DR}}
\newcommand{\DRC}{\operatorname{DRC}}
\newcommand{\isom}{\stackrel{\sim}{\longrightarrow}}
\newcommand{\Ann}[1]{\on{Ann}(#1)}
\newcommand{\fm}{\mathfrak m}
\newcommand{\Mdk}{\Mbar^{\m, 1/\k}}
\newcommand{\field}{K}
\newcommand{\Mdm}{\Mbar^\m}
\newcommand{\m}{{\bd m}}

\section{Introduction}
Inside the moduli space $\ca M_{g,n}$ of smooth pointed curves $(C,p_1, \ldots, p_n)$ there are natural closed subsets 
\begin{equation} \label{eqn:strataofdiff} \ca H_g^\k(\m) = \left\{(C,p_1, \ldots, p_n) : \omega_C^{\otimes \k} \cong \mathcal{O}_C\left(\sum_{i=1}^n m_i p_i\right)\right\} \subset \ca M_{g,n},\end{equation}
where $\m=(m_1, \ldots, m_n) \in \mathbb{Z}^n$ is a vector of integers summing to $\k(2g-2)$. Since the above isomorphism of line bundles is equivalent to the  existence of a meromorphic $\k$-differential on $C$ with zeros and poles at the points $p_i$ with specified orders $m_i$, these subsets are called \emph{strata of meromorphic $\k$-differentials}. These strata appear naturally in algebraic geometry, the theory of flat surfaces and Teichm\"uller dynamics and have been studied intensively in the past, see the surveys \cite{ZorichFlat, wright, chentmdynamics} and the references therein. Motivated by problems in symplectic geometry, Eliashberg asked whether there was a natural way to extend these strata and their fundamental classes to the Deligne-Mumford-Knudsen compactification $\Mbar_{g,n}$ and how to compute the resulting cycle class. 

For $\k=0$ there are two geometric avenues to defining such an extension. The first is via relative Gromov-Witten theory and the space of rubber maps to $\mathbb{P}^1$ (\cite{liruan, Li2002A-degeneration-, Li2001Stable-morphism,Graber2005Relative-virtua}). This is based on the observation that for a smooth curve $C$, a meromorphic $0$-differential on $C$ as above corresponds to a morphism $C \to \mathbb{P}^1$ with given ramification profiles over $0,\infty$. The second series of approaches, viable for any $\k \geq 0$, uses that $\ca H_g^\k(\m)$ can be obtained by pulling back the zero section $e$ of the universal Jacobian $\ca J \to \ca M_{g,n}$ via the Abel-Jacobi section
\[\sigma : \ca M_{g,n} \to \ca J, (C,p_1, \ldots, p_n) \mapsto \omega^\k \left(-\sum_{i=1}^n m_i p_i \right).\] 
The map $\sigma$ does not extend naturally to $\Mbar_{g,n}$, but various geometric extensions of its domain and target have been proposed that yield cycles on $\Mbar_{g,n}$ (\cite{Kass2017The-stability-s, Holmes2017Jacobian-extens, Marcus2017Logarithmic-com, Holmes2017Extending-the-d, abreupacini}). These constructions all produce the same cycle class on $\Mbar_{g,n}$, which we denote $\overline{\DRC}$; an overview of one construction is given in \ref{sec:extending_AJ}. 

Pixton \cite{Pixton} defined a class $P_g^{g,\k}(\tilde \m)$ in the tautological ring of $\Mbar_{g,n}$. The equality 
\begin{equation} \label{eqn:Hconjecture2017}
\overline{\DRC} = 2^{-g} P_g^{g,k}(\tilde \m)
\end{equation}
was conjectured by Pixton for $k=0$, and in \cite{Holmes2017Extending-the-d} for all $k$. An introduction to Pixton's formula in the case $k=0$ can be found in \cite[Section 6.4]{calculusmodcur}, and in the general case in \cite{Janda2016Double-ramifica}. The conjectured equality $2^{-g}P_g^{g,\k}(\tilde \m) = \overline{\DRC}$ for $k=0$ was proven in \cite{Janda2016Double-ramifica}. Since this preprint was posted, the equality $2^{-g}P_g^{g,\k}(\tilde \m) = \overline{\DRC}$ for all $k$ has been established in \cite{bae2020pixtons}. 


A new geometric approach to extending the cycle appears for $k \geq 1$: assuming that one of the integers $m_i$ is negative or not divisible by $\k$, the papers \cite{Farkas2016The-moduli-spac, Schmitt2016Dimension-theor} define a cycle $H_{g,\m}^\k$ obtained as a weighted\footnote{This means the fundamental classes of the components of $\widetilde{\mathcal H}_g^\k(\m)$ are summed with explicit positive integer weights, see \ref{eq:weighted_fun_formula}.} fundamental class of an explicit closed subset $\widetilde{\mathcal H}_g^\k(\m) \subset \Mbar_{g,n}$ extending $\ca H_g^\k(\m)$, and propose

\vspace{0.3 cm}
\noindent \textbf{Conjecture A (\cite{Farkas2016The-moduli-spac, Schmitt2016Dimension-theor})} Let $k \geq 1$ and $\m=(m_1, \ldots, m_n) \in \mathbb{Z}^n$ with $m_1 + \ldots + m_n = \k(2g-2)$. Assume that one of the $m_i$ is negative or not divisible by $\k$ and let $\tilde \m = (m_1 +\k, \ldots, m_n +\k)$. Then
 \[H_{g,\m}^\k = 2^{-g} P_{g}^{g,\k}(\tilde \m) \in A^g(\Mbar_{g,n}).\] 
 At the time these papers were written the geometric class $\overline{\DRC}$ had not been defined for $k>0$; from our current perspective it seems most natural simply to conjecture that all three classes ($\overline{\DRC}$, $2^{-g}P_g^{g,\k}(\tilde \m)$ and $H_{g,\m}^\k$) are equal whenever they are defined. 
 
The main result of our paper is the following.
\begin{theorem} \label{thm:main_intro_basic}
For $\k \geq 1$ and at least one of the $m_i$ either negative or not divisible by $\k$, the equality
\begin{equation}\label{eq:thm_main_basic}
 \overline{\DRC} = H_{g,\m}^\k
\end{equation}
holds in the Chow ring of $\Mbar_{g,n}$. 
\end{theorem}
Combined with the recent proof of $2^{-g}P_g^{g,\k}(\tilde \m) = \overline{\DRC}$ in \cite{bae2020pixtons}, this yields a proof of Conjecture A.

In fact, not only do we prove the equality \ref{eq:thm_main_basic} of cycle classes, but as a byproduct of our proof we demonstrate how the weights in the weighted fundamental class $H_{g,\m}^\k$ arise from intersection multiplicities of the Abel-Jacobi section with the zero section in the construction of \cite{Holmes2017Extending-the-d}. Further, with a little extra work our method allows us to compute not only the multiplicities of the cycle, but even give a presentation for the Artin local rings at generic points of the double ramification locus (\ref{thm:local_ring_description}). 

In the remainder of this introduction, we recall the definition of $H_{g,\m}^\k$ and the construction of $\overline{\DRC}$ from \cite{Holmes2017Extending-the-d} before stating a more refined version of our main result in \ref{Sec:mainresult}. We give a sketch of the proof in \ref{Sect:introproofsketch} and discuss some future research directions in \ref{Sect:relandoutlook}. We finish by giving a more detailed overview of the relations between the various approaches for defining the extended cycles that we discussed before.

\subsection{\texorpdfstring{The moduli space of twisted $\k$-differentials}{The moduli space of twisted k-differentials}}
A first idea for extending the stratum $\ca H_g^\k(\m)$ of $\k$-differentials is to consider its closure $\overline{\ca H}_g^\k(\m) \subset \Mbar_{g,n}$. Stable curves $(C,p_1, \ldots, p_n)$ in this closure have been characterized in  \cite{Bainbridge2016Strata-of--k--d, BCCGM} in terms of existence of $\k$-differentials on the components of $C$ satisfying certain residue conditions. For $\k=1$ and all $m_i\geq 0$, the closure $\overline{\mathcal{H}}_g^1(\m)$ is of pure codimension $g-1$, and \cite{Pandharipande2019Tautological-re} gives a conjectural relation of the fundamental class of this closure to Witten’s $r$-spin classes. 

A larger compactification containing the closure $\overline{\mathcal{H}}_g^1(\m)$, the \emph{moduli space of twisted $\k$-differentials} $\widetilde {\ca H}_g^\k(\m)$, has been proposed by Farkas and Pandharipande in \cite{Farkas2016The-moduli-spac}. The idea here is that as the curve $C$ becomes reducible, it is no longer reasonable to ask for an isomorphism of line bundles $\omega_C^{\otimes \k} \cong \mathcal{O}_C(\sum_{i=1}^n m_i p_i)$, since these line bundles will have different degrees on the various components of $C$. However, these multidegrees can be balanced out by twisting the line bundles by (preimages of) the nodes of $C$.

The way this balancing happens is encoded in a \emph{twist} on the stable graph $\Gamma$ of $C$. This is a map $I$ from the set of half-edges of $\Gamma$ to the integers, satisfying $I(h)=-I(h')$ if $(h,h')$ form an edge, together with a further combinatorial condition (see \ref{def:twist} for details). Given a twist $I$ on the dual graph $\Gamma$ of a stable curve $C$, let $\nu_I: C_I \to C$ be the map normalizing the nodes $q \in C$ belonging to edges $(h,h')$ with $I(h) \neq 0$. Let $q_h, q_{h'} \in C_I$ be the corresponding preimages of $q$ under $\nu_I$. Then the curve $(C,p_1, \ldots, p_n)$ is contained in  $\widetilde {\ca H}_g^\k(\m)$ if and only if there exists a twist $I$ on its stable graph, such that we have an isomorphism of line bundles
\begin{equation} \label{eqn:twistdiffcond} \omega_{C_I}^{\otimes k} \cong \mathcal{O}_{C_I}\left(\sum_{i=1}^n m_i p_i + \sum_{\substack{(h,h')\in E(\Gamma)\\I(h)\neq 0}} (I(h)-\k) q_h + (I(h')-\k) q_{h'}\right)\end{equation}
on $C_I$. This corresponds to requiring the existence of a $\k$-differential on the components of the partial normalization $C_I$ of $C$ with zeros and poles at markings and preimages of nodes, where the multiplicities at the node preimages are dictated by the twist $I$.

The space $\widetilde {\ca H}_g^\k(\m)$ is a closed subset of $\Mbar_{g,n}$  containing $\overline{\ca H}_g^\k(\m)$ but possibly having additional components supported in the boundary of $\Mbar_{g,n}$. It turns out that these extra components are essential when trying to associate a natural cycle class to the extension of the strata of $\k$-differentials. Assume we are in the case that $\k \geq 1$ and that at least one of the $m_i$ is either negative or not divisible by $\k$. Then it is shown in \cite{Farkas2016The-moduli-spac} (for $\k=1$) and \cite{Schmitt2016Dimension-theor} (for $\k>1$) that $\widetilde {\ca H}_g^\k(\m) \subset \Mbar_{g,n}$ has pure codimension $g$. In this situation, instead of studying the fundamental class of $\widetilde {\ca H}_g^\k(\m)$ (as a reduced substack), the papers \cite{Farkas2016The-moduli-spac, Schmitt2016Dimension-theor} consider a certain weighted fundamental class $H_{g,\m}^\k \in A^g(\Mbar_{g,n})$ of $\widetilde {\ca H}_g^\k(\m)$. 

To define this weighted class, let $Z$ be an irreducible component of $\widetilde {\ca H}_g^\k(\m)$. Denote by $\Gamma$ the generic dual graph of a curve $C$ in $Z$ and let $I$ be the\footnote{It is not a priori clear that there could not be two different twists on $\Gamma$ which both ensure that \ref{eqn:twistdiffcond} holds for the generic point $C$ of $Z$, but we show in \ref{pro:uniquegentwist} that this cannot happen.} generic twist on $\Gamma$. 
Then it is shown in \cite{Farkas2016The-moduli-spac, Schmitt2016Dimension-theor} that $\Gamma$ and $I$ must be of a particular form. Indeed, the graph $\Gamma$ is a so-called \emph{simple star graph}, having a distinguished \emph{central vertex} $v_0$ such that every edge has exactly one endpoint at the central vertex. The remaining vertices are called the \emph{outlying vertices}. All markings $i$ with $m_i$ negative or not divisible by $\k$ must be on the central vertex. Moreover, the twist $I$ on $\Gamma$ has the property that for all edges $e=(h,h')$, with $h$ incident to $v_0$ and $h'$ incident to an outlying vertex, we have that $I(h')$ is positive and divisible by $\k$. By slight abuse of notation we write $I(e)=I(h')$ in this case. 

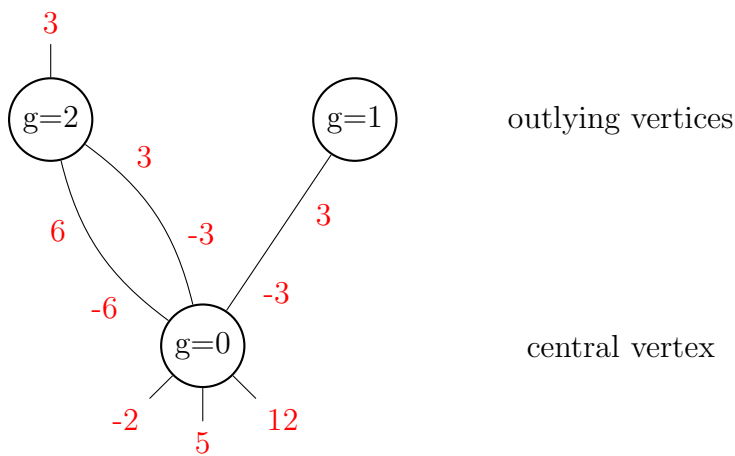
\begin{figure}[ht]
\begin{center}
 \begin{tikzpicture}
  
  \node[circle,draw=black,fill=white, inner sep=3pt, thick, minimum size=5pt] (g0) at (0,0) {g=0};
  \node[circle,draw=black,fill=white, inner sep=3pt, thick, minimum size=5pt] (g1) at (2,3) {g=1};
  \node[circle,draw=black,fill=white, inner sep=3pt, thick, minimum size=5pt] (g2) at (-2,3) {g=2};
%


  \draw (g0) --node[near start,below right, text=gray]{-3}  node[near end,below right, text=gray]{3} (g1);
  \draw (g0) to [out=103.7, in=-36.3,looseness=1.0] node[near start,above right, text=gray]{-3}  node[near end,above right, text=gray]{3} (g2);
  \draw (g0) to [out=143.7, in=-76.3,looseness=1.0] node[near start,below left, text=gray]{-6}  node[near end,below left, text=gray]{6} (g2);
  \draw (g2) -- (-2,4) node[above, text=gray]{3};
  \draw (g0) -- (-0.7,-0.7) node[below left, text=gray]{-2};
  \draw (g0) -- (0,-1) node[below, text=gray]{5};
  \draw (g0) -- (0.7,-0.7) node[below right, text=gray]{12};
  
  \node at (5.5,0){\text{central vertex}};
  \node at (5.5,3){\text{outlying vertices}};
 \end{tikzpicture}
\end{center}
\caption{Example of a simple star graph for $g=4$, $\k=3$ and $\m=(-2,5,3,12)$ with the twists $I$ of the half-edges and the weights $m_i$ of the marked points indicated in gray. \label{Fig:simplestar}}
\end{figure}

With this notation in place, we can define\footnote{The papers \cite{Farkas2016The-moduli-spac, Schmitt2016Dimension-theor} give a slightly different definition of the weighted fundamental class. We recall this definition in \ref{sec:FP_formula} and comment why it is equivalent to the formula above.} the weighted fundamental class $H_{g,\m}^\k$ of $\widetilde {\ca H}_g^\k(\m)$ as
\begin{equation}\label{eq:weighted_fun_formula}
 H_{g,\m}^\k = \sum_{(Z,\Gamma,I)} \frac{\prod_{e \in E(\Gamma)} I(e)}{\k^{\# V(\Gamma)-1}} [Z] \in A^g(\Mbar_{g,n}),
\end{equation}
where $Z$ runs over the components of $\widetilde {\ca H}_g^\k(\m)$ and as above $\Gamma, I$ are the generic dual graph and twist on $Z$. 


Conjecture A above then relates these weighted fundamental classes to the explicit tautological cycles $P_g^{g,\k}(\tilde \m)$ proposed by Aaron Pixton in \cite{Pixton}. In our paper, we show how both the twisted differential space $\widetilde {\ca H}_g^\k(\m)$ and its weighted fundamental class $H_{g,\m}^\k$ naturally arise from a construction presented by the first-named author in \cite{Holmes2017Extending-the-d}. 

%
 

\subsection{Extending the Abel-Jacobi map}\label{sec:extending_AJ}
Let $\ca J$ be the universal semi-abelian jacobian over $\Mbar_{g,n}$, often written $\on{Pic}^{\ul 0}_{\ca C/\Mbar_{g,n}}$. It has connected fibres, and parametrizes line bundles of multidegree zero on the fibres of the universal curve $\ca C \to \Mbar_{g,n}$. Inside the open set $\ca M_{g,n} \subset \Mbar_{g,n}$ the strata $\ca H_g^\k(\m)$ of $\k$-differentials can be obtained as the pullback of the zero-section $e$ of $\ca J$ via the Abel-Jacobi section 
\[\sigma : \ca M_{g,n} \to \ca J, (C,p_1, \ldots, p_n) \mapsto \omega^\k \left(-\sum_{i=1}^n m_i p_i \right)=:\omega^\k \left(-\m P \right).\] 
While $\sigma$ does not in general extend over $\Mbar_{g,n}$, in \cite{Holmes2017Extending-the-d} the first author defines a ``universal'' stack $\Md \to \Mbar_{g,n}$, birational over $\Mbar_{g,n}$, on which $\sigma$ does extend to a morphism $\sigma^\loz: \Md \to \ca J_{\Md}$, where $\ca J_{\Md}$ is the pullback  of $\ca J$ to $\Md$. 
Moreover, the scheme-theoretic pullback $\DRL^\loz$ of the unit section $e$ of $\ca J_{\Md}$ along $\sigma^\loz$ is proper over $\Mbar_{g,n}$. 
Denote $\DRC^\loz$ the cycle-theoretic pullback of the class $[e]$ under $\sigma^\loz$, supported on $\DRL^\loz$, and by $\overline{\DRC} \in A^g(\Mbar_{g,n})$ its pushforward under the proper map $\DRL^\loz \to \Mbar_{g,n}$.

\subsection{Main result} \label{Sec:mainresult}
Refining \ref{thm:main_intro_basic}, the main result of our paper is the following.
\begin{theorem} \label{thm:main_intro}
 The image of the double ramification locus $\DRL^\loz$ under the map $\Md \to \Mbar_{g,n}$ is the moduli space $\widetilde{\ca H}_g^\k(\m) \subset \Mbar_{g,n}$ of twisted $\k$-differentials. Moreover, for $\k \geq 1$ and at least one of the $m_i$ either negative or not divisible by $\k$, we have that 
\begin{equation}\label{eq:thm_main}
 \overline{\DRC} = H_{g,\m}^\k\in A^g(\Mbar_{g,n}). 
\end{equation}
In fact, this is true in the strong sense that these two cycles supported on $\widetilde{\ca H}_g^\k(\m)$ have the same weight at each irreducible component (they are equal as cycles, not just cycle classes). 
\end{theorem} 
For the last point of the theorem, the equality of $\overline{\DRC}$ and $H_{g,\m}^\k$ on the cycle level, observe that the formula \ref{eq:weighted_fun_formula} allows us to define $H^\k_{g, \m}$ as a cycle, not just a cycle class. And, under the assumptions of the theorem, the locus $\DRL^\loz$ has the expected codimension, and so $\overline{\DRC}$ makes sense as a cycle, not just a cycle class. Then in fact the equality \ref{eq:thm_main} holds as an equality of cycles, not only up to rational equivalence (in contrast to Conjecture A above, which only makes sense up to rational equivalence). 

We give an outline of the proof in \ref{Sect:introproofsketch}, where we will also discuss in more detail how the multiplicities in the formula \ref{eq:weighted_fun_formula} come up for the cycle $\overline{\DRC}$. Our method of proof actually yields more precise information than required for the conjecture; we can compute not only the multiplicities of the cycle, but even give a presentation for the Artin local rings at generic points of the double ramification locus (see \ref{thm:local_ring_description}). 

%
%
The above theorem gives a concrete interpretation for the weights appearing in the definition of $H_{g,\m}^\k$. It is also a crucial component of the proof of Conjecture A.

\begin{corollary} \label{cor:main_intro}
Conjecture A is true. 
\end{corollary}
\begin{proof}
The equality $\overline{\DRC} = 2^{-g}P_g^{g,k}(\tilde \m)$ is proven in \cite{bae2020pixtons}, so this follows from \ref{thm:main_intro}. 
\end{proof}

At the time this preprint was posted to the arXiv the equality $\overline{\DRC} = 2^{-g}P_g^{g,k}(\tilde \m)$ was known over the locus of compact-type curves by previous work \cite{Holmes2017Multiplicativit} with Pixton, showing Conjecture A to be true when restricted to the locus $\Mbar_{g,n}^{ct} \subset \Mbar_{g,n}$ of compact type curves.

\subsection{Sketch of the proof} \label{Sect:introproofsketch}
The main difficulty in the proof of \ref{thm:main_intro} is to compute the intersection multiplicity of the Abel-Jacobi map $\sigma^\loz: \Md \to \ca J_{\Md}$ with the unit section $e$ of $\ca J_{\Md}$ along the different components of $\DRL^\loz$. For this, we use classical deformation theory to first compute the Zariski tangent space at a general point and then show how to extend this study to higher order deformations. 

To set up the deformation theory, we first need to choose local coordinates on $\Md$. Here, it turns out that it is more convenient to work with a slight variant $\Mdk \to \Mbar_{g,n}$ of $\Md$, for which it is easier to write down local charts around the general points of $\DRL$. 
The precise construction of $\Mdk$ is given in \ref{sec:DR_construction} (where we also make more concrete the relationship with the construction of Marcus and Wise \cite{Marcus2017Logarithmic-com}), but for us the two key properties are
\begin{enumerate}
\item The map $\Mdk \to \Mbar_{g,n}$ is log \'etale and birational and
the map $\sigma\colon \ca M \to \ca J$ sending $(C, P)$ to $[\omega^\k(-\m P)]$ extends uniquely to a map $\bar\sigma\colon\Mdk \to \ca J$.
\item We can compute the tangent space to $\Mdk$ explicitly. 
\end{enumerate}
The double ramification locus is then $\DRL = \bar\sigma^* e$, where $e$ is the unit section in $\ca J$. The concrete local charts for $\Mdk$ can be used to show that the image of $\DRL$ in $\Mbar_{g,n}$ is exactly the twisted differential space $\widetilde{\ca H}_g^\k(\m)$.

With this setup established, the equality of weights in \ref{thm:main_intro} comes about in an interesting way: let $Z$ be an irreducible component of $\widetilde{\ca H}_g^\k(\m)$ with generic stable graph $\Gamma$ and twist $I$. Then a general point $p \in Z$ has exactly
  \[\# \{p' \in \DRL \text{ over } p\} = \k^{\#E(\Gamma) - \#V(\Gamma)+1} = \k^{b_1(\Gamma)}\]
  preimages $p'$ in $\DRL$. This is something that can be easily checked in the local charts of $\Mdk$. 
 In \ref{def:DRL_DR_1_over_k}, we define a cycle $\DRC$ supported on $\DRL$. At each preimage, its multiplicity is
  \[\on{mult}_{p'} \DRC = \prod_{e \in E(\Gamma)} \frac{I(e)}{\k} = \frac{\prod_{e \in E(\Gamma)} I(e)}{\k^{\#E(\Gamma)}}.\]
 Hence, the pushforward $\overline{\DRC}$ of $\DRC$ has multiplicity
 \[\on{mult}_p \overline{\DRC} = \k^{\#E(\Gamma) - \#V(\Gamma)+1} \frac{\prod_{e \in E(\Gamma)} I(e)}{\k^{\#E(\Gamma)}} = \frac{\prod_{e \in E(\Gamma)} I(e)}{\k^{\# V(\Gamma)-1}},\]
 which is exactly the weight of $[Z]$ in the class $H_{g,\m}^k$. It is also easy to see that the cycle $\DRC$ on $\DRL$ equals the fundamental class of (the possibly nonreduced) $\DRL$ (see \ref{lem:length_equals_multiplicity}), so we are left with studying the multiplicity of $\DRL$ at its generic points.


\Cref{sec:TpMd} is concerned with the computation of the tangent space to $\DRL$.
 Suppose we are given a point $p \in \DRL \sub \Mdk$, which is a general point of some irreducible component of $\DRL$. Let $\Gamma$ be the generic stable graph and $I$ be the generic twist on this component.

The maps $e$ and $\bar\sigma$ induce maps on tangent spaces
\begin{equation*}
\begin{tikzcd}
 T_p\Mdk \arrow[r,yshift=0.15cm,"T_p\bar\sigma"] \arrow[r,yshift=-0.15cm,swap,"T_pe"] & T_{e(p)} \ca J
\end{tikzcd}
\end{equation*}
and the difference $b=T_p\bar\sigma- T_pe$ factors via the tangent space $T_e \ca J_p$ to the fibre $\ca J_p$ of $\ca J$ over $p$. This induces an exact sequence
\begin{equation*}
0 \to T_p\DRL \to  T_p\Mdk \stackrel{b}{\to} T_e\ca J_p; 
\end{equation*}
it thus remains to analyse carefully the map $b$. For $\ca C_p$ the stable curve corresponding to the point $p$, the domain and target of $b$ are easily identified in terms of cohomology groups of sheaves on $\ca C_p$. 
Instead of studying the cokernel of $b$, it will be more convenient to use Serre duality and compute the kernel of the linear dual $b^\vee$, which is dual to $\coker(b)$. In \ref{thm:ker_b_v} we show that $\ker(b^\vee)$ has a natural basis, with one element for each outlying vertex $v$ of $\Gamma$ connected to the central vertex only by edges with twists $I>\k$. In \ref{thm:dim_T_p_DR} we conclude that  
\[\dim T_p \DRL = \dim_p \DRL + \# \{e \in E(\Gamma) : I(e)>\k\},\]
so we have one ``direction of nonreducedness'' for each edge $e$ with $I(e)>\k$, corresponding to an infinitesimal deformation smoothing the corresponding node. 

While this description is quite simple, the deformation-theoretic computation that derives it is fairly long and involved. We decompose the tangent space $T_p\Mdk$ into a direct sum of four pieces, corresponding to different types of deformations. Then the dual $b^\vee$ decomposes in four summands accordingly and we compute the intersection of their kernels. In the course of these computations,  we need to show that for the $\k$-differential on the central component of $\ca C_p$, we have that sums of $\k$th roots of its $\k$-residues\footnote{A generalization of the residue of a $1$-differential, see \ref{sec:generic_non-vanishing_residue} for a definition.} at (subsets of the) nodes of $\ca C_p$ are generically non-vanishing. We show a corresponding general result, which might be of independent interest, in \ref{sec:generic_non-vanishing_residue}.

That the tangent space to the double ramification locus can be computed via first-order deformation theory is unsurprising, but in order to prove \ref{thm:main_intro}, we need to compute the local rings of the double ramification locus, which is much more involved. It is not hard to show that an Artin local ring is determined by its functor of deformations, but reconstructing the Artin ring from the deformations is in practise often difficult. 

Write $E$ for the set of edges of the dual graph of the tautological stable curve $\ca C_p$ over $p$. The universal deformation of $\ca C_p$ comes with a natural projection map to $\on{Spec} \field[[x_e:e \in E]]$, which we can see as the space of deformations which smooth the nodes. Here $\field$ is our base field, which we assume to be of characteristic zero. We slice $\DRL$ with a generic subvariety 
 of codimension equal to the dimension of $\DRL$, obtaining a space $\DRL'$ whose tangent space has dimension equal to the number of edges $e$ with twist $I(e)>\k$. We use our tangent space computation to show that the natural map $\DRL'\to \on{Spec} \field[[x_e:e \in E]]$ is a closed immersion; it remains to identify the image. From the explanation above, one can reasonably guess that the image might be cut out by the ideal $$(x_e^{I(e)/\k} : e \in E) \sub \field[[x_e:e \in E]].$$ 

We conclude the proof by showing that for an Artin ring $A'$, a map $\on{Spec} A' \to \on{Spec} \field[[x_e:e \in E]]$ lifts along $\DRL'\to \on{Spec} \field[[x_e:e \in E]]$ if and only if the elements $x_e^{I(e)/\k}$ are sent to zero under the corresponding ring map $\field[[x_e:e \in E]] \to A'$. The proof works by writing $A'$ as an iterated extension of Artin rings and lifting the map one step at a time. That is, we have Artin rings $A_0=\field, A_1, \ldots, A_M=A'$ and short exact sequences
\[0 \to J_i \to A_{i} \to A_{i-1} \to 0\]
of $\field$-vector spaces, such that $A_i \to A_{i-1}$ is a morphism of $\field$-algebras with kernel $J_i \subset A_i$ satisfying $J_i \frak m_{A_i}=0$ for the maximal ideal $\frak m_{A_i}$ of $A_i$. Then we show that for each $i$, the obstruction of lifting an $A_i$-point of $\DRL'$ to an $A_{i+1}$-point of $\DRL'$ over  $\on{Spec} \field[[x_e:e \in E]]$ is exactly that all elements $x_e^{I(e)/\k}$ are sent to zero in $A_i$.

\begin{remark}
If we had worked $\Md$ instead of $\Mdk$, a similar description would be possible, but both the multiplicities and the cardinalities of fibres of the double ramification locus over the twisted differential space would have to be expressed in terms of the gcd/lcm of the twists (though in the end everything would of course cancel to give the same answer). This would have made the deformation-theoretic calculation more complicated, and seemed to us better avoided. 
\end{remark}
%
%

Once again, the key input is our result in \ref{sec:generic_non-vanishing_residue} on the generic-non-vanishing of $\k$-residues. 

\subsection{Relation to previous work and outlook} \label{Sect:relandoutlook}
\subsubsection*{Compactification via log geometry}
In the paper \cite{Guere2016A-generalizatio}, Gu\'er\'e uses logarithmic geometry to construct a moduli space of $\k$-log canonical divisors sitting over $\widetilde{\ca H}_g^\k(\m)$ and carrying a natural perfect obstruction theory and virtual fundamental class. For  $\k=1$ and one of the $m_i$ negative, the pushforward of this virtual class equals the weighted fundamental class $H_{g,\m}^k$. However, for general $\k$ the multiplicity of this pushforward at a component with stable graph $\Gamma$ and twist $I$ is equal to $\prod_{e \in E(\Gamma)} I(e)$, and thus different from the multiplicities obtained here and conjectured in \cite{Schmitt2016Dimension-theor}. This could indicate that for $\k>1$ the definition of the space in \cite{Guere2016A-generalizatio} needs to be adapted. We hope that the computations in the present paper may shed some light on the necessary modifications.

\subsubsection*{The cases of excess dimension}
Until now, our paper has focused on the case $\k \geq 1$ and one of the $m_i$ negative or not divisible by $\k$, in which case $\widetilde{\ca H}_g^\k(\m)$ was of pure codimension $g$. In general, by \cite[Theorem 21]{Farkas2016The-moduli-spac} all components of the space $\widetilde{\ca H}_g^\k(\m)$ have \emph{at most} codimension $g$. In these remaining cases, the behaviour is as follows:
\begin{itemize}
 \item for $\k=0$, the principal component $\overline{\ca H}_g^\k(\m)$ is of codimension exactly $g$ (unless all $m_i = 0$), but there are components in the boundary of $\Mbar_{g,n}$ of various excess dimensions;
 \item for $\k=1$ and all $m_i \geq 0$, the principal component $\overline{\ca H}_g^1(\m)$ is of pure codimension $g-1$, with all other components supported in the boundary and of codimension $g$,
 \item for $\k>1$ and all $m_i = \k m_i' \geq 0$ divisible by $\k$, the space $\mathcal{H}_{g}^k(\m)$ decomposes as a disjoint union
 \[\mathcal{H}_{g}^k(\m) = {\ca H}_g^1(\m') \cup \mathcal{H}_{g}^k(\m)';\]
 where ${\ca H}_g^1(\m')$ is the locus where the $\k$-differential is a $\k$-th power of a $1$-differential, and $\mathcal{H}_{g}^k(\m)'$ is the complement. Then $\overline{\ca H}_g^1(\m') \subset \widetilde{\ca H}_g^\k(\m)$ is a union of components of codimension $g-1$, with all other components (i.e. $\overline{\mathcal{H}}_{g}^k(\m)'$ and those supported in the boundary) having codimension $g$.
\end{itemize}
In all of these cases, the cycle $\overline{\DRC}$ still makes sense and by \ref{thm:main_intro} it is indeed supported on the locus $\widetilde{\ca H}_g^\k(\m) \subset \Mbar_{g,n}$. Similarly, the formula of Pixton's cycle $P_g^{g,\k}(\tilde \m)$ makes sense in these cases, and in \cite{Holmes2017Extending-the-d}, the first author shows that for $\k=0$ we have $\overline{\DRC}=2^{-g} P_g^{g,k}(\tilde \m)$. 

We expect that in the cases $\k \geq 1$ and $\m=\k \m' \geq 0$ the cycle $\overline{\DRC}$ should behave as follows:
\begin{itemize}
 \item on a component $Z$ of $\widetilde{\ca H}_g^\k(\m)$ of codimension equal to $g$, it should be \[\frac{\prod_{e \in E(\Gamma))} I(e)}{\k^{\#V(\Gamma)-1}} [Z]\] as before (where $\Gamma,I$ are the generic twist and dual graph),
 \item on the components $\overline{\ca H}_g^1(\m')$ of codimension $g-1$ it should be given by the first Chern class of an appropriate excess bundle (for the Abel-Jacobi section meeting the unit section) times the fundamental class of $\overline{\ca H}_g^1(\m')$.
\end{itemize}
It seems likely, that the deformation-theoretic tools in the present paper can be applied to prove these expectations, and explicitly identify the excess bundle.

The perspective above could also help shed further light on a second conjecture made in \cite{Schmitt2016Dimension-theor}. There, for a nonnegative partition $\m'$ of $2g-2$, a class $[\overline{\ca H}_g^1(\m')]^{\text{vir}}$ was defined by the formula
\[[\overline{\ca H}_g^1(\m')]^{\text{vir}} + \sum_{(Z,\Gamma,I)} \left({\prod_{e \in E(\Gamma))} I(e)} \right) [Z]  = 2^{-g} P_g^{g,1}(\tilde \m'), \]
where $Z$ runs through the boundary components of $\widetilde{\ca H}_g^1(\m')$ and $\tilde \m'=(m_1'+1, \ldots, m_n'+1)$. 
The idea was that $[\overline{\ca H}_g^1(\m')]^{\text{vir}}$ should be a contribution to the Double ramification cycle of the partition $\m'$, supported on $\overline{\ca H}_g^1(\m')$. From our perspective, this should just be the contribution of $\overline{\DRC}$ supported there. Then, since the locus $\overline{\ca H}_g^1(\m')$ appears as a component of $\overline{\ca H}_g^\k(\k \m')$ for any $\k>1$, the following conjecture was made.

\vspace{0.3 cm}
\noindent \textbf{Conjecture A' (\cite{Schmitt2016Dimension-theor})} Let $\k \geq 1$ and $\m=\k \m'$ for a nonnegative partition $\m'$ of $2g-2$. Then we have 
\[[\overline{\ca H}_g^1(\m')]^{\text{vir}} + [\overline{\mathcal{H}}_{g}^k(\m)'] + \sum_{(Z,\Gamma,I)} \frac{\prod_{e \in E(\Gamma))} I(e)}{\k^{\#V(\Gamma)-1}} [Z] = 2^{-g} P_g^{g,\k}(\tilde \m),\]
where $Z$ runs through the boundary components of $\widetilde{\ca H}_g^\k(\m)$.
\vspace{0.3 cm}

From the perspective of defining the Double ramification cycle via an extension of the Abel-Jacobi map, this behaviour is expected: the space $\Md$ for the partition $\k \m'$ of $\k(2g-2)$ agrees with the space for the partition $\m'$ of $2g-2$, and the Abel-Jacobi section for $\k \m'$ is simply the composition of the section for $\m'$ with the \'etale morphism
\[\ca J \to \ca J, (C,\mathcal{L}) \mapsto (C,\mathcal{L}^{\otimes \k}).\]
Thus, over the locus $\overline{\ca H}_g^1(\m')$, the intersection of the Abel-Jacobi section with the unit section should produce the same contribution to the cycle $\overline{\DRC}$.

\subsubsection*{Smoothing differentials}
The papers \cite{BCCGM, Bainbridge2016Strata-of--k--d} give criteria for a nodal curve $(C,p_1, \ldots, p_n)$ to lie in the locus $\overline{\mathcal{H}}_g^\k(\m)$. Being contained in this closure is equivalent to having some one-parameter deformation $(C_t, p_{1,t}, \ldots, p_{n,t})_{t \in \Delta}$ with the general curve being contained in ${\mathcal{H}}_g^\k(\m)$. The criteria of \cite{BCCGM, Bainbridge2016Strata-of--k--d} are phrased in terms of the existence of $\k$-differentials on the components of $C$ satisfying some vanishing conditions for sums of $\k$-th roots of their $\k$-residues at nodes of $C$. On the other hand, in our deformation-theoretic computations in \ref{sec:computing_length} we see that for a point in a boundary component of the Double ramification locus, the obstruction to smoothing the nodes while remaining in the Double ramification locus is exactly related to a nonvanishing of such sums of $\k$-th roots of $\k$-residues. While these computations are not directly applicable to the problem of classifying $\overline{\mathcal{H}}_g^\k(\m)$, it seems plausible that the methods of our paper can be applied in this direction. We thank Adrien Sauvaget for pointing out this connection and plan to pursue this in forthcoming work.

In a related direction, the recent paper \cite{BCGGM3} constructs a smooth compactification of the closure $\overline{\mathcal{H}}_g^\k(\m)$ and gives a modular interpretation for this new compactification. Here, it is an interesting question how this relates to the compactification 
obtained by taking the closure of ${\mathcal{H}}_g^\k(\m) \subset \DRL^\loz$ inside the Double ramification locus of $\Md$.
%
%

\subsection{An overview of different definitions of Double ramification cycles}\label{sec:overview_of_DRs}
In this section, we want to summarize the existing definitions of Double ramification cycles in the literature and the known equivalences between them.

Several authors gave elementary geometric constructions of the DR class on partial compactifications of $\ca M_{g,n}$ inside $\Mbar_{g,n}$ (for example, the compact-type locus), and computed them it in the tautological ring. Examples include \cite{Hain2013Normal-function}, \cite{Grushevsky2012The-zero-sectio}, \cite{Grushevsky2012The-double-rami}, and \cite{Dudin2015Compactified-un}.

The following are the different constructions of a DR cycle on all of $\Mbar_{g,n}$:
\begin{itemize}
 \item In the case $\k=0$, Li, Graber and Vakil gave a construction as the pushforward of a virtual fundamental class on  spaces of \textbf{rubber maps} (\cite{Li2002A-degeneration-}, \cite{Li2001Stable-morphism},\cite{Graber2005Relative-virtua}, see also \cite{liruan}).
 \item Pixton (\cite{Pixton}) proposed the formula $2^{-g} P_g^{g,\k}(\tilde \m)$ for the DR class as an explicit tautological class, defined via a \textbf{graph sum}.
 \item Kass and Pagani proposed extending the cycle as the pullback of a universal Brill-Noether class on a \textbf{compactified Jacobian} via the Abel-Jacobi section (\cite{Kass2017The-stability-s} and \cite[\S 2.4]{Holmes2017Jacobian-extens}).
 \item Marcus and Wise used techniques from \textbf{logarithmic and tropical geometry} (\cite{Marcus2017Logarithmic-com}) to construct a space on which the Abel-Jacobi map extends. 
 \item The first-named author gave a definition using a universal \textbf{extension of the Abel-Jacobi map} as described above (\cite{Holmes2017Extending-the-d}).
   \item Abreu and Pacini gave an explicit ``\textbf{tropical blowup}'' of $\Mbar_{g,n}$ (i.e. a blowup dictated by an explicit refinement of $\Mbar_{g,n}^\textup{trop}$) resolving the Abel-Jacobi map to the Esteves' compactified Jacobian over $\Mbar_{g,n}$ and use this to define a Double ramification cycle (\cite{abreupacini}).
 \item Finally, for $\k \geq 1$ and one of the $m_i$ is negative or not divisible by $\k$, there is the definition of the DR cycle as the \textbf{weighted fundamental class} $H_{g,\m}^\k$, proposed by Janda, Pandharipande, Pixton, and Zvonkine for $\k=1$ (\cite{Farkas2016The-moduli-spac}) and the second-named author for $\k > 1$ (\cite{Schmitt2016Dimension-theor}).
\end{itemize}
In \ref{Fig:DRequivalences} we illustrate the known equivalences between these definitions. In particular, \cite{bae2020pixtons} (which came out after the first appearance of this paper) completes the proof that they are all in fact equivalent. 

\tikzset{
between/.style args={#1 and #2}{
at=($(#1)!0.5!(#2)$)
}
}

\begin{figure}[ht]
\begin{center}
 \begin{tikzpicture}
  \tikzstyle{every node}=[fill=white, inner sep=5pt, ultra thick]
  \tikzstyle{allk}=[]
  \node[draw=black, minimum width = 3cm] (AP) at (0,9) {\begin{tabular}{c} Tropical\\blowup  \end{tabular}};
  \node[draw=black, minimum width = 3cm] (MW) at (0,0) {\begin{tabular}{c} Log/Tropical\\geometry  \end{tabular}};
  \node[draw=black, minimum width = 3cm] (H) at (0,3) {\begin{tabular}{c} Extension\\of AJ map  \end{tabular}};
  \node[draw=black, minimum width = 3cm] (KP) at (0,6) {\begin{tabular}{c} Compactified\\Jacobian  \end{tabular}};
  \node[draw=black, minimum width = 3cm] (JPPZS) at (5,3) {\begin{tabular}{c} Weighted\\fund. class  \end{tabular}};
  \node[draw=black, minimum width = 3cm] (P) at (10,3) {\begin{tabular}{c} Graph\\sum  \end{tabular}};
  \node[draw=black, minimum width = 3cm] (LGV) at (10,0) {\begin{tabular}{c} Rubber \\maps  \end{tabular}};
  
  \draw (JPPZS) node[below, inner sep = 0cm,  outer sep = 0.9cm]{\tiny{for  $\k \geq 1$ and some $m_i<0$ or $\k \nmid m_i$}};
  \draw (LGV) node[below, inner sep = 0cm, outer sep = 0.9cm]{\tiny{for  $\k =0$}};

  \draw (H) to [out=45, in=-45,looseness=1.45] node[midway,right]{\cite{abreupacini}} (AP);
  \draw (MW) --node[midway,left]{\cite{Holmes2017Extending-the-d}} (H);
  \draw (KP) --node[midway,left]{\cite{Holmes2017Jacobian-extens}} (H);
  \draw (MW) --node[midway,above]{\cite{Marcus2017Logarithmic-com}} (LGV);
  \draw (P) --node[midway,left]{\cite{Janda2016Double-ramifica}} (LGV);
  \draw[dotted] (P) --node[midway,above]{\tiny{Conj. A}} (JPPZS);
  \draw (H) --node[midway,above]{[HS19]$^*$} (JPPZS);
  \draw (H)  to [out=30, in=150] node[midway,above]{\cite{bae2020pixtons}} (P);
 \end{tikzpicture}
\end{center}
\caption{Equivalences between different definitions of Double ramification cycles; 
[HS19]$^*$ denotes the present article. \label{Fig:DRequivalences}}
\end{figure}
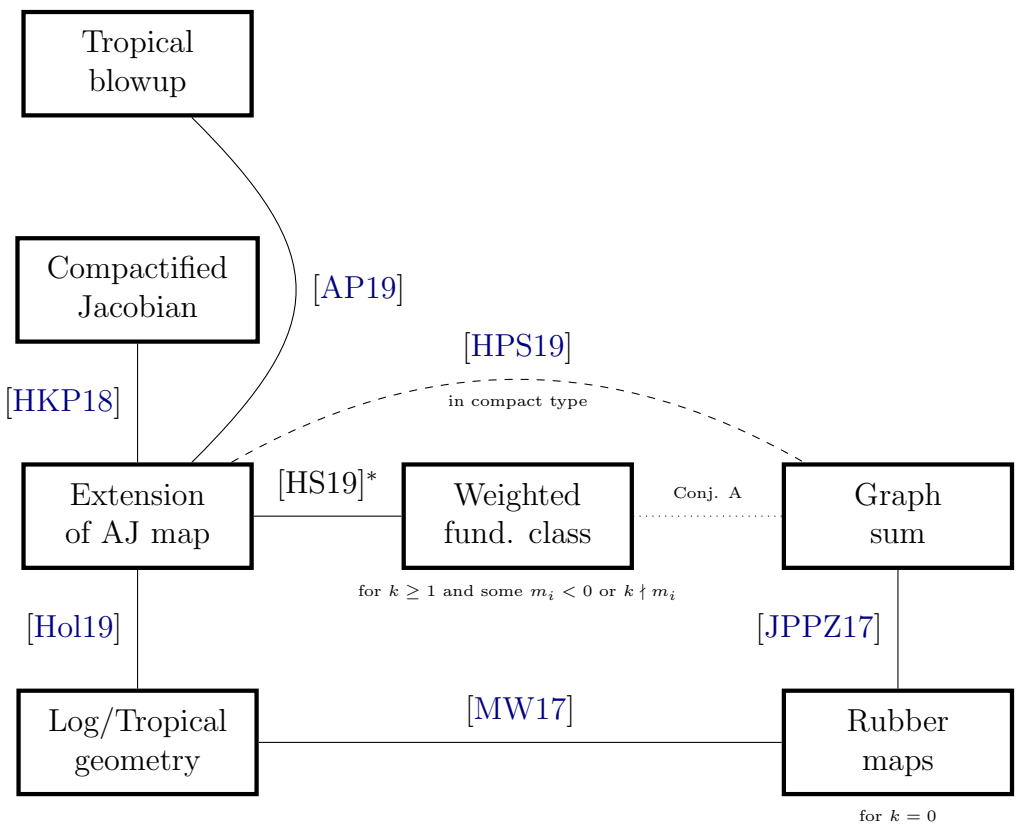

\subsection{Outline of the paper}

The main purpose of this paper is to analyse very carefully the infinitesimal structure of the double ramification locus, eventually enabling us to compute the multiplicities of its components and thus compare it to the cycle of twisted differentials. In \ref{sec:DR_construction} we describe the construction of the space $\Mdk$, the variant of $\Md$ on which we perform our computations (see \ref{Sect:introproofsketch} above). We also make more concrete the relationship with the construction of Marcus and Wise \cite{Marcus2017Logarithmic-com}.

%

\Cref{sec:TpMd,sec:coker_of_AJ_map} are devoted to the computation to the tangent space to the double ramification locus. In the brief \ref{sec:TpMd} we compute the tangent space of the space $\Mdk$, in which the double ramification locus naturally lives. \Cref{sec:coker_of_AJ_map} is much more substantial, and contains the computation of the tangent space of the double ramification locus itself. A key technical lemma on the non-vanishing of certain residues is postponed until \ref{sec:generic_non-vanishing_residue}, as it may be of independent interest and we wished to keep its exposition self-contained. 

Once we understand the tangent space to the double ramification locus, in \ref{sec:computing_length} we can compute explicitly its local ring, and in particular the length of the local ring. In \ref{sec:concluding_the_proof} we use this to deduce the desired formula of the Double ramification cycle as a weighted fundamental class. 

Finally, in the \ref{sec:serre_duality,sec:explicit_def_via_cech} we recall some standard results on Serre duality and deformation theory via \v Cech cocycles that are used in several places in the proof. This material is well-known, but we include it to fix notation, and because the very explicit forms of these results that we need are somewhat scattered about in the literature. 

\subsection*{Acknowledgements}
The first author would like to thank Martin Bright, Bas Edixhoven, and Robin de Jong for helpful conversations during the preparation of this article. He is very grateful to Bart de Smit for a discussion about the structure of Gorenstein Artin local rings, without which he would have wasted a lot of time on a dead-end. The second author would like to thank Felix Janda and J\'er\'emy Gu\'er\'e for valuable conversations. Both authors want to thank Adrien Sauvaget for interesting discussions and helpful comments.

We are very grateful to Gabriele Mondello for sharing with us the unpublished note \cite{MondelloPluricanonical-}, which gave an alternative proof of some of the results in \cite{Schmitt2016Dimension-theor} and contained a number of very helpful ideas. 

We also want to thank Alex Abreu, Bas Edixhoven, Quentin Gendron, Martin M\"oller, Marco Pacini, Rahul Pandharipande, Nicola Pagani and Adrien Sauvaget for useful comments on a preliminary version of the paper.


\subsection{Notation and conventions}
\subsubsection*{List of notations}
\begin{longtable}{ p{.30\textwidth}  p{.70\textwidth} } 
$\field$ & the ground field, assumed to be of characteristic zero\\
$\ca M$ & the moduli space $\ca M_{g,n}$ of smooth curves of genus $g$ with $n$ marked points\\
$\Mbar$ & the moduli space $\Mbar_{g,n}$ of stable curves of genus $g$ with $n$ marked points\\
$\ca J \to \Mbar$ & the universal semi-abelian jacobian over $\Mbar_{g,n}$\\
$\omega^\k \left(-\m P \right)$ & the twisted $\k$-canonical line bundle $\omega_C^\k \left(-\sum_{i=1}^n m_i p_i \right)$ of a curve $C$\\
$I$ & a twist, given by a function $I$ on the half-edges of a stable graph $\Gamma$ as described in \ref{def:twist}\\
$V^{out}$ & the outlying vertices (i.e. those that are not the central vertex) in a simple star graph, as described in \ref{def:simple_star}\\
$\Mbar \stackrel{f}{\longleftarrow} U \stackrel{g}{\longrightarrow} \bb A^E$ & a combinatorial chart of $\Mbar$, as described in \ref{sec:combinatorial_charts}\\
$\Md$ & the universal stack $\Md/\Mbar$ on which the Abel-Jacobi map $\sigma : \ca M \to \ca J$ extends, constructed in \cite{Holmes2017Extending-the-d}\\
$\sigma^\loz  \colon \Md \to \ca J$ & the extension of $\sigma : \ca M \to \ca J$ to $\Md$\\
$\DRL^\loz$ & the scheme theoretic pullback $\DRL^\loz \subset \Md$ of the unit section $e \subset \ca J$ under $\sigma^\loz$\\
$\DRC^\loz$ & the double ramification cycle on $\DRL^\loz$\\
$\overline{\DRC}$ & the cycle in $\Mbar$ obtained by the pushforward of $\DRC^\loz$ via the proper morphism $\DRL^\loz \to \Mbar$\\
$\bb A^E_I$ & an affine toric variety $\bb A^E_I \to \bb A^E$ associated to a fixed combinatorial chart $\Mbar \stackrel{f}{\longleftarrow} U \stackrel{g}{\longrightarrow} \bb A^E$ and twist $I$ on $\Gamma$, with equations as given in \ref{sect:affinepatches}\\
$\Mdm_{I,U}$ & the pullback of $\bb A^E_I \to \bb A^E$ under the map $g:U \to \bb A^E$ from the combinatorial chart\\
$\Mdm$ & the stack $\Mdm \to \Mbar$ obtained by gluing the patches $\Mdm_{I,U}$ for a cover of $\Mbar$ by combinatorial charts\\
$\sigma^\m \colon \Mdm \to \ca J$ & the extension of $\sigma : \ca M \to \ca J$ to $\Mdm$\\
$\DRL^\m, \DRC^\m$ & the double ramification locus and cycle constructed inside $\Mdm$\\
$\Mdm_{\k | I}$ & the open substack in $\Mdm$ where the twist $I$ is divisible by $\k$ \\ 
$\Mbar^{\m, 1/\k}$ & the partial normalization of $\Mdm_{\k | I}$ obtained by taking $\k$th roots in the defining equations of $\bb A^E_I$, pulling back to the combinatorial charts $U$ and gluing\\
$I'$ & the divided twist $I'(e)=I(e)/\k$ for the case that $I$ is divisible by $\k$\\
$\barsigma  \colon \Mbar^{\m, 1/\k} \to \ca J$ & the extension of $\sigma : \ca M \to \ca J$ to $\Mbar^{\m, 1/\k}$ \\
$\Mbar^{\m, \k, ''}$ & the normalization of $\Mdm_{\k | I}$, an open substack of the normalization $\Md$ of $\Mdm$\\
$\DRL^{1/\k}, \DRC^{1/\k}$ & the double ramification locus and cycle constructed inside $\Mbar^{\m, 1/\k}$\\
$U_u$ & the spectrum $\on{Spec} \ca O_{U,u}$ of the local ring of a point $u \in U$ inside a combinatorial chart $\Mbar \stackrel{f}{\longleftarrow} U \stackrel{g}{\longrightarrow} \bb A^E$\\
$z_h, z_{h'}$ & coordinates in a singular coordinate chart of the curve $\ca C_{U_u}/U_u$, for an edge $e={h,h'}$ of the stable graph $\Gamma$ of $\ca C_{U,u}$, where $z_h$ vanishes on the component of $\ca C_{U,u}$ incident to $h$\\
$\ca T$ & the `correction' line bundle on the universal curve $\ca C_I / \Mbar_{I,U}^{\m,1/\k}$ described in \ref{sec:universal_section} 
\\
$H^1(\Gamma,\field)$ & the first cohomology group of the graph $\Gamma$ with coefficients in $\field$, isomorphic to $\field^{b_1(\Gamma)}$\\
$\gamma(e,e')$ & cycle of length $2$ given by the composition of the directed edge $e$ with the inverse of the directed edge $e'$, assuming $e,e':v \to v'$ have same source and same target\\
$L_v$ & sub-vector space of $T_u \Mbar$ of dimension at most $1$ defined in \ref{sec:TpMd}\\
$\ca J_p$ & the fibre of the universal jacobian $\ca J \to \Mbar^{\m, 1/\k}$ over the point $p: \on{Spec} \field \to \Mbar^{\m, 1/\k}$\\
$b: T_p \Mdk \to T_e \ca J_p$ & the difference $T_p \barsigma - T_p e$ of the differential of the Abel-Jacobi map $\barsigma$ and the unit section $e$ of $\ca J$\\
$b_\Omega, b_\Gamma, b_{>1}, b_{L_v}$ & the restrictions of $b$ to the direct summands $H^1(\ca C_p, \Omega^\vee(-P)), H^1(\Gamma, \field), \bigoplus_{e:I'(e)>1} \field, L_v$ of $T_p \Mdk$ in \ref{sec:AJ_on_T_p}\\
$\Omega$ &  the sheaf of relative differentials\\
$\omega$ & the relative dualising sheaf
\end{longtable}

\subsubsection*{Generalities}
We have fixed integers $g \ge 0$, $n >0$, $\k >0$ with $2g-2 + n>0$, and integers $m_1, \dots, m_n$ summing to $\k(2g-2)$ with at least one $m_i <0$ or not divisible by $\k$. We will write $\ca M$ for $\ca M_{g,n}$, $\Mbar$ for $\Mbar_{g,n}$ etc. We write $\ca J$ for the universal semi-abelian jacobian over $\Mbar$, often written $\on{Pic}^{\ul 0}_{\ca C/\Mbar}$. Then the section 
\[\sigma : \ca M \to \ca J, (C,p_1, \ldots, p_n) \mapsto \omega^\k \left(-\sum_{i=1}^n m_i p_i \right)=:\omega^\k \left(-\m P \right)\] 
lives naturally in $\ca J(\ca M)$, but in general does \emph{not} extend to the whole of $\Mbar$. 

We work throughout over a fixed field $\field$, which we assume to have characteristic zero. Our proof is entirely algebraic, except for the crucial application of a result of Sauvaget \cite[Corollary 3.8]{Sauvaget2017Cohomology-clas} in \ref{sec:generic_non-vanishing_residue}, which we expect to admit an algebraic proof. When $\k >1$ we use very often the characteristic-zero assumption, but for $\k=1$ it can often be avoided; its main purpose is in allowing us to apply Sauvaget's result mentioned above, and in \ref{lem:kernel_description} where we use that a function with vanishing differential is locally constant. As such it may well be possible with the methods here to determine what happens in small characteristic; it seems very likely that the multiplicities of the twisted differential space will be different in this case. 

\begin{remark}
Our results do not require that the ground field $K$ be algebraically closed. When we talk about the graph of a curve over a field, we are implicitly saying that the irreducible components are geometrically irreducible, and the preimages of the nodes in the normalisation are all rational points. At later points we will assert that various $\k$-differentials locally have $\k$-th roots; this should be interpreted over a suitable finite extension, (our characteristic-zero assumption ensures that adjoining $\k$-th roots yields an etale extension, and thus does not affect the deformation theory. Alternatively, because the computations of the tangent spaces and lengths of local rings are invariant under etale extensions, the reader may assume without loss of generality that the ground field $\k$ is algebraically closed throughout \ref{sec:coker_of_AJ_map,sec:computing_length}. 

We expect that most readers will be mainly interested in the case of algebraically closed fields, so to minimise clutter we do not explicitly discuss these field extensions, but allow the interested reader to insert them when necessary. 
\end{remark}

\subsubsection{Graphs and twists}
A graph $\Gamma$ consists of a finite set $V$ of vertices, a finite set $H$ of half-edges, a map `$\on{end}$' from the half-edges to the vertices, an involution $i$ on the half-edges, and a \emph{genus} $g\colon V \to \bb Z_{\ge 0}$. Graphs are connected, and the genus $g(\Gamma$) is the first Betti number plus the sum of the genera of the vertices. 

\emph{Self-loops} are when two distinct half-edges have the same associated vertex and are swapped by $i$. \emph{Edges} are sets $\{h, h'\}$ (of cardinality 2) with $i(h) = h'$. \emph{Legs} are fixed points of $i$, and $L$ denotes the set of legs. A \emph{directed edge} $h$ is a half-edge that is not a leg; we call $\on{end}(h)$ its \emph{source} and $\on{end}(i(h))$ its \emph{target}, and sometimes write it as $h\colon \on{end}(h) \to \on{end}(i(h))$. We write $E = E(\Gamma)$ for the set of edges. 

The \emph{valence} $\on{val}(v)$ of a vertex is the number of non-leg half-edges incident to it, and we define the \emph{canonical degree} $\mathrm{can}(v) = 2g(v) - 2 + \on{val}(v)$, so that 
\begin{equation*}
2g(\Gamma) - 2 = \sum_v \mathrm{can}(v). 
\end{equation*}


 A \emph{closed walk} in $\Gamma$ is a sequence of directed edges so that the target of one is the source of the next, and which begins and ends at the same vertex. We call it a \emph{cycle} if it does not repeat any vertices or (undirected) edges. 

A \emph{leg-weighted graph} is a graph $\Gamma$ together with a function $\m$ from the set $L$ of legs to $\bb Z$ such that $\sum_{l \in L}\m(l) = \k(2g(\Gamma)-2)$. 
\begin{definition}\label{def:twist}
A \emph{twist} of a leg-weighted graph is a function $I$ from the half-edges to $\bb Z$ such that:
\begin{enumerate}
\item for all legs $l \in L$, we have $\m(l) = I(l)$. 
\item If $i(h) = h'$ and $h \neq h'$ then $I(h) +I(h') =0$;
\item For all vertices $v$, $\sum_{\on{end}(h) = v} I(h) -  \k \cdot \mathrm{can}(v)$ = 0. 
\end{enumerate}
\end{definition}
We write $\on{Tw}(\Gamma)$ for the (non-empty) set of twists of a leg-weighted graph $\Gamma$.


\begin{remark}
In \cite{Holmes2017Extending-the-d} these twists were called `weightings', and were denoted $w$. The present notation is much closer to that used by \cite{Farkas2016The-moduli-spac}; we have made this change to facilitate comparison to \cite{Farkas2016The-moduli-spac}, and because the letter $w$ was already over-loaded. 
\end{remark}

\begin{remark}
Farkas and Pandharipande impose two additional conditions (which they call `vanishing' and `sign'), which together state that $\Gamma$ cannot contain any directed cycle for which every directed edge $h$ has $I(h) \ge 0$, and at least one $h$ has $I(h)>0$. 

We do not need to impose this condition as it will drop out automatically from our geometric setup; more precisely, the fibre of a chart $\Mdm_{I,U}$ of $\Mdm$ over the origin in $\bb A^E$ (see \ref{sec:charts_of_Md} for this notation) is easily seen to be empty if either of these conditions is not satisfied. 

If one forgets the values of the integers $I(h)$ and remembers only their signs and whether they vanish, the above condition is exactly equivalent to `Suzumura consistency', a condition arising in decision theory \cite{Bossert2008}. 
\end{remark}

\begin{definition}\label{def:simple_star}
We say a leg-weighted graph $\Gamma$ is a \emph{simple star graph} if all legs with negative weight or weight not divisible by $\k$ are attached to the same vertex (which we call the \emph{central vertex}), and every edge has exactly one half-edge attached to the central vertex (in particular, there are no self-loops). We call the non-central vertices the \emph{outlying vertices}, and the set of them is $V^{out}$. 
\end{definition}

\subsubsection{The weighted fundamental class of the space of twisted differentials}\label{sec:FP_formula}


In this section we recall the definition of the class $H_{g,\m}^\k \in A^{g}(\Mbar)$ given in \cite[\S A.4]{Farkas2016The-moduli-spac} (for $\k=1$) and \cite[Section 3.1]{Schmitt2016Dimension-theor} (for $\k>1$) and explain why it is equivalent to the definition as a weighted fundamental class of $\widetilde{\mathcal{H}}_g^\k(\m)$ presented in the introduction. 

First, recall that given any integer $\k \geq 1$ and a partition $\m$ of $\k(2g-2)$ of length $n$, we have
\[\ca H_{g}^{\k}(\m) = \left\{(C,p_1, \ldots, p_{n}) : \omega_{C}^{\k}\left(-\sum_{i} m_i p_i\right) \cong \mathcal{O}_{C}\right\} \subset \ca M_{g,n},\]
the corresponding stratum of $\k$-differentials. This closed, reduced substack has pure codimension $g-1$ if $\k=1$ and all $m_i \geq 0$, and pure codimension $g$ if there exists $i$ such that $m_i$ is negative or not divisible by $\k$. As before we denote by $\overline{\ca H}_{g}^{\k}(\m)$ its closure in $\Mbar_{g,n}$.



Write $S$ for the set of simple star graphs of genus $g$ (see \ref{def:simple_star}). We say a twist $I$ of a simple star graph is \emph{positive} (writing $\on{Tw}^+(\Gamma)$ for the set of positive twists) if $I(h)>0$ and $\k$ divides $I(h)$ for every half-edge $h$ attached to an outlying vertex. In this case, by slight abuse of notation, we write $I(e)=I(h)$ for the edge $e=(h,h')$ to which $h$ belongs.


With this notation, Janda, Pandharipande, Pixton, Zvonkine (for $\k =1$) and the second author (for $k > 1$) define
\begin{align*}
H_{g,\m}^\k = \sum_{\Gamma \in S}\sum_{I \in \on{Tw}^+(\Gamma)}\frac{\prod_{e \in E(\Gamma)}I(e)}{|\on{Aut}(\Gamma)|\k^{\# V^{out}}}\xi_{\Gamma *}\big[ \left[\overline{\ca H}_{g(v_0)}^{\k}(\m|_{v_0}, -I|_{v_0} - \k)\right]\cdot \\
\prod_{v \in V^{out}(\Gamma)}\left[ \overline{\ca H}_{g(v)}^1(\frac{\m}{\k}|_v, \frac{I}{\k}|_v - 1)\right] \big]. 
\end{align*}
Here $\ca H_{g(v_0)}^{\k}(\m|_{v_0}, -I|_{v_0} - \k)$ denotes the cycle in $\ca M_{g(v_0), n(v_0)}$ with $n(v_0)$ the number of half-edges attached to $v_0$, and with weighting given by restricting the weighting $\m$ to those legs attached to $v_0$, and given by $-I(e)-\k$ at the half-edge belonging to the edge $e$ of $\Gamma$. The cycles $\ca H_{g(v)}^1(\frac{\m}{\k}|_v, \frac{I}{\k}|_v - 1)$ on the outlying vertices $v$ are defined analogously, where we use that all markings on them have weights $m_i$ divisible by $\k$ and all twists $I$ are likewise divisible by $\k$ (again, see \cite{Schmitt2016Dimension-theor} for details). 

Now we comment why this is a weighted fundamental class of the space $\widetilde{\mathcal{H}}_g^\k(\m)$. Given a boundary component $Z$ of this space, let $\Gamma$ be the generic dual graph of a curve $C$ in $Z$ and let $I$ be the twist on $\Gamma$ such that the condition \ref{eqn:twistdiffcond} is satisfied for this generic curve $C$. By \cite[Proposition A.1.] {Schmitt2016Dimension-theor} every node of $C$ such that the corresponding edge has twist $I=0$ can be smoothed while staying in  $\widetilde{\mathcal{H}}_g^\k(\m)$. Thus since $Z$ is assumed a generic point of $\widetilde{\mathcal{H}}_g^\k(\m)$, all edges of $\Gamma$ must have nonzero twist. Then this condition tells us that the various components $C_v$ of $C$ vary within appropriate strata of $\k$-differentials. But the codimension of $Z$ is at most $g$ by  \cite[Theorem 21]{Farkas2016The-moduli-spac}. A short computation shows that this is only possible if at all but one of the vertices $v$, the curve $C_v$ varies in a stratum of $\k$-th powers of holomorphic $1$-differentials (which is the case of excess-dimension). This implies that all twists must be divisible by $\k$ and that there is exactly one vertex carrying all the negatively twisted half-edges as well as markings $i$ with $m_i<0$ or not divisible by $\k$. This easily implies that the generic dual graph $\Gamma$ of $Z$ is a simple star graph and that the twist $I$ on $\Gamma$ is positive.

Conversely, one checks that condition \ref{eqn:twistdiffcond} is satisfied on all the loci on which the cycle $H_{g,\m}^\k$ above is supported. This shows that it is indeed a weighted fundamental class of $\widetilde{\mathcal{H}}_g^\k(\m)$. On the other hand, the weights agree with those given in the introduction: the closures of strata of differentials (which are pushed forward via $\xi_\Gamma$) are generically reduced and thus all have multiplicity $1$. The factor $1/|\on{Aut}(\Gamma)|$ exactly accounts for the fact that the gluing morphism $\Gamma$ has degree $|\on{Aut}(\Gamma)|$.

Thus the definition of $H_{g,\m}^\k$ given in the introduction coincides with the definitions from \cite{Farkas2016The-moduli-spac, Schmitt2016Dimension-theor}.



\subsubsection{Combinatorial charts}

%
%
%
%

If $p \colon \on{Spec} \field \to \Mbar$ is a geometric point corresponding to a curve $C$, the associated graph $\Gamma_C$ comes with a leg-weighing from the integers $m_i$. If a node of the curve over $p$ has local equation $xy - r$ for some $r \in \ca O^{et}_{\Mbar, p}$, then the image of $r$ in the monoid $\ca O^{et}_{\Mbar, p}/(\ca O^{et}_{\Mbar, p})^\times$ is independent of the choice of local equation. In this way, for each edge $e \in E(\Gamma_p)$ we obtain an element $\ell_e \in \ca O^{et}_{\Mbar, p}/(\ca O^{et}_{\Mbar, p})^\times$, recalling that edges of the graph correspond to nodes of the curve.


Given a leg-weighted graph $\Gamma$ with edge set $E$, define\footnote{In \cite{Holmes2017Extending-the-d} this was denoted $\Mbar_\Gamma$. } $$\bb A^E = \on{Spec}\field[a_e:e \in E].$$ To any point $a$ in $\bb A^E$ we associate the graph $\Gamma_a$ obtained from $\Gamma$ by contracting exactly those edges $e$ such that $a_e$ is a unit at $a$. Denote by $[a_e] \in \ca O^{et}_{\bb A^E, p}$ the image of the function $a_e$ on $\bb A^E$ in $\ca O^{et}_{\bb A^E, p}$.

\begin{definition}\label{sec:combinatorial_charts}
A \emph{combinatorial chart of $\Mbar$} consists of a leg-weighted graph $\Gamma$ and a diagram of stacks
\begin{equation*}
\Mbar \stackrel{f}{\longleftarrow} U \stackrel{g}{\longrightarrow} \bb A^E
\end{equation*}
satisfying the following six conditions:
\begin{enumerate}
\item $U$ is a connected scheme
\item $g\colon U \to \bb A^E$ is smooth
\item $f\colon U \to \Mbar$ is \'etale
\item the pullbacks of the boundary divisors in $\Mbar$ and $\bb A^E$ to $U$ coincide
\item $0 \in \bb A^E$ is in the image of $g$. 
\end{enumerate}
Let $p\colon \on{Spec}\field \to U$ be any geometric point, yielding natural maps 
\begin{equation*}
\ca O^{et}_{\Mbar, f \circ p} \stackrel{f^\flat}{\to} \ca O^{et}_{U, p} \stackrel{g^\flat}{\leftarrow} \ca O^{et}_{\bb A^E, g \circ p}. 
\end{equation*}
\begin{enumerate}[resume*]
\item 
Let $C=f(p)$ and $a=g(p)$, then we require an isomorphism
\begin{equation*}
\phi_p\colon \Gamma_C \to \Gamma_a
\end{equation*}
such that
$f^\flat(\ell_e) = g^\flat ([a_{\phi_p(e)}])$ up to units in $\ca O^{et}_{U, p}$ for every edge $e$ (which necessarily makes this $\phi_p$ unique if it exists). Moreover, the map $\phi_p$ sends the leg-weighting on $\Gamma_{f \circ p}$ coming from the $-m_i$ to the leg-weighting on $\Gamma_{g \circ p}$ coming from that on $\Gamma$. 
\end{enumerate}
\end{definition}



This definition is as in \cite{Holmes2017Extending-the-d} but with the logarithmic structures excised (since we do not need them). We see in \cite{Holmes2017Extending-the-d} that $\Mbar$ can be covered by combinatorial charts. 

\section{Constructing suitable moduli spaces}\label{sec:DR_construction}

\subsection{Recalling the construction of \texorpdfstring{$\DR$}{DR}}\label{sec:DR_construction_prev}

We begin by recalling the basic construction of the cycle $\DR$ from \cite{Holmes2017Extending-the-d}. First one constructs a certain stack $\Md/\Mbar$ such that the rational map $\sigma \colon \ca M \to \ca J$ extends to a morphism $\sigma^\loz  \colon \Md \to \ca J$. Writing $e$ for the unit section of $\ca J$ (viewed as a closed subscheme) and $[e]$ for its Chow class, it is shown in \cite{Holmes2017Extending-the-d} that the scheme-theoretic pullback $\DRL^\loz$ of $e$ along $\sigma^\loz$ is proper over $\Mbar$. We would like to now take the cycle-theoretic pullback of the class of $e$ along $\sigma^\loz$, but the latter is not (known to be) a regular closed immersion, so we do not know how to make sense of this pullback. Instead, we consider the induced section $ \Md \to \ca J_{\Md} =  \ca J \times_{\Mbar} \Md$, and pull back the class of the unit section along this section (using that the latter is a regular closed immersion as $\ca J$ is smooth over $\Mbar$) to obtain a cycle $\DRC^\loz$ on $\Md$. This cycle $\DRC^\loz$ is naturally supported on $\DRL^\loz$, and so by properness can be pushed down to a cycle on $\Mbar$, which we denote $\overline{\DRC}$, the compactified double ramification cycle. Many more details and properties of the construction, and a comparison to other constructions in the literature, can be found in \cite{Holmes2017Extending-the-d}, \cite{Holmes2017Multiplicativit} and \cite{Holmes2017Jacobian-extens}. 

In this article we will work with a slight variant of the stack $\Md$ of \cite{Holmes2017Extending-the-d}; this is only for convenience, but the intricacy of the calculations we have to carry out make every available bit of notational efficiency worth using. We also note that $\Md$ depends not only on $g$ and $n$, but also on the $m_i$ and $\k$, hence the notation is not good --- we will take the opportunity to correct this.

The stack $\Md$ is built by glueing together normal toric varieties, in particular it is normal. We will begin by introducing a `non-normal' analogue $\Mdm$ of $\Md$ which is close to (but not yet quite) what we want. The resulting double ramification cycle will be unchanged, by compatibility of the refined gysin pullback with the proper pushforward, see \ref{sec:comparing_DRs} for more details. 


%
%

\subsection{Construction of \texorpdfstring{$\Mdm$}{Mbar\textasciicircum m}}\label{sec:charts_of_Md}

Fix a combinatorial chart 
\begin{equation*}
\Mbar \stackrel{f}{\longleftarrow} U \stackrel{g}{\longrightarrow} \bb A^E
\end{equation*}
and a twist $I$ on $\Gamma$. If $e = \{ h, h'\}$ is an edge of $\Gamma$, and $\gamma$ is a cycle in $\Gamma$, we define 
\begin{equation}\label{eq:I_gamma_e}
I_\gamma(e) = \left\{ 
\begin{array}{ccl}
0 & \text{if } & h \notin \gamma \text{ and } h' \notin \gamma \\
I(h) & \text{if }& h \in \gamma\\
I(h') & \text{if }& h' \in \gamma. \\
\end{array}
 \right.
\end{equation}
In the free abelian group on symbols $a_e: e \in E$ we consider the submonoid generated by the $a_e$ and by the expressions
\begin{equation}\label{eq:generators_non_normal}
\prod_{e \in E} a_e^{I_\gamma(e)}
\end{equation}
as $\gamma$ runs over cycles in $\Gamma$, and we denote the spectrum of the associated monoid ring by $\bb A^E_I$. Equivalently, $\bb A^E_I$ is the spectrum of the subring of $\field[a_e^{\pm 1}:e \in E]$ generated by the $a_e$ and by the expressions in \ref{eq:generators_non_normal}. 
%
%
Note that this is slightly different from the monoid rings constructed in \cite{Holmes2017Extending-the-d}, where we worked with 
sub-polyhedral cones of $\mathbb{Q}_{\geq 0}^E$, cut out by equations: monoids coming from cones are always saturated, and so yield normal varieties, whereas here we want to work with not-necessarily-saturated monoids. In \ref{sect:affinepatches} we give explicit equations for (a slight variant on) the $\bb A^E_I$. 

%

We write $\Mdm_{I,U}$ for the pullback of $\bb A^E_I$ to $U$. We want to argue that these $\Mdm_{I,U}$ naturally glue together to form a stack $\Mdm$ over $\Mbar$. The first part of the gluing can even be done over $\bb A^E$. Indeed, fixing a  graph $\Gamma$, as $I$ runs over twists of $\Gamma$ the $\bb A^E_I$ naturally glue together as $I$, c.f. \cite[\S 3]{Holmes2017Extending-the-d}\footnote{Note that the hard thing in that reference is proving quasi-compactness of the resulting object, but since it is clear that the normalisation of the object constructed here is that built in \cite{Holmes2017Extending-the-d} the quasi-compactness comes for free here.}. We denote the glued object by $\tilde {\bb A}^E \to \bb A^E$. 

\begin{example}
In the case $k=0$, suppose the graph $\Gamma$ has two edges and two (non-loop) vertices $u$ and $v$. Suppose the leg weighting is $+n$ at $u$ and $-n$ at $v$. Twists consist of a flow of $a$ along edge $e$ from $u$ to $v$, and $n-a$ along the other edge $e'$ (again from $u$ to $v$), for $a \in \bb Z$:
\begin{equation}
 \begin{tikzcd}
\arrow[r, "+n"] &   u \arrow[rr, bend left, "I(e)= i"] \arrow[rr, bend right, swap, "I(e') = n-i"]&  & v \arrow[r, "+n"] & {}\\
\end{tikzcd}
\end{equation}

In this setting $\bb A^E = \on{Spec} K[a_e, a_{e'}]$. There are two directed cycles, and the expression \ref{eq:generators_non_normal} yields $a_e^i a_{e'}^{i-n}$ and $a_e^{-i} a_{e'}^{n-i}$. The form of $\bb A^E_I$ then depends on $I$: we have 
\begin{equation}
\begin{split}
i<0: &  \hspace{1cm} \bb A^E_I = \on{Spec} K[a_e^{\pm 1}, a_{e'}^{\pm 1}]\\
i=0: &  \hspace{1cm} \bb A^E_I = \on{Spec} K[a_e, a_{e'}^{\pm 1}]\\
0 < i < n: &  \hspace{1cm} \bb A^E_I = \on{Spec} \frac{K[a_e, a_{e'}, s^{\pm 1}]}{(a_{e'}^{n-i}s - a_e^i)} = \on{Spec} \frac{K[a_e, a_{e'}, t^{\pm 1}]}{(a_{e}^{i}t - a_{e'}^{n-i})}\\
i=n: &  \hspace{1cm} \bb A^E_I = \on{Spec} K[a_e^{\pm 1}, a_{e'}] \\
i>n: &  \hspace{1cm} \bb A^E_I = \on{Spec} K[a_e^{\pm 1}, a_{e'}^{\pm 1}].\\
\end{split}
\end{equation}
A more detailed explanation of these equations can be found in \ref{eqn:toriccoordinates} below. These patches are then all glued together along the torus $ \on{Spec} K[a_e^{\pm 1}, a_{e'}^{\pm 1}]$ to form $\tilde {\bb A}^E$. Note that (in the case where $n$ is not prime) this differs slightly from the example in \cite[remark 3.4]{Holmes2017Extending-the-d} (where a toric interpretation is given) as the rings above are not normal for $0 < i < n$ whenever $n$ and $i$ have a common factor. 

These patches can naturally be seen as charts of a (non-normal) toric blowup. In more involved examples (e.g. \cite[remark 3.5]{Holmes2017Extending-the-d}) there is no canonical way to embed the patches in a blowup, though see also \cite{abreupacini} for a general approach to compactifying. 

While there are infinitely many charts glued together, only those for $0 \le i \le n$ are relevant, the others do not enlarge the space. This is how we glue infinitely many patches to obtain a quasi-compact space. 
\end{example}

%
%
%
%

We now return to the general construction. For a fixed combinatorial chart $U$, pulling these $\tilde {\bb A}^E$ back to $U$ we obtain a stack covered by patches $\Mdm_{I,U}$. Then running over a cover of $\Mbar$ by combinatorial charts, yields a collection of stacks over $\Mbar$ which are easily upgraded to a descent datum. We denote the resulted `descended' object by $\Mdm$. Comparing with the construction of $\Md$ in \cite{Holmes2017Extending-the-d}, one sees the normalisation of $\Mdm$ is $\Md$. Imitating the proof of \cite[Theorem 3.5]{Holmes2017Extending-the-d} shows that the map $\Mdm \to \Mbar$ is separated, of finite presentation, relatively representable by algebraic spaces, and an isomorphism over $\ca M$. If we equip the above objects with their natural log structures, it is also log \'etale. From separatedness and the implication $(1) \implies (2)$ of \cite[lemma 4.3]{Holmes2017Extending-the-d}, we see that the map $\sigma\colon \ca M \to \ca J$ extends (uniquely) to a morphism $\sigma^\m\colon \Mdm \to \ca J$. 

\begin{definition}
We define the \emph{double ramification locus} $\DRL^\m \tra \Mdm$ to be the schematic pullback of the unit section of $\ca J$ along $\sigma^\m$.
\end{definition}  Now $\Md \to \Mdm$ is proper, and by \cite[prop 5.2]{Holmes2017Extending-the-d} the map $\DRL^\loz \to \Mbar$ is proper. But since $\DRL^\loz \to \DRL^\m$ is surjective, by \cite[\href{https://stacks.math.columbia.edu/tag/03GN}{Tag 03GN}]{stacks-project} also $\DRL^\m \to \Mbar$ is proper.
\begin{definition}\label{def:DRprime}
We define the double ramification cycle $\DRC^\m$ to be the cycle-theoretic pullback of the unit section of $\ca J\times_{\Mbar} \Mdm$ along the section induced by $\sigma^\m$, yielding a cycle on $\DRL^\m$. 
\end{definition}The pushforward of $\DRC^\m$ to $\Mbar$ makes sense by properness of $\DRL^\m \to \Mbar$, and the compatibility of the refined gysin pullback with the proper pushforward (see \ref{lem:gysin_compatibility}) implies that the pushforward of $\DRC^\m$ to $\Mbar$ coincides with pushforward of $\DRC^\loz$ to $\Mbar$. See \ref{sec:comparing_DRs} for further details. 

\subsection{A partial normalisation of \texorpdfstring{$\Mdm$}{Mbar\textasciicircum m}}\label{sect:affinepatches}

As discussed in \ref{sec:FP_formula}, on the components of the Double ramification locus supported in the boundary, the twist $I$ is generically divisible by $\k$. 
If we restrict the construction in \ref{sec:charts_of_Md} to twists $I$ taking values in $\k \bb Z$, we obtain an open substack $\Mdm_{\k | I}$ of $\Mdm$. We can define a finite surjective map $\Mdk \to \Mdm_{\k | I}$ by replacing the generators in (\oref{eq:generators_non_normal}) by
\begin{equation}\label{eq:generators_normal}
\prod_{e \in \gamma} a_e^{\pm I(e)/\k}. 
\end{equation}
More concretely, we obtain $\Mdk$ by gluing together patches $\Mdk_{I,U}$ similarly to the procedure in \ref{sec:charts_of_Md}. But now, the patch $\Mdk_{I,U}$ is the pullback of the space $\bb A^E_{I'} \to \bb A^E$ to $U$, where $I'(e) = I(e)/\k$. For clarity and later use, we now give explicit equations for the $\bb A^E_{I'}$ (and hence implicitly for $\Mdk_{I,U}$ since it arises by pulling back $\bb A^E_{I'}$ to $U$). 

Let $\Upsilon$ be the set of cycles $\gamma$ in $\Gamma$ and recall that $E$ is the set of edges in $\Gamma$. Then naturally we can see $\bb A^E_{I'}$ as a subscheme of $\mathbb{A}^{\Upsilon} \times \mathbb{A}^{E}$ cut out by explicit equations. Let $((a_\gamma)_{\gamma \in \Upsilon}, (a_e)_{e \in E})$ be coordinates on $\mathbb{A}^{\Upsilon} \times \mathbb{A}^{E}$, then the generators (\oref{eq:generators_normal}) translate into a system of equations in the $a_\gamma, a_e$. Indeed, given $f \in \mathbb{Z}^\Upsilon$ and $e \in E$ define the integer
\begin{equation}
M_{e, f} = \sum_{\gamma} f_\gamma I'_\gamma(e), 
\end{equation}
where $I'_\gamma = I_\gamma/\k$ (c.f. \ref{eq:I_gamma_e}). 
Then a set of equations cutting out $\bb A^E_{I'} \subset \mathbb{A}^{\Upsilon} \times \mathbb{A}^{E}$ is given by the vanishing of the
\begin{equation} \label{eqn:toriccoordinates}
\Psi_f =   \left(\prod_{\gamma \in \Upsilon : f_\gamma >0} a_\gamma^{f_\gamma} \right) \cdot \prod_{e \in E: M_{e, f}<0} a_e^{-M_{e, f}} -  \left(\prod_{\gamma \in \Upsilon : f_\gamma <0} a_\gamma^{-f_\gamma} \right) \prod_{e \in E: M_{e, f}>0} a_e^{M_{e, f}}
\end{equation}
as $f$ runs through $\mathbb{Z}^{\Upsilon}$. In particular, for any cycle $\gamma$ we have for the inverted cycle $i(\gamma)$, walking in opposite direction, that $a_\gamma a_{i(\gamma)}=1$, which forces $a_\gamma \neq 0$. Apart from that, the most simple equations in the system above are of the form
\begin{equation} \label{eqn:small_toriccoordinates}
a_\gamma \prod_{e \in \gamma:I(e)<0}a_e^{-I'(e)} = \prod_{e \in \gamma:I(e)>0}a_e^{I'(e)}. 
\end{equation}
We will see later that these are the only equations that matter for computing the tangent space to $\Mdk$.

To get the description of $\Mdk_{I',U}$ over $U$ one inserts for the variables $a_e$ the components of the function $g : U \to \bb A^E$ from our combinatorial chart, and obtains equations for $\Mdk_{I',U} \subset U \times \mathbb{A}^{\Upsilon}$.

\begin{remark}\label{rem:ref}
A shorter but less explicit description of the polynomials $\Psi_f$ of \ref{eqn:toriccoordinates} can be obtained by saturating an ideal obtained from the equations \ref{eqn:small_toriccoordinates}. Let $R \coloneqq \field[a_\gamma : \gamma \in \Gamma][a_e: e \in E]$, and let $A$ be the $R$-algebra obtained by formally adjoining inverses to the $a_e$. Let $I$ be the ideal of $A$ generated by 
\begin{equation}
a_\gamma - \prod_{e \in \gamma} a_e^{I'(e)}, 
\end{equation}
and let $I_R$ be the intersection of $I$ with $R$. Then $I_R$ is exactly the ideal generated by the $\Psi_f$ of \ref{eqn:toriccoordinates}. Note that this is \emph{not} in general equal to the ideal generated by polynomials coming from expressions in the form \ref{eqn:small_toriccoordinates}. 
\end{remark}




The map $\Mdk \to \Mdm_{\k | I}$ is finite birational, but in general neither the source nor the target is normal, thus the map does not need to be an isomorphism. Indeed, we have
\begin{lemma}\label{lem:deg_of_normalisation}
Let $p \in \Mdm_{\k | I}$ lie over a simple star graph $\Gamma$ with outlying vertex set $V^{out}$. Then the fibre over $p$ of the map $\Mdk \to \Mdm_{\k | I}$ contains exactly $\k^{\#E(\Gamma)-\#V^{out}} = \k^{b_1(\Gamma)}$ points. 
\end{lemma}
\begin{proof}
Let $\Mdm_{I,\Gamma}$ be a chart containing $p$ with $I=\k \cdot I'$ (here we use $p \in \Mdm_{\k | I}$). Then the preimage of $\Mdm_{I,\Gamma}$ in $\Mdk$ is $\Mdk_{I',\Gamma}$ and the map $\Mdk_{I',\Gamma} \to \Mdm_{I,\Gamma}$ is a base change of the map 
\begin{equation} \label{eqn:kthpowermap}
\bb A^E_{I'} \to \bb A^E_I, ((a_\gamma)_\gamma, (a_e)_e) \mapsto (((a_\gamma)^\k)_\gamma, (a_e)_e).
\end{equation}
Now $p$ corresponds to a point where all the $a_e = 0$, and the values of $a_\gamma$ for $\gamma$ in a basis of $H^1(\Gamma,\mathbb{Z})$ can be chosen freely. 
Once these values of $a_\gamma$ are fixed, all other $a_{\gamma'}$ are determined by \ref{eqn:toriccoordinates}. Thus the number of preimage points under the map \ref{eqn:kthpowermap} is exactly $\k^{b_1(\Gamma)}$, and so the same is true for the pullback $\Mdk_{I',\Gamma} \to \Mdm_{I,\Gamma}$.
%
%
\end{proof}


\begin{remark}
The reader only interested in the case $\k = 1$ will note that in this case the maps 
\begin{equation*}
\Mdk \to \Mdm_{\k | I} \to \Mdm
\end{equation*}
are all isomorphisms, and $I'(e) = I(e)$. 
\end{remark}

\begin{lemma}\label{lem:lci}
Suppose that $\Gamma$ is a simple star graph. Then $\Mdk_{I', U}$ is a local complete intersection over $\field$. 
\end{lemma}
As in \cite{Holmes2017Extending-the-d}, the stack $\Mdk_{I', U}$ can be defined relative to $\bb Z$, in which generality the same lemma holds, with the same proof. The requirement that $\Gamma$ be a simple star seems necessary; the graph 
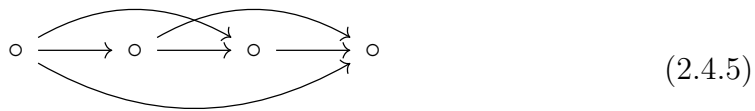
\begin{equation}
 \begin{tikzcd}
  \circ \arrow[r] \arrow[rr, bend left]\arrow[rrr, bend right] & \circ \arrow[rr, bend left]  \arrow[r]& \circ \arrow[r] & \circ \\
\end{tikzcd}
\end{equation}
seems to give a counterexample in general, though we have not checked all details. 
\begin{proof}
Recall the notion of a syntomic morphism (\cite[\href{https://stacks.math.columbia.edu/tag/01UB}{Tag 01UB}]{stacks-project}) generalizing the definition of being a local complete intersection over a field. In particular, the stack $\Mdk_{I', U}$ is a local complete intersection over $\field$ if and only if $\Mdk_{I', U} \to \on{Spec} \field$ is syntomic. Now the class of syntomic morphisms is closed under composition and base-change, and the morphism $\Mdk_{I', U} \to \on{Spec} \field$ factors as
\[\Mdk_{I', U} \to \bb A_{I'}^E \to \on{Spec} \field;\]
moreover the first morphism is smooth (and hence syntomic) as a base change of the smooth morphism $U \to \bb A^E$. Thus it suffices to check that $\bb A_{I'}^E \to \on{Spec} \field$ is syntomic, i.e. that $\bb A_{I'}^E$ is a local complete intersection.


\textbf{Step 1:} Choosing a spanning tree in $\Gamma$ induces a collection $\Upsilon' \sub \Upsilon$ of cycles in $\Gamma$ forming a basis of $H^1(\Gamma, \bb Z)$. Given a cycle $\gamma\in \Upsilon$, writing $\gamma$ as an integral linear combination of elements of $\Upsilon'$ induces an element of $\bb Z^{\Upsilon'}$ (with all coefficients in $\{ -1, 0, 1\}$), whose image in $\bb Z^\Upsilon$ under the natural inclusion $\bb Z^{\Upsilon'} \to \bb Z^\Upsilon$ we denote $f_\gamma$. We denote by $\delta_\gamma \in \bb Z^\Upsilon$ the indicator function for $\gamma$. Then the corresponding expression $\Psi_{f_\gamma- \delta_\gamma}$ (as defined in \ref{eqn:toriccoordinates}) contains no terms $a_e$ with non-zero exponents. 

\textbf{Step 2:} Consider the collection of polynomials consisting of the $\Psi_{f_\gamma- \delta_\gamma}$ for $\gamma \in \Upsilon$. We then claim that the subscheme $Z$ of $\mathbb{A}^{\Upsilon} \times \mathbb{A}^{E}$ cut out by these polynomials is smooth over $\field$ of dimension $\#\Upsilon' + \# E$. First, for $\gamma \in \Upsilon'$ the equation $\Psi_{f_\gamma- \delta_\gamma} = 0$ can be re-written as $1=1$, so can be ignored. Then if $i(\gamma) \in \Upsilon'$, the equation $\Psi_{f_\gamma- \delta_\gamma} = 0$ yields $a_\gamma a_{i(\gamma)}=1$, so $a_{i(\gamma)}$ is inverted. For all other $\gamma \in \Upsilon'$ we can move all the $a_{\gamma'}$ with $\gamma' \in \Upsilon'$ to the left side of the equation (perhaps inverting them), thus writing $a_{\gamma}$ as a product of $a_{\gamma'}^{\pm 1}$. Thus in fact $Z$ a graph of a suitable function $(\bb A^1 \setminus \{0\})^{\Upsilon'} \times \bb A^E \to \bb A^{\Upsilon \setminus \Upsilon'}$, obtained by solving for those coordinates $a_\gamma$ for $\gamma \in \Upsilon \setminus \Upsilon'$. In particular, $Z$ is smooth of the claimed dimension.

\textbf{Step 3:} Since $\bb A_{I'}^E$ has dimension $\#E$, it suffices to show that it is cut out from $Z$ by the $\Psi_{\delta_{\gamma'}}$ as $\gamma'$ runs over $\Upsilon'$. First, given $\gamma \in \Upsilon$, we claim that $\Psi_{\delta_\gamma}$ is contained in the ideal of $\Gamma(Z, \ca O_Z)$ generated by the $\Psi_{\delta_{\gamma'}} : \gamma' \in \Upsilon'$. 
We may assume neither $\gamma$ nor $i(\gamma)$ lies in $\Upsilon'$. Then $\gamma$ consists of two directed edges, say $\gamma = e_1 \circ i(e_2)$, with the $e_i$ going from the central vertex to an outlying vertex. Suppose that the spanning tree contains the edge $e_0$ to that outlying vertex. Then $\gamma$ can be written as a difference of two cycles in $\Upsilon'$: $\gamma = \gamma_1 - \gamma_2$ with $\gamma_1 = e_1 \circ i(e_0)$ and $\gamma_2 = e_2 \circ i(e_0)$. Then 
\begin{equation*}
\Psi_{\delta_{\gamma}} = a_\gamma a_{e_2}^{I'(e_2)} - a_{e_1}^{I'(e_1)},  
\end{equation*}
which is evidently contained in the ideal generated by
\begin{equation*}
\Psi_{f_\gamma - \delta_\gamma} = a_{\gamma_1} - a_{\gamma_2} a_\gamma , \Psi_{\delta_{\gamma_1}} = a_{\gamma_1}a_{e_0}^{I'(e_0)} - a_{e_1}^{I'(e_1)}, \Psi_{\delta_{\gamma_2}} = a_{\gamma_2}a_{e_0}^{I'(e_0)} - a_{e_2}^{I'(e_2)},
\end{equation*}
using that all $a_\gamma$ are invertible on $Z$. 


\textbf{Step 4:} It remains to treat the case of a $\Psi_f$ coming from an arbitrary element $f\in\bb Z^\Upsilon$. The element $f$ induces an element of $H^1(\Gamma, \bb Z)$, which we can view as a subset of $\bb Z^E$; write $F$ for the image of $f$ in $\bb Z^E$. If $F$ is zero then the equation $\Psi_f = 0$ already holds on $Z$. If $F$ is non-zero then there exist a cycle $\gamma \in \Upsilon$ satisfying the assumptions of \ref{lem:cancellation}, and $\Psi_{\delta_\gamma}$ is in our ideal by Step 3, so we may replace $f$ by $f - \delta_\gamma$. Now the sum of the absolute values of the coefficients of $F$ is a positive integer strictly greater than the corresponding term for $f - \delta_\gamma$, so this process must terminate. 
%
%
%
\end{proof}

\begin{lemma}\label{lem:cancellation}
Let $f \in \bb Z^\Upsilon$, and let $\gamma \in \Upsilon$; write $\gamma = (h_1, h_2)$, and set $e_j = \{h_j, i(h_j)\}$. Assume that $M_{e_1, f} \ge 1$ and $M_{e_2, f} \le -1$. Then $\Psi_f$ is contained in the ideal generated by $\Psi_{\delta_\gamma}$ and $\Psi_{f - \delta_\gamma}$. 
\end{lemma}
\begin{proof}
A small calculation with the expressions \ref{eqn:toriccoordinates}. 
\end{proof}

\subsection{Charts and coordinates on the universal curve}
\label{sec:curve_coordinates}

For the deformation theoretic computations later, it will be necessary to fix a system of coordinates on the universal curve  in the neighbourhood of a given point. Suppose we have a combinatorial chart 
\begin{equation*}
\Mbar \stackrel{f}{\longleftarrow} U \stackrel{g}{\longrightarrow} \bb A^E
\end{equation*}
and a point $u \in U(\field)$ mapping to the origin in $\bb A^E$ (if desired we make a finite separable extension of $\field$ so that this exists). Write $U_u = \on{Spec}\ca O_{U,u}$. A \emph{smooth coordinate chart} of the tautological stable curve $\ca C_{U_u}/U_u$ consists of an open subscheme $V \hra \ca C_{U_u}^{sm}$ of the smooth locus of $\ca C_{U_u}$ over $U_u$ with connected fibre over $u$. 
 A \emph{singular coordinate chart} of $\ca C_{U_u}/U_u$ consists of an isomorphism from the strict henselisation\footnote{The strict henselisation is the local ring for the \'etale topology; intuitively, it can be thought of as playing a similar role to that of an $\epsilon$-neighbourhood in the complex analytic world. } of $\ca C_{U_u}$ at a non-smooth point (corresponding to an edge $e = \{h, h'\}$ of $\Gamma$) to the strict henselisation of $\ca O_{U, u}[z_h,z_{h'}]/(z_hz_{h'} - \ell_e)$ at the non-smooth point over $u$, where $z_h$ (resp. $z_{h'}$) vanishes on the component to which $h$ (resp. $h'$) connects. 

We will repeatedly make use of the following

\begin{situation}\label{situation:coordinates}
We fix
\begin{itemize}
\item
 a combinatorial chart $\Mbar \stackrel{f}{\longleftarrow} U \stackrel{g}{\longrightarrow} \bb A^E$ with $\Gamma$ a simple star;
\item 
a twist $I$ of $\Gamma$;
\item  
a $\field$-point $u$ of $U$ lying over the origin of $\bb A^E$;
\item 
a $\field$-point $p$ of $\Mdk_{I,U}$ lying over $u$;
\item
an fpqc cover of the universal stable curve $\ca C_{U_u}/U^u$ consisting of a finite collection of smooth and non-smooth charts as described above. 
\end{itemize}
\end{situation}

\subsection{The universal section \texorpdfstring{$\barsigma$}{sigmabar}}\label{sec:universal_section}

Composing with the map $\sigma^\m\colon \Mdm \to \ca J$ yields a map $\barsigma\colon \Mdk \to \ca J$ extending the section $\sigma\in \ca J(\ca M)$; this is the same as the abel-jacobi map of \cite[\S 4]{Marcus2017Logarithmic-com}, c.f. \ref{sec:MW_connection}. Here we make this map explicit. We fix a combinatorial chart $\Mbar \stackrel{f}{\longleftarrow} U \stackrel{g}{\longrightarrow} \bb A^E$ and a twist $I$, and we work on the chart $\Mdk_{I,U}$ of $\Mdk$. 

Write $\ca C_{I}$ for the universal stable curve over $\Mdk_{I,U}$. On $\ca C_I$ we have the line bundle $\omega^\k(-\m P)$ which has total degree zero on every fibre, but need not have multidegree $\ul 0$ (the zero vector) if the fibres are not irreducible, and so we cannot define $\sigma = [\omega^\k(-\m P)]$. The idea behind the definition of $\Mdk_{I,U}$ is that the multidegree of $\omega^\k(-\m P)$ can be `corrected' to $\ul 0$ by adding on vertical divisors supported over the boundary (`twistors'); details can be found in \cite{Holmes2017Extending-the-d}. 

Later we will need an explicit description of the pullback of this `corrected' bundle to the tautological curve over a (connected) scheme $T \to \Mdk_{I,U}$ such that the composite $T \to \Mdk_{I,U} \to U$ factors via the strict henselisation at the point $u$ (this is to ensure that the local coordinates $z_h$, $z_{h'}$ below make sense on $T$). We will describe this line bundle by giving its pullback to the fpqc cover chosen above, together with transition functions. We first define a `correction' line bundle $\ca T$. Begin by choosing a function 
\begin{equation*}
\lambda\colon \{ \text{directed  edges of }\Gamma\} \to \ca O_T(T)^\times
\end{equation*}
such that $\lambda(-e) = \lambda(e)^{-1}$, and such that for every loop $\gamma$ in $\Gamma$ we have 
\begin{equation*}
\prod_{e \in \gamma} \lambda(e) = a_\gamma. 
\end{equation*}
This is possible: choose a spanning tree $\Gamma' \subset \Gamma$. Then the edges $e$ in $\Gamma$ but not $\Gamma'$ correspond to a basis $\gamma_e$ of the space of cycles (where $\gamma_e$ first takes the edge $e$ and then takes the unique path inside $\Gamma'$ closing the loop). Given this, a possible choice of $\lambda$ is to set  $\lambda=1$ on all edges of $\Gamma'$ and $\lambda(e)=a_{\gamma_e}$ on the remaining edges.

Choose also an orientation on each non-loop edge. The bundle $\ca T$ will then be trivial on each chart of the cover, and we will choose a generating section $\bd 1$ on the smooth charts, and $\tau_e$ on the non-smooth chart corresponding to an edge $e$ of $\Gamma$. The transition function on an intersection of smooth charts sends $\bd 1$ to $\bd 1$. When a non-smooth chart corresponding to an oriented edge $e = \{h, h'\}$ meets a smooth chart, the connected components of the intersection will be contained in $V(z_h)$ or $V(z_{h'})$. On a connected component contained in $V(z_{h'})$ the transition function on the overlap is given by $\tau = \lambda(e)^\k z_h^{I(h)}\bd 1 $, and on a connected component contained in $V(z_{h})$ the transition function is given by $\tau = z_{h'}^{I(h')}\bd 1$.

When two non-smooth charts meet their intersection is necessarily contained in a smooth chart, and so the transition functions are uniquely determined by the previous cases. Then the section $\barsigma$ is defined by the line bundle $\omega^\k(-\m P) \otimes \ca T$; the reader can check that it has multidegree $\ul 0$, or can find the details in \cite[\S5]{Holmes2017Extending-the-d}. Moreover, one verifies that for a family with generically smooth fibre, the bundle $\ca T$ restricts to the trivial bundle on this smooth fibre. This means that on the smooth fibre, $\omega^\k(-\m P) \otimes \ca T$ is just the Abel-Jacobi section and from the separatedness of $\ca J$ it follows that $\omega^\k(-\m P) \otimes \ca T$ is indeed the unique extension to the whole family.

Later on we will want to make some of these choices in a `natural' way on a simple star graph. Suppose thus that $\Gamma$ is a simple star (see \ref{def:simple_star}), and for each outlying vertex $v$ choose one edge $e_v$ to $v$. We take the orientation to be the `outgoing' one from the centre to the outlying vertices. And we uniquely determine $\lambda$ by requiring it to take the value $1$ on $e_v$. A basis of cycles is given by going out along $e_v$ and back along a different edge. If a cycle $\gamma$ is given by $e_v$ and $e'$ then the glueing at the node corresponding to $e'$ gets `adjusted' by exactly $a_\gamma$. Because of the choice of orientation, it is only the glueing on the outlying vertices that gets adjusted by the $a_\gamma$.

\begin{remark}\label{rem:q_root_T}
Because we work on this particular normalisation $\Mdk$, we can also define canonically a $\k$-th root of $\ca T$. The construction is similar to that of $\ca T$; we choose generating sections on the smooth and non-smooth charts (denoted $\bd 1^{1/\k}$ and $\tau_e^{1/\k}$ respectively), then glue on overlaps by the formulae $\tau^{1/\k} = \lambda(e) z_h^{I'(h)}\bd 1^{1/\k} $ and $\tau^{1/\k} = z_{h'}^{I'(h')}\bd 1^{1/\k}$. We denote this new line bundle by $\ca T^{1/\k}$; there is then a unique isomorphism $(\ca T^{1/\k})^{\otimes \k} \to \ca T$ sending $(\bd 1^{1/\k})^\k \mapsto \bd 1$ and $(\tau_e^{1/\k})^\k \mapsto \tau_e$. 
\end{remark}

\begin{definition}\label{def:DRL_DR_1_over_k}
We define the \emph{double ramification locus} $\DRL^{1/\k}$ to be the schematic pullback of the unit section of the universal jacobian along the map $\barsigma$. We define $\DRC = \DRC^{1/\k}$ to be the cycle-theoretic pullback of the unit section of the base change $\ca J \times_{\Mbar} \Mdk$ along the section $\Mdk \to \ca J \times_{\Mbar} \Mdk$ induced by $\barsigma$, as a cycle class on $\DRL^{1/\k}$.
\end{definition} By \cite{Schmitt2016Dimension-theor} we know that all the generic points of $\widetilde {\ca H}_g^\k(\m)$ lie in the locus $\Mdm_{\k | I}$ of $\Mdm$ where the twists are divisible by $\k$. We will use this to show in \ref{sec:comparing_DRs} that to compute $\DRC^\m$ it suffices to compute the multiplicities of $\DRC^{1/\k}$. For most of the rest of this paper, we will be working to compute the multiplicities of $\DRC^{1/\k}$.

\subsection{Comparing the various double ramification cycles and loci}\label{sec:comparing_DRs}

Recall from \ref{sec:DR_construction} that we have various moduli spaces and double ramification loci (with associated cycles), which we summarise in the following diagram:
\[
\begin{tikzcd}
  &\DRL^\loz \arrow[d,symbol=\subset] \arrow[rdd] & &\\
  & \Md \arrow[rdd,very near start, sloped, "\text{proper}" ]& &\\
  \DRL^{1/\k} \arrow[rr] \arrow[d,symbol=\subset] & & \DRL^\m \arrow[dr, near start, sloped, "\text{proper}"]\arrow[d,symbol=\subset] & \widetilde{\mathcal{H}}_g^k(\m) \arrow[d,symbol=\subset]\\
  \Mdk \arrow[r, "\text{proper}"] & \Mdm_{\k | I} \arrow[r, hook, "\text{open imm.}"] & \Mdm \arrow[r] & \Mbar
\end{tikzcd}
\]

In \ref{lem:set_theoretic_image_H} below we will show that that $\DRL^\m \to \Mbar$ factors \emph{set-theoretically} through $\widetilde{\mathcal{H}}_g^k(\m)$. 

Note that (unless $\k = 1$) we do \emph{not} have a map $\Md \to \Mdk$, since $\Mdk$ is a partial normalisation of an open subscheme of $\Mdm$. In \ref{sec:computing_length} we compute the lengths of the local rings of the subscheme $\DRL^{1/\k} \tra \Mdk$. Ultimately we want to show an equality of cycles $\overline{\DRC} = H_{g,\m}^\k$ on $\Mbar$ (recalling that $\overline{\DRC}$ is by definition the pushforward of $\DRC^\loz$ to $\Mbar$), so we need to compare the cycles on these various spaces, and to compare the length with the intersection multiplicity. We begin with a general lemma.

\begin{lemma}\label{lem:gysin_compatibility}
Let $X \stackrel{f}{\to} Y \to \Mbar$ be birational representable morphisms of reduced stacks, with $f$ proper (here `birational' means inducing isomorphisms between some dense open substacks). Suppose that the morphism $\sigma$ extends to $\sigma_X\colon X \to \ca J$ and $\sigma_Y \colon Y \to \ca J$ (necessarily unique, by reducedness). Define $\DRL_X \tra X$ and $\DRL_Y \tra Y$ by pulling back the unit section of $\ca J$ along $\sigma_X$, resp. $\sigma_Y$, and assume that they have the expected codimension $g$. 

Define $\DRC_X$ and $\DRC_Y$ as cycles supported on $\DRL_X$ resp. $\DRL_Y$ as in \ref{def:DRprime}. Then $f_*\DRC_X = \DRC_Y$, an equality of cycles on $\DRL_Y$. 
\end{lemma}
\begin{proof}
 We proceed as in the proof of \cite[theorem 6.7]{Holmes2017Extending-the-d}. Namely, we have a commutative diagram
\begin{center}
\begin{tikzcd}[sep = huge]
\ca J_{X} \arrow[r, "f_J"] \arrow[d] & \ca J_{Y} \arrow[d]\\
X \arrow[r, "f"] \arrow[u, bend left, "\sigma_X"] \arrow[u, bend right, swap, "e_X"] & Y \arrow[u, bend left, "\sigma_{Y}"] \arrow[u, bend right, swap, "e_{Y}"]  \\
\end{tikzcd} 
\end{center}
(here the upward-pointing arrows are closed immersions, by separatedness of $\ca J$, so we can also see them as cycles). Since $f$ is proper and birational, we see that ${f_J}_*[e_X] = [e_{Y}]$. By the commutativity of proper pushforward and the refined Gysin homomorphism\footnote{Here we were not able to find the precise compatibility result we require in the literature (for example, \cite[theorem 3.12]{Vistoli1989Intersection-th} would require that $\ca J_X$ and $\ca J_Y$ be schemes). However, since our $\DRL$ loci have the expected codimension, the cycle $\DRC$ make sense as actual cycles, not just rational equivalence classes. The stated equality can thus be checked locally on $\Mbar$, so we may reduce to the case where all objects in sight are schemes, whereupon we can simply apply \cite[theorem 6.2(a)]{Fulton1984Intersection-th}. }, we see that 
\begin{equation*}
\sigma_{Y}^!{f_J}_*[e_X] = f_* \sigma_X^![e_{X}]
\end{equation*}
(an equality of cycles on $\DRL_Y$), hence 
\begin{equation*}
\DRC_Y = \sigma_{Y}^![e_{Y}] =\sigma_{Y}^!{f_J}_*[e_X] = f_* \sigma_X^![e_{X}] = f_*\DRC_X. \qedhere
\end{equation*}
\end{proof}

Unfortunately, since the map $\Mdm_{\k | I} \to \Mdm$ is not proper, we cannot apply this lemma to compare the double ramification cycles on $\Mdm_{\k | I}$ and on $\Mdm$. 

Recall that by the discussion of \ref{sec:FP_formula}, the underlying (reduced) substack of $\Mbar$ corresponding to $H_{g,\m}^\k$ is the twisted differential space $\widetilde{\mathcal{H}}_{g}^\k(\m)$.
\begin{lemma}\label{lem:set_theoretic_image_H}
The maps $\DRL^\m \to \Mbar$ and $\DRL^{1/\k} \to \Mbar$ factor set-theoretically via $\widetilde{\mathcal{H}}_{g}^\k(\m) \tra \Mbar$. 
\end{lemma}
\begin{proof}
This is clear from the description of the universal bundle in \ref{sec:universal_section} (noting that the same construction of the latter works on $\Mdm$ as on $\Mdk$). 
\end{proof}

\begin{lemma}\label{lem:quasi_finite}
The maps $\DRL^{1/\k} \to \Mbar$ and $\DRL^{\m} \to \Mbar$ are quasi-finite. 
\end{lemma}
\begin{proof}
The map $\Mdk \to \Mdm$ is quasi-finite, so it suffices to check this for $\DRL^\m\to \Mbar$. The map is finitely presented (since there are only finitely many combinatorial charts to consider, by \cite[lemma 3.8]{Holmes2017Extending-the-d}). For a given chart over a given point in $\Mbar$, moving in the fibre of the chart corresponds to shifting the glueing map of the line bundle $\ca T$ at the nodes. 
In particular, it is clear that at most one point of the fibre of the chart can lie in $\DRL^\m$. 
\end{proof}

%
%

\begin{lemma}\label{lem:pure_codim}
The subscheme $\DRL^{1/\k}$ has pure codimension $g$ in $\Mdk$, and $\DRL^\m$ has pure codimension $g$ in $\Mdm$. 
\end{lemma}
\begin{proof}
We give the proof for $\DRL^{1/\k}$; the other case is almost identical. 
Recall that $\DRL^{1/\k}$ is constructed by intersecting two sections in the universal jacobian over $\Mdk$, and the latter is smooth over $\Mdk$ of relative dimension $g$. As such, every generic point of $\DRL^{1/\k}$ has codimension at most $g$, since $\DRL^{1/\k}$ can be cut out locally by $g$ equations. The substack $\widetilde{\mathcal{H}}_{g}^\k(\m)$ has pure codimension $g$ in $\Mbar$ by construction, so we are done by combining \ref{lem:set_theoretic_image_H,lem:quasi_finite}. 
%
\end{proof}

Combining \ref{lem:set_theoretic_image_H,lem:quasi_finite,lem:pure_codim} also yields
\begin{lemma}\label{lem:generic_points_lie_over}
Every generic point of $\DRL^{1/\k}$ and of $\DRL^\m$ lies over a generic point of $\widetilde{\mathcal{H}}_{g}^\k(\m)$. 
\end{lemma}

\begin{lemma}\label{lem:compare_k_k_inverse}\label{lem:length_equals_multiplicity}
Let $\overline p$ be a generic point in $\widetilde{\mathcal{H}}_{g}^\k(\m)$. Then the multiplicity of the cycle $\overline \DRC$ at $\overline p$ is equal to the sum of the lengths of the Artin local rings of $\DRL^{1/\k}$ at (necessarily generic) points $p$ in $\DRL^{1/\k}$ lying over $\overline p$. 
\end{lemma}
This lemma is almost obvious from the definition of the proper pushforward, but we must take a little care as the map $\Mdk \to \Mdm$ is not in general proper, and we must compare the cycle-theoretic multiplicity with the length. 
\begin{proof}
First, \ref{lem:gysin_compatibility} implies that pushforwards of $\DRC^\loz$ and $\DRC^\m$ coincide, so we are reduced to showing the same statement where we replace $\overline{\DRC}$ by the pushforward of $\DRC^\m$ to $\Mbar$. By \ref{lem:generic_points_lie_over}, every generic point of $\DRL^\m$ lies over a generic point of $\widetilde{\mathcal{H}}_{g}^\k(\m)$, so by the discussion of \ref{sec:FP_formula} we know that every generic point of $\DRL^\m$ is contained in $\Mdm_{\k | I}$. 

Now $\Mdm_{\k | I} \hra \Mdm$ is an open immersion and not (in general) proper. But we can still pushforward cycles (\emph{not} cycle classes) along it, simply taking the closure of the image of a prime cycle, and equipping it with the same multiplicity (thus, the same formula as used for proper pushforward). Since $\DRL_{\k | I}^\m = \DRL^\m \cap \Mdm_{\k | I}$ is of pure codimension $g$, we see that $\DRC_{\k | I}^\m$ on it makes sense as a cycle, not just a cycle class. Since every generic point of $\DRL^\m$ is contained in $\Mdm_{\k | I}$, we see that this `naive pushforward' of $\DRC_{\k | I}^\m$ to $\DRL^\m$ coincides with the cycle $\DRC^\m$. 

Now we apply \ref{lem:gysin_compatibility} to the proper morphism $\Mdk \to \Mdm_{\k | I}$, to see that the pushforward of $\DRC^{1/\k}$ to $\DRL_{\k | I}^\m$ yields $\DRC_{\k | I}^\m$. Thus, we see that the composite of maps on \emph{cycles} (not just classes): 
\begin{itemize}
\item proper pushforward $\Mdk \to \Mdm_{\k | I}$;
 \item naive pushforward $\Mdm_{\k | I} \to \Mdm$;
 \item proper pushforward $\Mdm \to \Mbar$
\end{itemize}
sends $\DRC^{1/\k}$ to $\overline{\DRC}$'. 

To conclude the proof, we just need to check that the cycle-theoretic multiplicity of $\DRC^{1/\k}$ at a generic point $p$ of $\DRL^{1/\k}$ coincides with the length of $\DRL^{1/\k}$ at $p$. This holds by \cite[proposition 7.1]{Fulton1984Intersection-th}, if we can show that $\Mdk$ is Cohen-Macaulay at the generic point of each component of $\DRL^{1/\k}$. Since generic points of $\DRL^{1/\k}$ lie over generic points of $\widetilde{\mathcal{H}}_{g}^\k(\m)$ (\ref{lem:generic_points_lie_over}), and the graphs at the latter are simple stars, the local rings of $\Mdk$ are local complete intersections (and hence Cohen-Macaulay) by \ref{lem:lci}. 
\end{proof}

\begin{lemma}\label{lem:surjective_generic_points}
Let $\overline p$ be a generic point of $\widetilde{\mathcal{H}}_{g}^\k(\m)$. Then there exists a generic point of $\DRL^{1/\k}$ mapping to $p$. 
\end{lemma}
\begin{proof}
Combining \ref{lem:quasi_finite,lem:pure_codim} shows that any point in $\DRL^{1/\k}$ mapping to $\overline p$ must be a generic point. 

Fix a (minimal) combinatorial chart containing $\overline p$. From \ref{sec:charts_of_Md}, the twist $I$ on the dual graph of $\ca C_{\overline p}$ determines an affine patch of $\Mdk$ whose image in $\Mbar$ contains $p$. Now in the fibre of $\Mdk$ over that point, the universal line bundle runs over all possible ways of glueing the bundle on the partial normalisation from \cite[definition 1]{Farkas2016The-moduli-spac} to a bundle on the curve $\ca C_{\overline p}$ itself. In particular, one of those `glueings' yields the bundle $\omega^\k(-\m P)\otimes \ca T$ itself, so $\DRL^{1/\k}$ meets that fibre. 
\end{proof}

\begin{proposition} \label{pro:uniquegentwist}
Let $\overline p$ be a generic point of $\widetilde{\mathcal{H}}_{g}^\k(\m)$ and let $\Gamma$ be the dual graph of $\ca C_{\overline p}$, such that $\Gamma$ is a simple star graph. Then there are exactly $\k^{\# E - \# V^{out}}$ points of $\DRL^{1/\k}$ mapping to $p$.  
\end{proposition}
\begin{proof}
Recall that we have the following diagram of maps and inclusions
\[
\begin{tikzcd}
 \DRL^{1/\k} \arrow[r] \arrow[d,symbol=\subset] & \DRL^\m \arrow[r]\arrow[d,symbol=\subset] & \widetilde{\mathcal{H}}_{g}^\k(\m) \arrow[d,symbol=\subset]\\
 \Mdk \arrow[r] & \Mdm \arrow[r] & \Mbar
\end{tikzcd}
\]
For every point $p' \in \DRL^\m$ over a generic point $\overline p \in \widetilde{\mathcal{H}}_{g}^\k(\m)$, by \ref{lem:deg_of_normalisation} there are $\k^{\# E - \# V^{out}}$ points of $\DRL^{1/\k}$ mapping to $p'$. Thus it suffices to show that there is a unique point $p' \in \DRL^\m$ mapping to $\overline p$.

By \ref{lem:surjective_generic_points} there is at least one such $p'$. 
On the other hand, fixing a combinatorial chart $\Mbar \leftarrow U \to \bb A^E$, the spaces $\Mdm_{I,U}$ (as $I$ runs over twists of $\Gamma$) cover the fibre of $\Mdm$ over $\overline p$. We claim that for a fixed $I$ there is at most one preimage point of $\overline p$ in $\Mdm_{I,U}$.

To see this, note that the Abel-Jacobi map $\sigma^\m:\Mdm \to \ca J$ is injective on the fibres of $\Mdm_{I,U} \to \Mbar$. Indeed, we have $\Mdm_{I,U} \subset U \times \bb A^\Upsilon$ and the additional coordinates in $\bb A^\Upsilon$ parametrize different ways to glue the line bundles on the components $C_v$ of the normalization of $\ca C_{\overline p}$. At most one of them maps to the trivial bundle under the Abel-Jacobi map, so at most one preimage point of $\overline p$ can lie in $\DRL^\m$. 

So it is enough to check that for a general point $\overline p$ of a component $Z$ of $\widetilde{\mathcal{H}}_{g}^\k(\m)$ supported in the boundary, there is at most one positive twist $I$ satisfying the twisted-differential condition. Let $\Gamma$ be the generic dual graph of $Z$, which is a simple star graph.


Let $(C,P)$ be a general point of $Z$, so we have an identification of the components/nodes of $C$ with the vertices/edges of $\Gamma$. 
By assumption there is some $m_i$ with $m_i<0$ or $m_i$ not divisible by $\k$. Then the vertex $v_0$ carrying the marking $p_i$ must be the central vertex of the star graph. 
Thus from the abstract dual graph $\Gamma$ we know uniquely which was the central vertex and which the outlying vertices $v \in V^{out}$. 

For each outlying vertex $v$ we take the corresponding component $C_v \subset C$ together with its inherited markings and preimages of nodes. As explained in \ref{sec:FP_formula}  this is a 
generic point of some $\mathcal{H}^1_{g(v)}(\mu')$ for a nonnegative partition $\mu'$ of $2g(v)-2$. In particular $g(v) \geq 1$. If we knew all such $\mu'$, this would allow us to reconstruct the positive twist $I$ on $\Gamma$. But by \ref{lem:twistunique}, knowing a general point $C_v$ of the space $\mathcal{H}^1_{g(v)}(\mu')$ uniquely determines the partition $\mu'$. Thus indeed the twist $I$ is uniquely determined.
\end{proof}

\begin{lemma} \label{lem:twistunique}
Fix $g\geq 1,n \geq 0$ and nonnegative partitions $\mu_1, \mu_2$ of $2g-2$ of length $n$. Then $\mathcal{H}^1_g(\mu_1)$ and $\mathcal{H}^1_g(\mu_2)$ have a common irreducible component $Y$ iff $\mu_1=\mu_2$. 
\end{lemma}
\begin{proof}
Assume $\mu_1 \neq \mu_2$. By the dimension results (\cite{Farkas2016The-moduli-spac}) $Y$ is of pure codimension $g-1$. By reordering we can ensure $\mu_1 =(m_1, m_2, \ldots), \mu_2 =(\tilde m_1, \tilde m_2, \ldots)$ with $0\leq m_1 < \tilde m_1$ (since partitions are not identical). Now on the general point $(C,p_1, \ldots, p_n) \in Y$ we have two equalities of line bundles
\begin{align*}
 \omega_C &\cong \mathcal{O}_C(m_1 p_1 + m_2 p_2 + \ldots + m_n p_n),\\
 \omega_C &\cong \mathcal{O}_C(\tilde m_1 p_1 + \tilde m_2 p_2 + \ldots + \tilde m_n p_n).
\end{align*}
Taking the first equation to power $\tilde m_1$, the second to the power $-m_1$ and tensoring both, we obtain
\begin{equation} \label{eqn:omegamodified} \omega_C^{\tilde m_1 - m_1} \cong \mathcal{O}_C(\underbrace{(\tilde m_1 \cdot m_1- m_1 \cdot \tilde m_1)}_{=0} p_1 + (\tilde m_1 \cdot m_2 - m_1 \cdot \tilde m_2) p_2 + \ldots ).\end{equation}
Thus the general point of $Y$ lies in the space $\mathcal{H}_{g}^{\tilde m_1-m_1}(\widehat \mu)$ for $\widehat \mu = \tilde m_1 \cdot \mu_1 - m_1 \cdot \mu_2$. By \cite[Theorem 1.1]{Schmitt2016Dimension-theor} the space $\mathcal{H}_{g}^{\tilde m_1-m_1}(\widehat \mu)$ has at most codimension $g-1$, so in fact $Y$ must also be a component of this space. But clearly, the condition \ref{eqn:omegamodified} is independent of the position of $p_1$ (the first entry of $\widehat \mu$ vanishes). Thus for $(C,p_1, \ldots, p_n) \in Y$ we also have $(C,q,p_2, \ldots, p_n) \in Y$ for all $q \in C \setminus \{p_2, \ldots, p_n\}$. However, the condition $\omega = \mathcal{O}(\tilde m_1 p_1 + \ldots)$ to be contained in $\mathcal{H}^1_g(\mu_2)$ \emph{does} depend on $p_1$. Indeed, since $g \geq 1$, for any fixed $p_1 \in C$ there are only finitely many points $p_1' \in C$ such that $\mathcal{O}(\tilde m_1 p_1)=\mathcal{O}(\tilde m_1 p_1')$. This implies that $Y$ cannot be a component of $\mathcal{H}^1_g(\mu_2)$, a contradiction.
%
\end{proof}

Together, the results of this section reduce the problem of computing the double ramification cycle to that of computing the length of its local ring at a generic point. This computation will occupy the remainder of the paper. 

\subsection{Relation to the construction of Marcus and Wise}\label{sec:MW_connection}

The construction of Marcus and Wise \cite{Marcus2017Logarithmic-com} produces a stack $\bd{Div}_{g,\bd m}$ over $\Mbar$; we will explain how it is related to our $\Mdm$. We will not need this in what follows, but we feel it may be useful to sketch the connection. 

We begin by outlining the construction of $\bd{Div}_{g,\bd m}$. First, we define the \emph{tropical multiplicative group} to be the functor on log schemes sending $X$ to $\bb G_m^{trop}(X) = \Gamma(X, \overline{M}_X^{gp})$, and a \emph{tropical line over a log scheme $S$} to be a $\bb G_m^{trop}$-torsor over $S$. Then $\bd{Div}$ is the stack in the strict \'etale topology on logarithmic schemes whose $S$-points are triples $(C, P, \alpha)$ where $C$ is a logarithmic curve over $S$, $P$ is a tropical line over $S$, and $\alpha : C \to P$ is an $S$-morphism. We write $\bd{Div}_{g, \bd m}$ for the open substack where the underlying curve has genus $g$, the marked points are labelled $1, \dots, n$ and the outgoing slope at the marked point labelled $i$ is given by $m_i$. 

Now $\bd{Div}_{g, \bd m}$ comes with a natural forgetful map to the stack of all log curves, and we write $\bd{Div}_{g, \bd m}^{st}$ for the pullback of the locus of \emph{stable} log curves. It also comes with an `abel-jacobi' map $\mathrm{aj}\colon\bd{Div}_{g,\bd a}  \to \on{Pic}_{g,n}$, described in \cite[\S 4]{Marcus2017Logarithmic-com}. Write $\Pi$ for the fibrewise connected component in $\on{Pic}_{g,n}$ of the section $\omega^\k(-\sum_i m_i p_i)$. We write $\bd{Div}_{g,\bd m}^{st,\Pi}$ for the pullback of $\Pi$ along $\mathrm{aj}$ to $\bd{Div}_{g, \bd m}^{st}$. Then unravelling the definitions yields a natural isomorphism of log stacks
\begin{equation*}
\bd{Div}_{g,\bd m}^{st,\Pi} \isom \Mdm. 
\end{equation*}

In particular, combining with \ref{lem:gysin_compatibility} yields
\begin{lemma}
The double ramification cycle $\overline{\DRC}\in A^*(\Mbar)$ coincides with that constructed in \cite{Marcus2017Logarithmic-com}. 
\end{lemma}
This generalises \cite[Theorem 7.3]{Holmes2017Extending-the-d} to arbitrary $\k$.

\section{The tangent space to \texorpdfstring{$\Mdk$}{Mbar\textasciicircum \{m,1/k\}} at a simple star}\label{sec:TpMd}


In this section we will give an explicit description of the tangent space $T_p\Mdk$. Since \'etale maps are isomorphisms on tangent spaces we will not distinguish between tangent spaces to a stack and tangent spaces to its charts. We will write $T_u\Mbar$ and $T_u\bb A^E$ in place of $T_{f \circ u}\Mbar$ and $T_{g \circ u}\bb A^E$. The natural map $\Mdk \to \Mbar$ induces a map $T_p\Mdk \to T_u\Mbar$, and we will describe its kernel and image. 

An element of $T_u\Mdk$ is given by a pointed map from the spectrum of $\field[t]/t^2$ to $\Mdk$. Suppose we are given a pointed map from $ \on{Spec} \field[t]/t^2$ to $\Mbar$. For each edge $e$ of $\Gamma$ we denote by $a_e \in t\field[t]/t^2$ the image of $\ell_e$ under the given map $\ca O_{\Mbar} \to \field[t]/t^2$ (see \ref{sec:combinatorial_charts} for the $\ell_e$; in particular, the choice of these coordinates means $\ell_e$ is really well-defined, not only up to units). By the description given in \ref{sec:charts_of_Md}, the point $p$ corresponds to giving an element $(a_\gamma)_p \in \field^\times$ for every cycle $\gamma$, subject to some compatibility conditions. Then, to specify a vector in $T_p \Mdk$ is to give an element $a_\gamma \in (\field[t]/t^2)^\times$ for every cycle $\gamma$ in $\Gamma$, lifting the element $(a_\gamma)_p \in \field^\times$, and subject to the relations \ref{eqn:toriccoordinates}. 

%

This description is valid for graphs of any shape, but we now specialise to the case where $\Gamma$ is a simple star graph (c.f. \ref{sec:FP_formula}), for which things become simpler. For each outlying vertex $v$ write $E_v$ for the set of directed edges from the central vertex $v_0$ to $v$, and choose one edge $e_v \in E_v$. For $e,e' \in E_v$ distinct edges, let $\gamma(e,e')$ be the directed cycle going out along $e$ and back along $e'$. Then the cycles $\gamma(e_v,e')$ for $v \in V^{out}$, $e' \in E_v \setminus \{e_v\}$ yield a basis of the homology of $\Gamma$. 

 For each outlying vertex $v$ there are exactly three possibilities:
\begin{enumerate}
\item
All $e \in E_v$ have $a_{e}^{I'(e)} = 0$;
\item 
All $e \in E_v$ have $a_{e}^{I'(e)}$ non-zero;
\item
There exist $e$ and $e'\in E_v$ such that $a_{e}^{I'(e)}=0$ and $a_{e'}^{I'(e')}\neq 0$. 
\end{enumerate}
In cases (1) and (2), the lifts of $a_\gamma$ from $\field^\times$ to $(\field[t]/t^2)^\times$ can be chosen completely freely for $\gamma$ in a basis of the cycles between $v_0$ and $v$, and these determine all the other $a_\gamma$. To see this, we can split the equations \ref{eqn:toriccoordinates} into two types. We consider first those where the element $f \in \mathbb{Z}^\Upsilon$ corresponds to a trivial class in the homology of $\Gamma$. Then the resulting equation does not contain any instances of the $a_e$, and simply imposes that product of the $a_\gamma$ is 1 (i.e. the group homomorphism from the free abelian group on cycles to $(\field[t]/t^2)^\times$ factors via the homology). On the other hand, if $f \in \mathbb{Z}^\Upsilon$ does not correspond to a trivial class in homology, then the resulting equation will have terms $a_e$ appearing on \emph{both} sides (since twists are positive on outgoing edges). Since the $a_e$ lie in $t\field[t]/t^2$, any lift of $(a_\gamma)_p$ from $\field^\times$ to $(\field[t]/t^2)^\times$ will automatically satisfy these equations (as $t^2 = 0$). Summarising, we can choose the lifts $a_\gamma$ freely for $\gamma$ in a basis of the homology, and the rest are uniquely determined. 

In case (3) we consider the equation 
\begin{equation} \label{eqn:agammasimple} a_{\gamma(e,e')} a_{e'}^{I'(e')}=a_{e}^{I'(e)}.\end{equation}
The terms $a_{e}^{I'(e)}$ and $a_{e'}^{I'(e')}$ do not differ by multiplication by a unit $a_{\gamma(e,e')}$ in $\field[t]/t^2$, so no lift of $\field[t]/t^2$ from $\Mbar$ to $\Mdk$ (i.e. choice of $a_\gamma$) exists.


Using this analysis, we easily determine the kernel and image of the map $T_p\Mdk \to T_u\Mbar$. The kernel is given by the set of lifts of the zero tangent vector, i.e. all $a_e = 0$ in the above discussion. Then it is clear that we are always in case (1), and the lift of $a_\gamma$ from $\field^\times$ to $(\field[t]/t^2)^\times$ can be freely chosen. The set of these choices yields a copy of $\field$, and we see that the kernel of $T_p\Mdk \to T_u\Mbar$ is given by $H^1(\Gamma, \field)$. 

To understand the image of $T_p\Mdk \to T_u\Mbar$ we must distinguish carefully between the cases (1), (2) and (3). Since each $a_e \in t\field[t]/t^2$, we know that $a_e^{I'(e)} = 0$ whenever $I'(e)>1$.
Thus, at a vertex $v$ with at least one incident edge $e$ satisfying $I'(e)>1$, we can never be in case (2). Then we are in case (1) if and only if $a_{e'}=0$ for all $e'$ at $v$ with $I'(e')=1$, and we can choose $a_e \in  t\field[t]/t^2$ arbitrarily for all $e$ with $I'(e)>1$. 

It remains to understand the contribution of the vertices $v$ with all edges $e$ satisfying $I'(e)=1$.
We define $L_v \sub \bigoplus_{e \in E_v} \field$ to be $\{0\}$ if there exists an $e\in E_v$ with $I'(e)>1$, and otherwise to be the set of tuples $(l_e)_e$ satisfying the linear equations 
\begin{equation*}
(a_{\gamma(e, e')})_p l_{e'} = l_{e} \text{ for all } e, e' \in E_v. 
\end{equation*}
Note that in this latter case, $L_v$ has dimension $1$.
Then the image of $T_p\Mdk \to T_u\Mbar$ is given by 
\begin{equation*}
H^1(\ca C_p, \Omega^\vee(-P)) \oplus \bigoplus_{e:I'(e)>1} \field \oplus \bigoplus_v L_v\,.
\end{equation*}
Here the first factor corresponds to the locally trivial deformations, not smoothing the nodes. These correspond to setting $a_e=0$ for all $e$ and thus can always be lifted by the analysis above. The second factor to the free choices of $a_e\in  t\field[t]/t^2 \cong \field$ for $I'(e)>0$, and the last factor to the contribution from vertices with all incident edges $e$ satisfying $I'(e)=1$. This yields an exact sequence
\begin{equation}\label{eq:TpMd_sequence}
0 \to H^1(\Gamma, \field) \to T_p\Mdk \to H^1(\ca C_p, \Omega^\vee(-P)) \oplus \bigoplus_{e:I'(e)>1} \field \oplus \bigoplus_v L_v \to 0, 
\end{equation}
but we can write down a `natural' splitting by setting $a_\gamma = (a_\gamma)_p$ for all $\gamma$, so we get 
\begin{equation}\label{eq:TpMd_explicit}
T_p\Mdk \isom  H^1(\Gamma, \field)  \oplus H^1(\ca C_p, \Omega^\vee(-P)) \oplus \bigoplus_{e:I'(e)>1} \field \oplus \bigoplus_v L_v. 
\end{equation}
This `natural' splitting depends heavily on the choices of coordinates in \ref{situation:coordinates}. Seeing $\Mdk_{I,U} \subset U \times \mathbb{A}^{\Upsilon}$ as in \ref{sect:affinepatches} this splitting is just sending the vector $v \in T_u U$ to $(v,0) \in T_p \Mdk_{I,U} \subset T_p U \times \mathbb{A}^{\Upsilon}$.

\section{The tangent space to the double ramification locus}\label{sec:coker_of_AJ_map}
In this section we will compute the tangent space to the double ramification cycle at the generic point of an irreducible component. Having in \ref{sec:TpMd} analysed the tangent space to $\Mdk$ and decomposed it into direct summands, we will begin by describing the Abel-Jacobi map on the tangent space. We will decompose the dual of the tangent map of the Abel-Jacobi map into four factors (\ref{sec:AJ_on_T_p}), and then in sections \oref{sec:bGamma} to \oref{sec:b_L} compute the intersection of the kernels of these four maps. Along the way we will need a result on the non-vanishing of sums of $\k$th roots of $\k$-residues of $\k$-differentials, which we state and prove in \ref{sec:generic_non-vanishing_residue} as it is somewhat disjoint from the rest of our story, and may be of some independent interest. 

\subsection{The Abel-Jacobi map on the tangent space}\label{sec:AJ_on_T_p}
Assume we are in \ref{situation:coordinates} with $p$ lying in $\DRL^{1/\k}$, so $\barsigma(p) = e(p) \in \ca J(\field)$. Taking the difference of the induced maps $T_p\barsigma$ and $T_pe$ on tangent spaces yields another map $T_p\Mdk \to T_{e(p)}\ca J$, and the composite map with the projection to $T_u \Mbar$
\begin{equation*}
T_p\Mdk \stackrel{T_p\barsigma - T_pe}{\longrightarrow} T_{e(p)}\ca J \to T_u \Mbar
\end{equation*}
is evidently the zero map. Write $\ca J_p$ for the fibre of $\ca J$ over $p$, then $T_e \ca J_p$ is the kernel of the projection $T_{e(p)}\ca J \to T_u \Mbar$, so we have an induced map
\begin{equation*}
b\colon T_p\Mdk \to T_e \ca J_p. 
\end{equation*}

The kernel of $b$ consists of exactly those vectors on which $T_p\barsigma$ and $T_pe$ agree, hence $\ker b = T_p\DRL^{1/\k}$. Dualising the natural exact sequence we obtain an exact sequence
\begin{equation*}
0 \to (\coker b)^\vee \to T_e\ca J_p^\vee \to \left({T_p\Mdk}\right)^\vee \to (T_p\DRL^{1/\k})^\vee \to 0. 
\end{equation*}
In the following we compute 
\[(\coker b)^\vee \subseteq T_e\ca J_p^\vee = H^1(\ca C_p, \ca O_{\ca C_p})^\vee = H^0(\ca C_p, \omega),\]
where the latter equality holds by Serre duality. 

To describe the proof strategy, note that from \ref{{eq:TpMd_explicit}} we have a chosen isomorphism 
\begin{equation*}
T_p\Mdk \isom  H^1(\Gamma, \field)  \oplus H^1(\ca C_p, \Omega^\vee(-P)) \oplus \bigoplus_{e:I'(e)>1} \field \oplus \bigoplus_v L_v.
\end{equation*}
We write the restrictions of $b$ to each subspace appearing as a direct summand as 
\begin{equation*}
\begin{array}{rrl}
 b_\Gamma\colon& H^1(\Gamma, \field) & \to H^1(\ca C_p, \ca O_{\ca C_p}),\\
 b_\Omega\colon& H^1(\ca C_p, \Omega^\vee(-P)) & \to H^1(\ca C_p, \ca O_{\ca C_p}),\\
 b_{>1}\colon&  \bigoplus_{e:I'(e)>1} \field & \to H^1(\ca C_p, \ca O_{\ca C_p}),\\
 b_{L_v}\colon& L_v & \to H^1(\ca C_p, \ca O_{\ca C_p}).
\end{array}
\end{equation*}
%
Elementary linear algebra yields
\begin{lemma}\label{lem:kernel_intersection}
Inside $T_e\ca J_p^\vee = H^0(\ca C_p, \omega)$ we have
\begin{equation*}
(\on{coker}b)^\vee = \on{ker}(b^\vee) = \ker(b_\Omega^\vee) \cap \ker(b_\Gamma^\vee) \cap \ker(b_{>1}^\vee) \cap \bigcap_{v\in V^{out}} \ker(b_{L_v}^\vee). 
\end{equation*}
\end{lemma}

In sections \oref{sec:bGamma} to \oref{sec:b_L} we compute the intersection on the right. 

To describe the final result, note that for a stable curve $C=\ca C_p$ with dual graph $\Gamma$ there exists a natural inclusion
\begin{equation} \label{eqn:holodiffinclusion} \bigoplus_{v \in V(\Gamma)} H^0(C_v, \omega_{C_v}) \to H^0(C,\omega),\end{equation}
taking differentials on the normalizations $C_v$ of the components of $C$ and descending them to $C$. The image is exactly the space of differentials on $C$ with vanishing residues at all nodes.

As a second ingredient, recall from \ref{sec:FP_formula} that for $p$ a general point of a boundary component of $\DRL^{1/\k}$, the stable graph of $\ca C_p$ is a simple star and on the components of $\ca C_p$ corresponding to the outlying vertices, the twisted $\k$-differential on $\ca C_p$ is the $\k$th power of a holomorphic abelian differential. To be more precise, recall that $p$ lying in $\DRL^{1/\k}$ means the line bundle $\omega^\k(-\m P) \otimes \ca T$ is trivial on $\ca C_p$. Let 
\[\phi_0 \in H^0(\ca C_p, \omega^\k(-\m P) \otimes \ca T)\]
be a generating section (unique up to scaling). Recall from \ref{sec:universal_section} that the bundle $\ca T$ has a generating section $\bd 1$ on the smooth part of $\ca C_p$. Let $\bd 1_v$ be the restriction of $\bd 1$ to the components $C_v$ of $C=\ca C_p$, then $\phi_0/\bd 1_v$ is a meromorphic section of $\omega^\k(-\m P)$ on $C_v$. As described in \ref{sec:FP_formula}, for $v$ an outlying vertex, this section is actually a holomorphic $\k$-differential, which is moreover a $\k$th power of a holomorphic differential. Denote one such choice of a $\k$th root by $(\frac{\phi_0}{\bd 1_v})^{1/\k}$. 

As a final piece of notation, let $V^{>1} \subset V(\Gamma)$ denote the set of outlying vertices, such that at least one edge $e$ incident to $v$ has $I'(e)>1$, and denote by $V^1$ the remaining outlying vertices. Then we can state the main result of this section (whose proof follows \ref{lem:kernel_formula}). 

\begin{theorem}\label{thm:ker_b_v}
 Let $p$ be the generic point of a boundary component of $\DRL^{1/\k}$. Then the kernel of $b^\vee$ inside $H^0(\ca C_p, \omega)$ is given by the injection 
\begin{equation*}
\bigoplus_{v\in V^{>1}} \field \to H^0(\ca C_p,  \omega)
\end{equation*}
sending $(c_v)_v$ to the section given by $0$ on the smooth locus of the central vertex, and $c_v (\frac{\phi_0}{\bd 1_v})^{1/\k}$ on the smooth locus of the outlying vertex $v$. 
\end{theorem}

\begin{theorem}\label{thm:dim_T_p_DR}
In the situation of \ref{thm:ker_b_v}, for the exact sequence
\begin{equation*}
0 \to T_p\DRL^{1/\k} \to T_p \Mdk \stackrel{b}{\to} T_e \ca J_p \to \coker (b) \to 0, 
\end{equation*}
the cokernel of $b$ has dimension equal to $\#V^{>1}$. The dimension of $T_p\DRL^{1/\k}$ is given by $\dim \DRL^{1/\k}$ plus the number of edges $e$ in the star graph having $I'(e) >1$. 
\end{theorem}
\begin{proof}
The assertion about the cokernel of $b$ comes from \ref{thm:ker_b_v}, and the equality of $\dim \ker(b^\vee)$ with $\dim \coker b$. To compute the dimension of $T_p \DRL^{1/\k}$ we must first compute the dimension of $T_p\Mdk$; following \ref{{eq:TpMd_explicit}} it is given by $$h^1(\ca C_p, \Omega^\vee(P)) + h^1(\Gamma) + \#\{e:I'(e)>1\} + \#V^1, $$
where we use that $L_v$ has dimension $1$ for $v \in V^1$. Then 
\begin{itemize}
\item $h^1(\ca C_p, \Omega^\vee(P)) = 3g - 3 + n - \# E$ (locally trivial deformations)
\item $h^1(\Gamma)  = 1 - \# V + \# E$
\item $\dim T_e \ca J_p = g$
\item $\dim \DRL^{1/\k} = 3g-3 + n - g$ (by \ref{lem:pure_codim}).  
\end{itemize}
We see that 
\begin{align*}
\dim T_p \DRL^{1/\k}  =& (3g - 3 + n - \# E) + (1 - \# V + \# E) + \#\{e:I'(e)>1\} \\ &+ \#V^1 + \#V^{>1} - g\\
=& (3g-3 + n - g) +  \#\{e:I'(e)>1\}\\
=& \dim \DRL^{1/\k}  + \#\{e:I'(e)>1\}
\end{align*}
as required. 
\end{proof}

\subsection{Computing the kernel of \texorpdfstring{$b_\Gamma^\vee$}{b\_Gamma\textasciicircum v}} \label{sec:bGamma}

\begin{lemma}
 For any point $p \in \Mdk$ with underlying curve $\ca C_p$, the kernel of $b_\Gamma^\vee$ is given by the inclusion
 \begin{equation*} \bigoplus_{v \in V(\Gamma)} H^0(C_v, \omega_{C_v}) \to H^0(C,\omega),\end{equation*}
 as described in \ref{eqn:holodiffinclusion}.
\end{lemma}
\begin{proof}
We start by recalling some generalities about line bundles on nodal curves. Let $C$ be a nodal curve with dual graph $\Gamma=(V,E)$. Choose some orientation for the edges $e \in E$ such that we can uniquely identify source and target $s(e), t(e) \in V$ of each edge. Moreover let $(C_v)_{v \in V}$ be the set of components of the normalization of $C$ and for an edge $e$ let $n'(e) \in C_{s(e)}$,  $n''(e) \in C_{t(e)}$ be the preimages of the nodes corresponding to the edge. 

Then a line bundle $L$ on $C$ is given by a collection of line bundles $(L_v)_{v \in V}$ on all components of its normalization together with identifications of the fibres $(\sigma_e : L_{s(e)}|_{n'(e)} \xrightarrow{\sim} L_{t(e)}|_{n''(e)})_{e \in E}$ of these line bundles at the pairs of points mapping to the same node. These identifications $(\sigma_e)_{e \in E}$ have a natural action by the group $(\mathbb{G}_m)^E$ by componentwise multiplication. The set of such identifications is a torsor under this action. Moreover, different identifications can give the same line bundle: multiplying all the fibres on a given vertex $v \in V$ by the same constant $\mu$, i.e. going from $(\sigma_e)_{e \in E}$ to $(\sigma_e \cdot \mu^{\delta_{v,t(e)}-\delta_{v,s(e)}})_{e \in E}$ does not change the line bundle. Moreover, multiplying \emph{all} fibres by the same constant does not even change the set of identifications $(\sigma_e)_{e \in E}$. This means we have an effective action of the group $\mathbb{G}_m^V/\mathbb{G}_m$ on $\mathbb{G}_m^E$ which does not change the line bundle on $C$. Making suitable choices we can identify the quotient of $\mathbb{G}_m^E$ by $\mathbb{G}_m^V/\mathbb{G}_m$ with the torus $T=\on{Hom}(H^1(\Gamma,\mathbb{Z}),\mathbb{G}_m)$. 

Now let $p \in \Mdk$ with underlying stable curve  $C=\ca C_p$; this determines a $\k$th root $\ca T^{1/\k}$ of the correction bundle $\ca T$ on $C$ as described in \ref{rem:q_root_T}. In fact, in the local charts for $\Mdk$, the additional coordinates $a_\gamma$ exactly parametrize the gluing data for $\ca T^{1/\k}$. Allowing $T$ to act on the bundle $\ca T^{1/\k}$ as described above yields a faithful action of $T$ on the fibre of $\Mdk$ over $C$, whose orbit is open (it is exactly the fibre of $\Mdk_{I,U}$ for the relevant weighting $I$  at $p$  and any neighbourhood $U$ of $p$). Hence the summand $H^1(\Gamma,\field)$ in $T_{p'} \Mdk$ is canonically identified with the tangent space to $T$ at $1$. On the other hand, there is also an action of $T$ on the jacobian of $C$, where elements of $T$ act on line bundles on $C$ in the way described above. However, we want to take the action obtained from this usual action by composing with the group morphism $T \to T, t \mapsto t^\k$. With respect to this new action, 
the Abel-Jacobi map  is equivariant. This follows since, on the fibre in $\Mdk$ over $C \in \Mbar$, the Abel-Jacobi map just sends $\ca T^{1/\k} \mapsto \omega_C^\k(-\m P) \otimes (\ca T^{1/\k})^\k$. Hence the tangent map $b_\Gamma : H^1(\Gamma,\field) \to H^1(C,\mathcal{O}_{C})$ to the Abel-Jacobi map on this fibre is given by the tangent map for the action of $T$ on the jacobian of $C$.  One then verifies that for the map 
\begin{equation*}
\tanmap \colon H^1(\Gamma, \field) = \frac{\field^E}{\field^V} = \frac{H^0(C, \field^\mathrm{nodes})}{H^0(C, \pi_*\ca O_{\tilde C})}  \to H^1(C, \ca O_{C}), 
\end{equation*}
coming from the long exact sequence of $0 \to \ca O_{C} \to \pi_*\ca O_{\tilde C} \to \field^\mathrm{nodes}\to 0$, where $\tilde C \to C$ is the normalisation, we have $b_\Gamma=\k \tanmap$. Thus, for computing kernels and cokernels we may as well work with the map $\tanmap$ above.

The kernel of $b_\Gamma^\vee$ is then equal to the left kernel of the Serre Duality pairing
\begin{equation*}
H^0(C, \omega) \times \tanmap\left(H^1(\Gamma, \field)\right) \to H^1(C, \omega) = \field.
\end{equation*}
This left kernel is equal to the inclusion $\bigoplus_v H^0(C_v, \omega_{C_v}) \subset H^0(C, \omega)$ from the statement of the lemma. 
To see that $\bigoplus_v H^0(C_v, \omega_{C_v})$ is contained in the left kernel, note that the cocycles in the image of $\tanmap$ are represented by constant functions on the overlaps of the \v Cech cover, so multiplying them with holomorphic differentials from $\bigoplus_v H^0(C_v, \omega_{C_v})$ does not produce poles. Hence the residue pairing indeed vanishes. However, note that the map $\tanmap$ is injective from the long exact sequence we used to define it. Thus, the dimension of the left kernel is 
\begin{align*}
 \dim H^0(C, \omega) - \dim H^1(\Gamma, \field) = g - b_1(\Gamma) = \sum_{v} g(v) = \sum_v \dim H^0(C_v, \omega_{C_v}),
\end{align*}
so indeed $\bigoplus_v H^0(C_v, \omega_{C_v})$ is equal to the left kernel.
\end{proof}

Though we will not need it in what follows, we mention
\begin{lemma}\label{lem:T_p_injective}
The projection map $T_p\DRL^{1/\k} \to T_p\Mbar$ is injective. 
\end{lemma}
\begin{proof}
We have $T_p \DRL^{1/\k} \subset T_p \Mdk \to T_p \Mbar$ and the kernel of $T_p \Mdk \to T_p \Mbar$ is given by $H^1(\Gamma,\field)$. On the other hand, $T_p \DRL^{1/\k}$ is the kernel of $b$, the differential of the Abel-Jacobi map, and the restriction of $b$ to  $H^1(\Gamma,\field)$ is $b_\Gamma=\k \tanmap$. So an element of the kernel of $T_p \DRL^{1/\k} \to T_p \Mbar$ is an element of the kernel of $\tanmap$ and $\tanmap$ is injective. 
\end{proof}

\subsection{Computing the kernel of \texorpdfstring{$b_\Omega^\vee$}{b\_Gamma\textasciicircum v}}
We put ourselves in \ref{situation:coordinates} with $p$ a $\field$-point of $\DRL^{1/\k}$. As described in \ref{sec:AJ_on_T_p}, we can choose a generating section 
\begin{equation*}
\phi_0 \in H^0(\ca  C_p,\omega^\k(-\m P)\otimes \ca T).
\end{equation*}
We will use the description in \ref{sec:explicit_def_via_cech} to compute explicitly the map \[b_\Omega\colon H^1(\ca C_p, \Omega^\vee(-P)) \to H^1(\ca C_p, \ca O_{\ca C_p}).\]
Recall that for covers of $\ca C_p$ as in \ref{situation:coordinates}, all non-trivial intersections of charts map to the smooth locus of $\ca C_p/\field$, and the line bundle $\ca T$ comes with a trivialisation on that smooth locus, described by a generating section $\bd 1$. 

 \begin{lemma}
Working in \v Cech cohomology for a cover of $\ca C_p$ as in \ref{situation:coordinates}, the map 
\[b_\Omega: \cecH^1(\ca C_p, \Omega^\vee(-P)) \to \cecH^1(\ca C_p, \ca O_{\ca C_p})\]
is induced by the map  
\[g_{ij} \mapsto - g_{ij} \frac{d(g_{ij}^\k\frac{\phi_0}{\bd 1})}{g_{ij}^\k\frac{\phi_0}{\bd 1}} = - g_{ij} d \log \left(g_{ij}^\k\frac{\phi_0}{\bd 1}\right)\] 
on $1$-cocycles $(g_{ij})_{ij}$ of $\Omega^\vee(-P)$.
 \end{lemma}
 For an interpretation of this formula, note that on the smooth locus of $\ca C_p$ we can interpret the $g_{ij}$ as tangent fields on the overlaps $U_{ij}$ of the cover. Thus the pairing $g_{ij}^\k\frac{\phi_0}{\bd 1}$ of $g_{ij}^\k$ with the meromorphic differential $\frac{\phi_0}{\bd 1}$ makes sense as a meromorphic function. Applying the external derivative $d$ gives a meromorphic differential, which we again pair with $g_{ij}$ to obtain a meromorphic function.  Given the simple nature of the formula, it feels as if there should be a simple conceptual proof of this lemma. But our description of the line bundle $\ca T$ was somewhat ad-hoc, making it necessary to keep careful track of all the glueing data. Since this bookkeeping is quite subtle, we have  written the proof out in a painful amount of detail. 
 \begin{proof}
 To prove the statement, we make (more) explicit the calculations from \ref{sec:explicit_def_via_cech}. Our short exact sequence $0 \to J \to A' \to A \to 0$ of $\field$-modules is just
 \begin{equation*}
 0 \to t\field \to \field[t]/t^2 \to \field \to 0, 
 \end{equation*}
and we represent an element $g \in \cecH^1(\ca C_p, \Omega^\vee(-P))$ as a cocycle $(g_{ij})_{ij}$ on the cover $\{V_i\}_i$ of $\ca C_p$. Then $\phi_0$ is a generating section of the line bundle $\ca L_\field = \omega^\k(-\m P) \otimes \ca T$. 

Recall that $g$ encodes a locally trivial deformation $C_{A'}/A'$ of $C_p$. To obtain it, define $U_i = V_i \times_\field \field[t]/t^2$, then we can choose isomorphisms $U_{ij} \to V_{ij}\times_\field \field[t]/t^2$ such that the inclusion $f_i\colon U_{ij} \to U_i$ is just the base-change of $V_{ij} \to V_i$ to $\field[t]/t^2$, and the inclusion $f_j\colon U_{ij} \to U_j$ is given by applying Spec to the ring homomorphism $r \mapsto r + tg_{ij}(dr)$. Glueing the $U_i$ together along the $U_{ij}$ yields the desired locally trivial deformation $C_{A'}/A'$. 

%


Recall from \ref{sec:explicit_def_via_cech} that, to compute the image of $g$ in $\cecH^1(\ca C_p, \omega^\k(-\m P)\otimes \ca T)$, we should consider the line bundle $\ca L_{A'} = \omega_{C_{A'}}^\k(-\m P) \otimes \ca T$, choose generating sections $\phi_i$ on $U_i$, $\phi_j$ on $U_j$, then pull them both back to $U_{ij}$ and compare. It will be very important to ensure that the pullback is performed in a functorial way, so we can effectively compare these pullbacks. 

For the first tensor factor of $\ca L_{A'}$, we identify $f_i^*( \omega_{C_{A'}} |_{U_i} )$ with $\omega_{U_{ij}}$ via the differential $df_i$ (and similarly with $df_j$).  The fact that the differential is naturally functorial later ensures that this gives compatible identifications. 

The second factor $\ca T$ is, in a sense, more tricky, because we defined $\ca T$ in terms of a cover with gluing maps and thus we need to be extra careful how to identify pullbacks of $\ca T$ under various compositions of maps. First, to fix terminology we recall the following very standard description of the pullback of a line bundle given by glueing data: 

\textbf{Digression (Pullbacks of line bundles)} Let $f:X \to Y$ be a morphism of schemes and $Y=\bigcup_i U_i$ an open cover. Assume a line bundle $\ca L$ on $Y$ is given by fixing
\begin{itemize}
 \item a generating section $\bd 1_{U_i}$ of $\ca L$ on $U_i$,
 \item functions $\rho_{ij}$ on $U_{ij}=U_i \cap U_j$, telling us that $\bd 1_{U_j}|_{U_{ij}} = \rho_{ij} \bd 1_{U_i}|_{U_{ij}}$.
\end{itemize}
Giving a section $s$ of $\ca L$ on $Y$ means giving local sections $s_i \bd 1_{U_i}$ satisfying $s_j \bd 1_{U_j}|_{U_{ij}} = s_j \rho_{ij} \bd 1_{U_i}|_{U_{ij}}$, in other words $s_j \rho_{ij} = s_i$ on $U_{ij}$.

Then the pullback of $\ca L$ under $f$ is given by the cover $X = \bigcup_i f^{-1}(U_i)$ with new generating sections $\bd 1_{f^{-1}(U_i)}$ and gluing functions $\eta_{ij}=f^* \rho_{ij}=\rho_{ij} \circ f$. The pullback of the section $s$ is specified by the local sections $(s_i \circ f) \cdot \bd 1_{f^{-1}(U_i)}$.

\textbf{End digression.}

Now for the line bundle $\ca T$ on $C_{A'}$ we use the cover by the $U_i$ above. For smooth $U_i$ we have a trivializing section $\bd 1$, which we here call $\bd 1_{U_i}$ to be more precise. For the singular chart $U_j$ associated to an edge $e = \{ h, h'\}$ we have the trivializing section $\tau_e$. The transition functions were $1$ between two smooth charts, and for a smooth chart $U_i$ and a singular chart $U_j$, the transition function is $\tau_e = \lambda(e)^\k z_h^{I(h)} \bd 1_{U_i}$ if $U_i$ is a chart on an outlying vertex and $\tau_e = (z_{h'})^{I(h')} \bd 1_{U_i}$ if $U_i$ is a chart on the central vertex. Since the formula uses the coordinates $z_h, z_{h'}$ on $U_j$, we implicitly identify $U_{ij}$ as a subset of $U_j$ here (since on $U_i$ the expression $z_h$ has no meaning). Note that overlaps between singular charts are already contained within the overlaps of smooth and singular charts, so the values of the cocycles on these patches are uniquely determined by those we have already listed.


Now we proceed to choose generating sections $\phi_i$ on $U_i$, $\phi_j$ on $U_j$, pull them both back to $U_{ij}$ and compare. 
We consider the case when $U_i$ is a smooth chart and $U_j$ is a singular chart on an outlying vertex. Denote by $\pi_i : U_i \to V_i$ and $\pi_j : U_j \to V_j$ the natural projections. Also, denote by $\bd 1_{V_i}$ and $\tau_{e,V_j}$ the trivializing sections of $\ca T|_{\ca C_p}$ on $V_i, V_j$. Then the section $\phi_0/\bd 1_{V_i}$ is a generating section of $\omega_{V_i}^\k(-\m P)$ on $V_i$ and so $\tilde \phi_i \coloneqq \pi_i^*(\phi_0/\bd 1_{V_i})$ is a generating section of $\omega_{U_i}^\k(-\m P)$ on $U_i$. Denoting by $\bd 1_{U_i}$ the trivializing section of $\ca T|_{U_i}$, we choose $\phi_i \coloneqq \tilde \phi_i \otimes \bd 1_{U_i}$ as the generating section of $\omega^\k(-\m P)\otimes \ca T$ on $U_i$.

For $U_j$ something different happens: on $V_j$ we have that $\phi_0/\tau_{e,V_j}$ is a generating section of $\omega_{V_j}^\k(-\m P)$. Denote by $\tilde \phi_j \coloneqq \pi_j^*(\phi_0/\tau_{e,V_j})$ the section of $\omega_{U_j}^\k(-\m P)$ and by $\tau_{e,U_j}$ the section of $\ca T$ on $U_j$, then $\phi_j\coloneqq \tilde \phi_j \otimes \tau_{e,U_j}$ is our chosen generating section of $\omega^\k(-\m P)\otimes \ca T$ on $U_j$. 

Now since $f_i$ is just an inclusion, indeed $f_i^* \phi_i = \tilde \phi_i|_{U_{ij}} \otimes \bd 1_{U_{ij}}$. On the other hand, $f_j$ is the composition of the automorphism $\Psi_{g_{ij}}: U_{ij} \to U_{ij}$ (obtained as $\on{Spec}$ of the ring map $r \mapsto r + tg_{ij}(dr)$) with the inclusion $U_{ij} \subset U_j$. But now note that when restricting $\phi_j$ to $U_{ij}$ we obtain
\[\phi_j|_{U_{ij}} = \tilde \phi_j|_{U_{ij}} \otimes \tau_{e,U_{ij}}.\]
However, by the original gluing data of $\ca T$ we have 
\[\tau_{e,U_{ij}} = \lambda(e)^k z_h^{I(h)} \bd 1_{U_{ij}}. \]
The crucial thing to observe is that the transition function $\lambda(e)^k z_h^{I(h)}$ is actually a pullback from $V_{ij}$ (it `does not contain the variable $t$'). Denote by $\pi_{ij}: U_{ij} \to V_{ij}$ the projection. Using again that $\tilde \phi_j = \pi_j^*(\phi_0/\tau_{e,V_j})$, we make a quick sanity check:
\begin{align*}
 \phi_j|_{U_{ij}} &= \tilde \phi_j|_{U_{ij}} \otimes \tau_{e,U_{ij}}\\
 &= \pi_j^*(\phi_0/\tau_{e,V_j})|_{U_{ij}} \otimes \lambda(e)^k z_h^{I(h)} \bd 1_{U_{ij}}\\
 &= \pi_j^*(\phi_0/\tau_{e,V_j} \lambda(e)^k z_h^{I(h)} )|_{U_{ij}} \otimes  \bd 1_{U_{ij}}\\
 &= \pi_{ij}^*(\phi_0/\tau_{e,V_j} \lambda(e)^k z_h^{I(h)} |_{V_{ij}} ) \otimes  \bd 1_{U_{ij}}\\
 &= \pi_{ij}^*(\phi_0|_{V_{ij}} /\bd 1_{V_{ij}}  ) \otimes  \bd 1_{U_{ij}}\\
 &= \pi_{ij}^*(\phi_0 /\bd 1_{U_{i}}  |_{V_{ij}}) \otimes  \bd 1_{U_{ij}}\\
 &= \pi_{i}^*(\phi_0 /\bd 1_{U_{i}}  )|_{U_{ij}} \otimes  \bd 1_{U_{ij}}\\
 &= \tilde \phi_i |_{U_{ij}} \otimes  \bd 1_{U_{ij}}\\
 &= \phi_i|_{U_{ij}}.
\end{align*}

To conclude, by \ref{lem:class_of_deformation} the class in $\cecH^1(\ca C_p, \omega^\k(-\m P)\otimes \ca T)$ that we seek is given by $f_i^*\phi_i - f_j^*\phi_j$. 
Now $f_i^* \phi_i = \tilde \phi_i|_{U_{ij}} \otimes \bd 1_{U_{ij}}$ and 
\[ f_j^* \phi_j =  \Psi_{g_{ij}}^* \left(\tilde \phi_i |_{U_{ij}} \otimes  \bd 1_{U_{ij}} \right) =\left(d\Psi_{g_{ij}} \tilde \phi_i |_{U_{ij}}\right) \otimes  \bd 1_{U_{ij}}. \]
Using \ref{lem:f_Omega} below, we see 
\[d\Psi_{g_{ij}} \tilde \phi_i |_{U_{ij}} = \tilde \phi_i|_{U_{ij}}  + tg_{ij}^{1 - \k}d(g_{ij}^{\k}\phi_0/\bd 1_{U_{ij}}).\]
This allows us to conclude
\begin{equation*}
f_i^*\phi_i - f_j^*\phi_j = -tg_{ij}^{1 - \k}d(g_{ij}^{\k}\phi_0/\bd 1_{U_{ij}}) \otimes \bd 1_{U_{ij}}, 
\end{equation*}
and pulling back along the `multiplication by $t$' isomorphism from $A$ to $J$ yields the element $[- g_{ij}^{1 - \k}d(g_{ij}^{\k}\phi_0/\bd 1_{U_{ij}}) \otimes \bd 1_{U_{ij}}]_{ij}$ in $\cecH^1(C_p, \omega^q(-\m P) \otimes \ca T)$. To translate this to an element of $\cecH^1(C_p, \ca O_{\ca C_p})$ we use the isomorphism $\omega^q(-\m P) \otimes \ca T \cong \ca O_{\ca C_p}$ via dividing by $\phi_0$. This gives the desired result 
\[- \frac{g_{ij}^{1 - \k}d(g_{ij}^{\k}\phi_0/\bd 1) \otimes \bd 1}{\phi_0} = - g_{ij} \frac{d(g_{ij}^\k\frac{\phi_0}{\bd 1})}{g_{ij}^\k\frac{\phi_0}{\bd 1}}\]
 \end{proof}
\begin{lemma}\label{lem:f_Omega}
Let $B$ be an $A'$-algebra, and $g \colon \Omega_{B_\field/\field} \to B_\field$ a $B_\field$-linear map. Define $f\colon B \to B; r \mapsto r + tg(d(r|_{B_\field}))$, where we use that the map $B_\field \to B, u \mapsto t\cdot u$ is well-defined. Then the map $f_\Omega\colon \Omega_{B/A'}^\k \to \Omega_{B/A'}^{\k}$ induced by the differential of $f$ is given by $f_\Omega(w) = w + t  g^{1-\k}d(g^\k (w|_{B_\field}))$, where we use that $\Omega_{B_\field/\field}^\k \to \Omega_{B/A'}^\k, \eta \mapsto t \cdot \eta$ is well-defined. 
 \end{lemma}

  \begin{proof}
 We prove the result for local generators $w = r_0 dr_1 \cdots dr_\k$ of $\Omega_B^\k$. In the computation below we implicitly use that $tdr_j = t(dr_j|_{B_\field})$ and that the restriction to the fibre $B_\field$ commutes with taking differentials. We then obtain 
 \begin{equation*}
 \begin{split}
 f_\Omega(w) &= f(r_0) df(r_1) \cdots df(r_\k) \\
 & = (r_0+tg(dr_0|_{B_k}))\prod_i (dr_i + td(gd(r_i|_{B_\field})))\\
 & = w + t\left( g(dr_0|_{B_\field}) \prod_i dr_i + r_0 \sum_i \left((dg)d(r_i|_{B_\field})\prod_{j \neq i} dr_j \right)\right)\\
 & = w + t\left( g(dr_0|_{B_\field})  + r_0 \sum_i dg\right)\prod_i dr_i|_{B_\field}\\
 & = w + t\left( g(dr_0|_{B_\field})  + \k r_0 dg\right)\prod_i dr_i|_{B_\field}\\
 & = w + tg^{1- \k}d(g^\k (r_0 dr_1 \ldots dr_\k|_{B_\field})). 
 \end{split}
 \end{equation*}
 Note that in the computation we used a natural extension of the differential $d$ to tensor powers of the tangent sheaf of $B$, when we apply it to $g$ and $g^\k$.
 \end{proof}

 Recall that
 \begin{equation*}
b_\Omega^\vee \colon H^0(\ca C_p,  \omega) \to H^0(\ca C_p,\ca Hom(\Omega^\vee(-P), \omega))
\end{equation*}
denotes the map induced by applying Serre Duality to $b_\Omega$. Our next goal is to give an explicit formula for this map; or rather, for its restriction to $\ker b_\Gamma^\vee = \oplus_v H^0(C_v,\omega_{C_v}) \sub H^0(\ca C_p,  \omega)$, since this is all we need later. 

We define a map of coherent sheaves 
on the smooth locus $\ca C_p^{sm}$:
%
  \begin{equation*}
 \begin{split}
\beta^{sm}\colon  \omega& \to \ca Hom(\Omega^\vee(-P), \omega)\\
 s & \mapsto  \left[f \mapsto f\left(s^{1 - \k}\left(\frac{\phi_0}{\bd 1}\right) d\left(\left(\frac{\phi_0}{\bd 1}\right)^{-1}s^\k\right)\right) = f\left(s d \log \left(\left(\frac{\phi_0}{\bd 1}\right)^{-1}s^\k\right)\right) \right].  
 \end{split}
 \end{equation*}
Note in particular that this map makes sense at markings of $\ca C_p$: while the differential $d \log((\phi_0/\bd 1)^{-1} s^k)$ can have a pole of order $1$ at a marking, we have that $f$ is a local section of $\Omega^\vee(-P)$, so a vector field vanishing at the marking. This cancels the possible pole.

Inside $H^0(\ca C_p,\omega)$ there is the subspace $\oplus_v H^0(C_v,\omega_{C_v})$ of global sections of $\omega$ with vanishing residues at all nodes. The following result shows that it makes sense to apply $\beta^{sm}$ to elements in this subspace.

\begin{lemma}
 Let $s \in \oplus_v H^0(C_v,\omega_{C_v}) \subset H^0(\ca C_p,\omega)$, then the section $\beta^{sm}(s|_{\ca C_p^{sm}})$ of $\ca Hom(\Omega^\vee(-P), \omega)$ on $\ca C_p^{sm}$ extends uniquely to all of $\ca C_p$ and we thus obtain a map 
 \begin{equation*}
\beta  \colon \bigoplus_v H^0(C_v,\omega_{C_v}) \to H^0(\ca C_p,\ca Hom(\Omega^\vee(-P), \omega)). 
\end{equation*}
\end{lemma}
\begin{proof}
 What we need to check is that for every local section $f$ of $\Omega^\vee(-P)$ around a node of $\ca C_p$, the section 
 \begin{equation} \label{eqn:betaexpression} f\left(s d \log \left(\left(\frac{\phi_0}{\bd 1}\right)^{-1}s^\k\right)\right)\end{equation}
 of $\omega$ on $\ca C_p^{sm}$ extends over the node. In other words, for each branch of the node we need to show that this differential has at most a simple pole at the node and that the residues at both sides of the node add to zero.
 
 For this, we make the following observation about $f$: working \'etale-locally we may assume a neighbourhood of the node is given by the spectrum of $R = \field[x,y]/(xy)$, so that $\Omega_{R/\field} = \frac{R\langle dx, dy \rangle}{xdy + ydx}$. Since $f$ is $\ca O_{\ca C_p}$-linear we have $xf(dy) + yf(dx) = 0$, and hence $f(dx)$ (resp. $f(dy)$) is divisible by $x$ (resp. by $y$). Thus, on both branches of the node, we can regard $f$ as a tangent field vanishing to order at least $1$ at the node.
 
 Then, as before, the term $d \log((\phi_0/\bd 1)^{-1} s^k)$ has at most a pole of order $1$, cancelling with the zero of $f$.
 So in fact the differential \ref{eqn:betaexpression} is regular at the nodes (it does not even have simple poles) and in particular the residues vanish on both sides and thus add to zero.
\end{proof}

 \begin{lemma} \label{Lem:bOmegavee}
We have 
\begin{equation*}
b_\Omega^\vee|_{\bigoplus_v H^0(C_v,\omega_{C_v})} = \beta  \colon \bigoplus_v H^0(C_v,\omega_{C_v}) \to H^0(\ca C_p,\ca Hom(\Omega^\vee(-P), \omega)). 
\end{equation*}
 \end{lemma}
 \begin{proof}
 Given $s \in H^0(C_v,\omega_{C_v})\sub H^0(\ca C_p, \omega)$ and $g=(g_{ij})_{ij} \in \cecH^1(\ca C_p, \Omega^\vee(-P))$ we want to show 
\[\langle s, b_\Omega(g)\rangle = \langle \beta(s),g \rangle \in H^1(\ca C_p, \omega_{C_p}) = \field\]
where $\langle -, -\rangle$ denotes the respective Serre-Duality pairings on both sides. 
Inserting the formulas above, we obtain
\begin{align*}
 &\langle s, b_\Omega(g)\rangle - \langle \beta(s), g \rangle\\
 =&\left(-s \cdot g_{ij} d \log(g_{ij}^k (\frac{\phi_0}{\bd 1})) - s d\log(s^k (\frac{\phi_0}{\bd 1})^{-1}) \cdot g_{ij} \right)_{ij}\\
 =&\left(-sg_{ij} \cdot d \log(g_{ij}^k s^k)\right)_{ij} = \left( - k d (sg_{ij}) \right)_{ij}.
\end{align*}
Thus difference of the two sides of the equality has the form $d \eta$ for the element $\eta = (- k s g_{ij})_{ij} \in \cecH^1(\ca C_p, \mathcal{O}_{\ca C_p})$ and thus is zero. Here we use again that on the overlaps of our cover, we can pair the differential $s$ and the sections $g_{ij}$ of $\Omega^\vee(-P)$ to obtain a local section of $\ca O_{\ca C_p}$. The computation above is taken from Mondello's unpublished note \cite{MondelloPluricanonical-}.
 \end{proof}

Recall from \ref{sec:AJ_on_T_p}, that on the components of $\ca C_p$ corresponding to outlying vertices, we could choose $\k$th roots $(\frac{\phi_0}{\bd 1_v})^{1/\k}$ of the twisted differential, all unique up to scaling.

\begin{lemma}\label{lem:kernel_description}
The kernel of $b_\Omega^\vee \oplus b_\Gamma^\vee$ is given by the map
\begin{equation}\label{eq:kernel_map}
 \bigoplus_{v \in V^{out}} \field \to H^0(\ca C_p,  \omega)
\end{equation}
sending $(c_v)_v$ to the section given by $0$ on the smooth locus of the component $C_{v_0}$ of the central vertex $v_0$, and $c_v (\frac{\phi_0}{\bd 1_v})^{1/\k}$ on the smooth locus of the component $C_v$ for the outlying vertices $v$. 
\end{lemma}
Note that while the sections $(\frac{\phi_0}{\bd 1_v})^{1/\k}$ are only unique up to scaling, the image of the map \ref{eq:kernel_map} is independent of these choices.
\begin{proof}
Since every element of the kernel of $b_\Omega^\vee \oplus b_\Gamma^\vee$ is in particular in the kernel $\oplus_v H^0(C_v,\omega_{C_v})$ of $b_\Gamma^\vee$, it suffices to know the description of $b_\Omega^\vee$ on this subspace.

On the smooth locus we have $$\ca Hom(\Omega^\vee(-P), \omega) = \omega(P),$$ and on this locus $b_\Omega^\vee(s) = s d \log(s^\k (\phi_0/\bd 1)^{-1})$. This vanishes iff $s^\k (\phi_0/\bd 1)^{-1}$ is a locally constant function. In other words, up to scaling, $s$ should be a $\k$th root of $\phi_0/\bd 1$. 

On the central vertex, this implies $s|_{C_{v_0}}=0$. Indeed, by assumption  there is a marking $i$ with $m_i<0$ or $m_i$ not divisible by $\k$. By the discussion in \ref{sec:FP_formula} this marking must be on the central vertex $v_0$, hence there cannot be a holomorphic differential $s$ with $\k$th power $\phi_0/\bd 1$. On the other hand, on the outlying vertices there do exist the sections $(\frac{\phi_0}{\bd 1_v})^{1/\k}$, so the kernel of $b_{\Omega}^\vee$ is exactly given by the map \ref{eq:kernel_map}.
\end{proof}

\subsection{Computing the residue pairing with an element of the kernel of \texorpdfstring{$b_\Omega^\vee \oplus b_\Gamma^\vee$}{b\_Omega \textasciicircum v + b\_Gamma\textasciicircum v} }

In \ref{sec:b_L} and in \ref{lem:lift_after_translate} it will be important to compute the value of the residue pairing between an element of the kernel of $b_\Omega^\vee \oplus b_\Gamma^\vee$ and a particular element $\delta$ of $H^1(C, \omega^\k(-\m P)\otimes \ca T)$. To avoid duplication we will carry out this computation here, in sufficient generality for both our applications. Let 
\begin{equation*}
0 \to J \to A' \to A \to 0
\end{equation*}
be a short exact sequence of $\field$-modules, where $A$ and $A'$ have the structure of (Artin local) $\field$-algebras, the map $A' \to A$ is a $\field$-algebra homomorphism, and $J\frak m_{A'} = 0$ where $\frak m_{A'}$ is the maximal ideal of $A'$. 

Suppose we are given a stable marked curve $C_{A'}$ over $A'$, whose fibre over $\frak m_{A'}$ is our curve $\ca C_p$. 

We consider a formally \'etale fpqc cover of $C_{A'}$
consisting of charts and gluing morphisms as follows:
\begin{enumerate}
\item
For each irreducible component $C_{A',v}$ of $C_{A'}$ (corresponding to a vertex $v$ of $\Gamma$), the smooth locus $U_v \coloneqq C_{A',v}^{sm}$ is a chart. Since $U_v$ is smooth and affine\footnote{Since $p$ is a point of a boundary component of $\DRL^{1/\k}$, the curve $\ca C_p$ is not smooth, hence these charts are indeed affine}, it is (non-canonically) a trivial deformation of its central fibre $U_{v,\field}=U_{v} \times_{A'} \field$, so we have $U_v = U_{v,\field} \times_\field A'$. Note that in particular, $U_v$ is the complement of finitely many $A'$ points of the smooth projective curve $C_v \times_\field A'$ over $A'$.
\item for each node (corresponding to an edge $e$ of $\Gamma$), a formally \'etale neighbourhood of the form $U_e \coloneqq \on{Spec}A'[[z_h, z_{h'}]]/(z_hz_{h'} - \ell_e)$. 
\item for each half-edge $h$ belonging to an edge $e=(h,h')$ with $h$ incident to a vertex $v$, the choice of a formal neighbourhood $\bb D = \on{Spec} A'[[t]][t^{-1}]$ of the preimage of the corresponding node in (the normalization of) $C_{A',v}$ and an open immersion $\bb D \to U_e$ given by the ring map
\[A'[[z_h, z_{h'}]]/(z_hz_{h'}-\ell_e) \to A'[[t]][t^{-1}]; z_h \mapsto t, z_{h'} \mapsto \frac{\ell_e}{t}\]
describing how $\bb D$ is glued into the singular chart.
\end{enumerate}
Note the following important assumption about the smoothing parameters $\ell_e$:
\begin{center}
 \textbf{We assume throughout this section that $\ell_e^{I'(e)} \in J$. }
\end{center}
This will be the case in our applications and in general for first order deformations, i.e. those over $A'=\field[t]/(t^2)$. Indeed, it is always true that $\ell_e \in J=t\field[t]/(t^2)$ since otherwise the chart $U_e$ is not nodal.

Now assume $C_A$ lies in the double ramification locus, so we can choose a generating section $\phi_0$ for the line bundle $\omega^\k_{C_A}(-\m P)\otimes \ca T$. For each of the charts $U_e$ we choose a lift $\phi_e$ of $\phi_0$, and for each $U_v$ a lift $\phi_v$. The differences $\phi_* - \phi_{*'}$ on the overlaps restrict to $0$ over $A$, hence give an element $\delta'$ of $H^1(\ca C_p, \omega^\k_{\ca C_p}(-\m P)\otimes \ca T)\otimes_\field J$ (see \ref{sec:explicit_def_via_cech}). Using the isomorphism $\omega^\k_{\ca C_p}(-\m P)\otimes \ca T \cong \ca O_{\ca C_p}$ via $\phi_0|_{\ca C_p}$, we can convert this into an element $\delta \in H^1(\ca C_p, \ca O_{\ca C_p})$

Now let $\bd c = (c_v (\frac{\phi_0}{\bd 1_v})^{1/\k})_v$ be an element of the kernel of $b_\Omega^\vee \oplus b_\Gamma^\vee$ (recall that we fixed the $\k$-th root of $\phi_0/\bd 1_v$ on the outlying vertices $v$). Then the product $\delta \bd c$ lies in $H^1(\ca C_p, \omega_{C_p})\otimes_\field J$, which is isomorphic via the residue pairing to $J$. In this section we make the image of $\delta \bd c$ in $J$ explicit. 

Recall from \ref{rem:q_root_T} that we have a canonically-defined root $\ca T^{1/\k}$ of $\ca T$, and on the smooth charts $U_v$ a generating section $\bd 1_v^{1/\k}$ of $\ca T^{1/\k}$ which is a $\k$th root of $\bd 1_v$. Multiplying with the section $(\phi_0/\bd 1_v)^{1/\k}$, we obtain a well-defined $\k$-th root of $\phi_0$ on the outlying components, which extends to the nodes. Now, we should not expect this to extend to a root over the whole of the central component. However, since $\phi_0$ is a generating section of the line bundle $\omega^\k(-\m P)\otimes \ca T$, this root will extend uniquely to a $\k$-th root $\phi_0^{1/\k}$ of $\phi_0$ on a small \'etale neighbourhood of each node in the central component. Here $\phi_0^{1/\k}$ is a local section of $\omega\otimes \ca T^{1/\k}$ around the nodes\footnote{Here we use that the marked points $P$ are disjoint from the small \'etale neighbourhood of the nodes}. 

Now the restriction $\phi_0|_{C_{v_0}}/\bd 1_{v_0}$ to the central component extends to an element of
$$H^0(C_{v_0}, \omega^\k_{C_{v_0}}(-m_0P + \sum_{e} (I(e) + \k)q_e)), $$ 
where $q_e$ is the point on $C_{v_0}$ corresponding to the edge $e$. Thus when taking the $\k$th root defined locally around the nodes we see that $\phi_0^{1/\k}|_{C_{v_0}}/\bd 1_{v_0}^{1/\k}$ yields 
a local section of
\begin{equation*}
\omega_{C_{v_0}}(\sum_{e} (I'(e) + 1)q_e). 
\end{equation*}


\begin{lemma}\label{lem:res_pairing}
The image of $\delta \bd c$ in $J$ under the residue pairing is given by 
\begin{equation}\label{eq:big_res_comp}
-\sum_{v \in V^{out}} c_v \sum_{e:v_0 \to v}a_{\gamma(e)}\ell_e^{I'(e)}\on{Res}_{q_e}(\phi_0^{1/\k}|_{C_{v_0}}/\bd 1_{v_0}^{1/\k}) \in J. 
\end{equation}
where $\gamma(e)$ is the cycle in $\Gamma$ going out along the distinguished edge $e_v : v_0 \to v$ and back along $e$ (in particular $a_{\gamma(e_v)} = 1$).  \end{lemma}
Some remarks on the statement, before giving the proof: 
\begin{itemize}
\item
This root $\phi_0^{1/\k}$ (defined above) only makes sense locally around nodes, but this is all we need for \ref{eq:big_res_comp} to make sense. 
\item
We emphasise that in \ref{eq:big_res_comp}, the residue is taken on the central component $C_{v_0}$ of $C$.

\item Recall that we assume throughout this section that $\ell_e^{I'(e)} \in J$. 
\end{itemize}
\begin{proof}
We compute the residue one-point-at-a-time, as in \ref{sec:serre_duality}. 

\textbf{Case 1: smooth points. }
On the smooth points of $C_{v_0}$ for the central vertex $v_0$, the section coming from $\bd c$ vanishes. On the outlying vertices $v$, the sections $\phi_v$ extending $\phi_0$ do not have a pole. Hence in both cases the residue vanishes. 

\textbf{Case 2: nodes. }
Here there is only one possible choice of patch. But to compute the residue we have to sum the residues coming from the two preimages of our point under the normalisation map $\pi\colon \tilde{\ca C_p} \to \ca C_p$. 

\textbf{Case 2.1: lift to the central vertex. }
Here again the section coming from $\bd c$ is zero, hence the residue vanishes. 

\textbf{Case 2.2: lift to an outlying vertex. }
Let $q$ be the chosen point on the outlying component $C_v$ mapping to a node. 
Our strategy will be as follows: as described above, the overlap of the singular chart and the smooth chart on the outlying vertex is given by $\bb D = \on{Spec} A'[[t]][t^{-1}]$, sitting inside the singular chart $U_e=\on{Spec} R'$ via the ring map 
\begin{equation*}
R'=A'[[z_h, z_{h'}]]/(z_hz_{h'}-\ell_e) \to A'[[t]][t^{-1}]; z_h \mapsto t, z_{h'} \mapsto \frac{\ell_e}{t}, 
\end{equation*}
where the branch $z_h=0$ corresponds to the \emph{central} vertex. On the other hand, since $C_v$ is smooth at $q$, we can take the inclusion $\bb D \to U_v$ to be the product of a small punctured formally \'etale neighbourhood $\on{Spec} \field[[t]][t^{-1}] \to C_v$ of $q \in C_v$ with $A'$ over $\field$. 

We will compute the difference $\phi_e - \phi_v$ on $\bb D$ and, by the deformation theory in the appendix, this gives an element 
$$\phi_e - \phi_v \in H^0(\bb D_\field, \omega_{\bb D_\field}^\k \otimes \ca T|_{\bb D_\field}) \otimes_\field J.$$
Recall that the collection of such differences exactly describes a cocycle representing the element $\delta'$ of $H^1(C_p, \omega^\k_{\ca C_p}(-\m P)\otimes \ca T)\otimes_\field J$. We obtain $\delta \in H^1(\ca C_p, \ca O_{\ca C_p})\otimes_\field J$ by dividing by $\phi_0|_{\ca C_p}$. 

To compute the pairing with $\bd c$, we now multiply $(\phi_e - \phi_v)/(\phi_0|_{\ca C_p})$ with the restriction of $c_v (\phi_0/\bd 1_v)^{1/\k} $ to $\bb D_\field$ and obtain
\[\frac{\phi_e - \phi_v}{\phi_0} \cdot c_v (\frac{\phi_0}{\bd 1_v})^{1/\k} = c_v (\phi_0)^{-1+1/\k} (\phi_e -  \phi_v)/ \bd 1_v^{1/\k}.\]
Then we can take the residue of this at $t=0$ and obtain the contribution to the pairing from the node $q$.

Now instead of abstractly using that $\phi_e - \phi_v \in H^0(\bb D_\field, \omega_{\bb D_\field}^\k \otimes \ca T|_{\bb D_\field}) \otimes_\field J$ and multiplying with the section $c_v (\phi_0)^{-1+1/\k} \bd 1_v^{-1/\k}$ defined over $\field$, we can instead use \emph{any} section $\rho \in H^0(\bb D, \omega^{-\k+1} \otimes \ca T^{-1+1/\k})$ such that $\rho$ restricts to $c_v (\phi_0)^{-1+1/\k}$ over $\field$ and compute the residues of $\rho (\phi_e - \phi_v) / \bd 1_v^{1/\k}$ on $\bb D$.

We now propose a particular choice of $\rho$: by \ref{Lem:liftqthroots}  below there is a unique
section $\phi_e^{1/\k}$ of $\omega \otimes \ca T^{1/\k}$ on $U_e$ such that $(\phi_e^{1/\k})^\k = \phi_e$ and such that $\phi_e^{1/\k}|_A = \phi_0^{1/\k}$. We choose the section $\rho=c_v (\phi_e^{1/\k})^{-\k+1}$, which obviously has the properties mentioned above. Thus the value of the pairing at $q$ is given by
\begin{equation} \label{eqn:residueterm}
c_v(\on{res}_{t=0} \left((\phi_e^{1/\k})^{-\k+1} (\phi_e -  \phi_v)/ \bd 1_v^{1/\k}\right)
\end{equation}
(lying in $J \subset A'$) and we will see that our choice of $\rho$ allows us to compute the residues from the terms involving $\phi_e$ and $\phi_v$ separately. 


For the first term we observe that $(\phi_e^{1/\k})^{-\k+1} \phi_e = \phi_e^{1/\k}$ is (the restriction to $\bb D$ of) a section of $\omega\otimes \ca T^{1/\k}$ on $U_e$. Thus on $U_e$ it has a representation 
\begin{equation*}
 \phi_e^{1/\k} = \tilde \phi_e(z_h, z_h') (\frac{d z_h}{z_h}) \otimes \tau_e^{1/\k}
\end{equation*}
with $\tilde \phi_e(z_h,z_h') \in R'=A'[[z_h, z_h']]/(z_h z_h'-l_e)$.
Pulling this back to $\bb D$ (and using the gluing maps of the sections of $\ca T^{1/\k}$) gives us
\begin{align*}
 \phi_e^{1/\k} &= \tilde \phi_e(t, \frac{l_e}{t}) (\frac{d t}{t}) \otimes \tau^{1/\k} \\
 &=\lambda(e) t^{I'(e)} \tilde \phi_e(t, \frac{l_e}{t}) (\frac{d t}{t}) \otimes \bd 1_v^{1/\k}.
\end{align*}
Dividing by $\bd 1_v^{1/\k}$ and Taylor expanding yields 
\begin{equation*}
a_{\gamma(e)}{t}^{I'(e)-1}\left[   \sum_{i \ge 0}\frac{1}{i!}\left(\frac{\ell_e}{{t}}\right)^i\frac{\partial^i \tilde \phi_e({z_h},{z_{h'}})}{\partial {z_{h'}}^i}(t,0) \right] d{t}. 
\end{equation*}
Here we use that $\lambda(e)=a_{\gamma(e)}$ is a valid choice according to the construction presented in \ref{sec:universal_section} (for the spanning tree $\Gamma' \subset \Gamma$ we choose the tree formed by the distinguished edges $e_v$ mentioned in \ref{lem:res_pairing}).

To compute the residue at $t=0$ we look at the terms whose order in $t$ is exactly $-1$. Since $\tilde \phi_e$ is a power series in $z_h, z_h'$, this forces $i \geq I'(e)$. On the other hand, by our assumption $l_e^{I'(e)} \in J$ and also $l_e \in  \frak m_{A'}$, so for $i>I'(e)$ we have $l_e^i = l_e^{i-I'(e)} l_e^{I'(e)} = 0$ since $\frak m_{A'} \cdot J = 0$. 
Thus all terms for $i>I'(e)$ vanish and hence the only term of the above sum making a possibly-non-zero contribution to the residue occurs when $i = I'(e)$, and the residue is given by 
\begin{equation*}
\frac{a_{\gamma(e)}}{I'(e)!}\ell_{e}^{I'(e)}\frac{\partial^{I'(e)}\tilde \phi_e({z_h},{z_{h'}})}{\partial {z_{h'}}^{I'(e)}}(0, 0). 
\end{equation*}

This finishes the computation of the residue for the term of \ref{eqn:residueterm} involving $\phi_e$. Now we want to argue that the sum (over the nodes $q$ connecting $C_v$ to $C_{v_0}$) of the residues of the terms $(\phi_e^{1/\k})^{-\k+1}  \phi_v/ \bd 1_v^{1/\k}$ in \ref{eqn:residueterm} vanishes. 

For this, we now choose some splitting of the short exact sequence $0 \to J \to A' \to A \to 0$ of $\field$-vector spaces, allowing us to write $A'=A \oplus J$ in some non-canonical way. Then on $\bb D$ we can write
\begin{equation*}
 (\phi_e^{1/\k})^{-\k+1}  \phi_v/ \bd 1_v^{1/\k} = (\phi_0^{-1+1/\k} + R_1)\cdot  (\phi_0 + R_2)/ \bd 1_v^{1/\k},
\end{equation*}
where as before $\phi_0^{-1+1/\k}$ and $\phi_0$ are sections of the base-changes of $\omega^{-\k+1}\otimes \ca T^{-1+1/\k}$ and $\omega^\k \otimes \ca T$ on $\bb D$ to $A$ (i.e. $\phi_0^{-1+1/\k} \in A[[t]][t^{-1}] \cdot (dt)^{-\k+1} \otimes \bd 1_v^{-1+1/\k}$ etc.), and $R_1, R_2$ are corresponding sections which vanish modulo $J$ (i.e. there exist representatives with coefficients in $J$). We note here that since $U_v = U_{v,\field} \times_\field A'$, and since the map $\bb D \to U_v$ is the pullback of $\bb D_\field \to U_{v,\field}$ to $A'$, the section $R_2$ is indeed the restriction of a section $\widehat{R}_2 \in H^0(U_v, \omega^\k(-\m P)\otimes \ca T)$ on $U_v$ (which is just the $J$-part of the section $\phi_v$ on $U_v$). 

%
%

Using that $J^2=0$ we can write out the product above and obtain three terms
\begin{equation*}
 (\phi_e^{1/\k})^{-\k+1}  \phi_v/ \bd 1_v^{1/\k} = \phi_0^{1/\k}/ \bd 1_v^{1/\k} + \phi_0^{-1+1/\k} R_2 / \bd 1_v^{1/\k} + \phi_0 R_1/ \bd 1_v^{1/\k}.
\end{equation*}
By assumption, the first term $\phi_0^{1/\k}/ \bd 1_v^{1/\k}$ has order $I'(e)-1 \geq 0$ at $t=0$, so it does not contribute a residue. On the other hand, for the third term we observe that $R_1$ is the restriction to $\bb D$ of a section $\widehat{R}_1$ on $U_e$ killed by $J$ (the `$J$-valued' part of $\phi_e^{-1+1/\k}$). If we write
\[\widehat{R}_1 = \psi(z_h, z_{h'}) \cdot \frac{(dz_h)^{-\k+1}}{z_h^{-\k+1}}\otimes \tau_e^{-1+1/\k}\]
for some $\psi(z_h,z_{h'}) \in J \cdot A'[[z_h, z_{h'}]]$ then restricting to $\bb D$ we have
\[R_1 = \widehat{R}_1|_{\bb D} = \psi(t, \frac{l_e}{t}) \cdot \frac{(dt)^{-\k+1}}{t^{-\k+1}}\otimes (\lambda(e)^{-\k+1} t^{-I'(h)\k+I'(h)} \bd 1_v^{-1+1/\k}),\]
but $\psi(t, \frac{l_e}{t})=\psi(t,0)$ since $\psi$ has coefficients in $J$ and since $l_e \in \frak m_{A'}$ so that $l_e \cdot J=0$. Looking at $\phi_0$ we know
\[\phi_0 = \zeta(t) \cdot t^{I(h)-\k} (dt)^\k \otimes \bd 1_v\]
for some $\zeta(t) \in A[[t]]$.
Combining the two terms (and using $I(h)=I'(h) \cdot \k$) we obtain
\[\phi_0 R_1/ \bd 1_v^{1/\k} = \psi(t,0) \zeta(t) \lambda(e)^{-\k+1} \cdot \underbrace{t^{\k-1-I(h)+I'(h)+I(h)-\k}}_{t^{I'(h)-1}} \cdot dt,\]
so the order of this term at $t=0$ is nonnegative and hence the residue vanishes.

Finally we look at the remaining term $\phi_0^{-1+1/\k} R_2 / \bd 1_v^{1/\k}$. Now using \ref{Lem:liftqthroots} we can find a section $\phi_v^{1/\k}$ of $\omega(-(\m/\k)P)\otimes \ca T^{1/\k}$ on $U_v$ such that $(\phi_v^{1/\k})^\k=\phi_v$ and such that $(\phi_v^{1/\k})|_{U_{e} \times_{A'} A} = \phi_0^{1/\k}$. Then the section $\phi_v^{-1+1/\k}:=(\phi_v^{1/\k})^{-\k+1}$ on $U_v$ has the same restriction to $\bb D_A$ as $\phi_0^{-1+1/\k}$. Since $R_2$ is killed by $J$, this means that $\phi_0^{-1+1/\k} R_2 / \bd 1_v^{1/\k} = \phi_v^{-1+1/\k}|_{\bb D} R_2 / \bd 1_v^{1/\k}$. Using that $R_2$ is the restriction of the $\k$-differential $\widehat{R}_2$, we see that $\phi_0^{-1+1/\k} R_2 / \bd 1_v^{1/\k}$ is the restriction of the holomorphic differential $\phi_v^{-1+1/\k} \widehat{R}_2$ on $U_v$. Since $U_v \subset C_v \times_\field A'$ is the complement of finitely many $A'$-points, we can see this as a meromorphic differential of the smooth projective curve $C_v \times_\field A'$ over $A'$. The residue of $\phi_0^{-1+1/\k} R_2 / \bd 1_v^{1/\k}$ at $t=0$ is just the residue of this differential at the $A'$-point corresponding to the node $q$. Then, by the Residue theorem for curves over Artin rings (\cite[Appendix B.1., page 272]{Conrad2000Grothendieck-du}), the sum of all residues at nodes $q$ connecting $C_v$ to $C_{v_0}$ is zero.
Summarising the above, we see that the image of $\delta \bd c$ in $J$ under the residue pairing is given by
\begin{equation*}
\sum_{v \in V^{out}} c_v \sum_{e=(h,h'):v_0 \to v} \frac{a_{\gamma(e)}}{I'(e)!}\ell_{e}^{I'(e)}\frac{\partial^{I'(e)}\tilde \phi_e({z_h},{z_{h'}})}{\partial {z_{h'}}^{I'(e)}}(0, 0),
\end{equation*}
where $z_{h'}$ is a local coordinate on the central vertex for the node. 

Since $l_e^{I'(e)} \in J$ and $\frak m_{A'} J =0$ we see that in the formula we can replace $\tilde \phi_e$ by its restriction $\tilde \phi_0$ to $\field$. Looking back at the definition of $\tilde \phi_e$, its restriction $\tilde \phi_0$ over $\field$ satisfies that on $U_{e} \times_{A'} \field$ we have 
\begin{align*}
 \phi_0^{1/\k} &= \tilde \phi_0(z_h, z_{h'}) \frac{dz_h}{z_h} \otimes \tau_e^{1/\k}\\
 &= \tilde \phi_0(z_h, z_{h'}) (- \frac{dz_{h'}}{z_{h'}}) \otimes (z_{h'}^{-I'(e)}\bd 1_{v_0}^{1/\k}).
\end{align*}
%
From this it is clear that the expression
\begin{equation}\label{eq:partial_experssion}
-\frac{1}{I'(e)!}\frac{\partial^{I'(e)}\tilde \phi_0({z_h},{z_{h'}})}{\partial {z_{h'}}^{I'(e)}}(0, 0)
\end{equation}
is the residue of the differential $\phi_0^{1/\k}|_{C_{v_0}}/\bd 1_{v_0}^{1/\k}$ at the point $q$ on the smooth `central' curve $C_{v_0}$. 
\end{proof}

\begin{lemma} \label{Lem:liftqthroots}
Let $R'$ be an $A'$-algebra, and $M'$ a locally-free $R'$-module of rank 1. Let $m' \in M'^{\otimes \k}$ and $l \in M =  M' \otimes_{A'} A = M'/JM'$ such that $l^{\otimes \k} = m'  + J M'^{\otimes \k}$ in $M^{\otimes \k}$. Assume also that the section $m'$ is generating. Then there exists a unique $l' \in M'$ such that $l'^{\otimes \k} = m'$ in $M'^{\otimes \k}$ and $l' +JM' = l \in M$. 
\end{lemma}
In this lemma our assumption of characteristic zero (or more precisely that $\k$ is invertible on $A$) is essential. 
\begin{proof}
Let us first discuss uniqueness. The condition $l' +JM' = l$ means $l'$ is unique up to an element of $J M'$. For a different $l''=l'+j \tilde m$ we have $(l'')^{\otimes \k} = l'^{\otimes \k} + \k j (l')^{\otimes \k-1} \tilde m$ (here we use $M'$ being locally free so we can commute tensor products). The fact that $m'$ is generating implies that $l'$ is generating. Thus $(l'')^{\otimes \k} = m' = l'^{\otimes \k}$ is only possible for $\k j \tilde m=0$. Since $\k$ is invertible in $A'$, this implies $j \tilde m=0$ hence $l''=l'$.

By the uniqueness part, we can work locally and so assume that $M'$ is free, so take $M' = R'$, $M = R$ (and identify $M^{\otimes \k} = R$ etc also). 
Choose any lift $\tilde l$ of $l$ to $R'$, and define $\epsilon = \tilde l^{\k} - m' \in JR'$. Now $l$ generates $R$ since $m'$ generates $R'$, and $\k$ is invertible, so there is a unique $j \in J$ such that $\k l^{\k -1} j = \epsilon$ (recall $J \frak m_{A'} = 0$). Then
\begin{equation*}
\begin{split}
(\tilde l + j )^\k & = \tilde l ^\k + \k \tilde l ^{\k -1} j \;\;\;  (\text{since} \; J^2 = 0)\\
& = \tilde l ^\k + \k l ^{\k -1} j  \;\;\; (\text{since} \; J \frak m_{A'} = 0)\\
& = m'. \qedhere
\end{split}
\end{equation*}
\end{proof}

\subsection{Intersecting with the kernels of \texorpdfstring{$b_{L_v}^\vee$}{b\_\{L\_v\}\textasciicircum v} and \texorpdfstring{$b_{>1}^\vee$}{b\_\{>1\}\textasciicircum v} }\label{sec:b_L}
Recall from \ref{sec:AJ_on_T_p} the maps
\begin{equation*}
 b_{>1}\colon  \bigoplus_{e:I'(e)>1} \field  \to H^1(\ca C_p, \ca O_{\ca C_p})
\end{equation*}
and 
\begin{equation*}
b_{L_v}\colon L_v \to H^1(\ca C_p, \ca O_{\ca C_p}).
\end{equation*}
The goal of this section is to prove the following lemma, which describes exactly the intersection of the kernels of $b_\Omega^\vee$, $b_\Gamma^\vee$, $b_{>1}^\vee$ and $b_{L_v}^\vee$ (c.f. \ref{lem:kernel_intersection}). It immediately allows us to conclude the proof of \ref{thm:ker_b_v}.

\begin{lemma}\label{lem:kernel_formula}
Let $\ca C_p$ be a general element of a boundary component of $\DRL^{1/\k}\subset \Mdk$, with dual graph $\Gamma$. Recall that $V^1$ is the set of outlying vertices $v$ such that $I'(e)=1$ for each $e:v_0 \to v$. Let $\bd c = (c_v (\frac{\phi_0}{\bd 1_v})^{1/\k})_v$ be an element of the kernel of $b_\Omega^\vee \oplus b_\Gamma^\vee$. Then $\bd c$ maps to zero under the maps $b_{>1}^\vee$ and $b_{L_v}^\vee$ if and only if $c_v=0$ for all $v \in V^1$. 
\end{lemma}
\begin{proof}[Proof of \ref{thm:ker_b_v}]
 The elements $\bd c$ of the kernel of $b_\Omega^\vee \oplus b_\Gamma^\vee$ with $c_v=0$ for $v \in V^1$ are exactly those parametrized by the $c_v$ for $v \in V^{>1}$ appearing in the map from \ref{thm:ker_b_v}.
\end{proof}
\begin{proof}[Proof of \ref{lem:kernel_formula}]
By definition, the element $\bd c$ is in the kernel of $b_{>1}^\vee$ and $b_{L_v}^\vee$ iff it pairs to zero with the image of each element
\[a = (a_e)_e \in \bigoplus_{e:I'(e)>1} \field \oplus \bigoplus_v L_v\]
under the map 
\[b_{>1} \oplus \bigoplus_v b_{L_v} : \bigoplus_{e:I'(e)>1} \field \oplus \bigoplus_v L_v \to  H^1(\ca C_p, \ca O_{\ca C_p}).\]
But recall that $b_{>1}$ and $b_{L_v}$ come from tangent maps of the Abel-Jacobi section for first-order deformations of $\ca C_p$ locally smoothing various nodes of $\ca C_p$. This is exactly the situation analysed in \ref{lem:res_pairing}. To be more precise, we apply \ref{lem:res_pairing} for $A = \field$, $A' = \field[t]/t^2$, and $J = t\field$. Then from the definition of $b_{>1}^\vee$ and $b_{L_v}^\vee$ we see that the local smoothing parameters $\ell_e$ are exactly $\ell_e = a_e t \in J$ and we have 
\[\delta = (b_{>1} \oplus \bigoplus_v b_{L_v})(a) \otimes t \in H^1(\ca C_p, \ca O_{\ca C_p}) \otimes J.\]
We check that the pairing of $\bd c$ with $\delta$ vanishes for all choices of $a$ iff $c_v=0$ for all $v \in V^1$.
%
%
%
%
%

First, we claim that $b_{>1}^\vee$ vanishes on the kernel of $b_\Omega^\vee \oplus b_\Gamma^\vee$. Indeed, in the formula \ref{eq:big_res_comp} from \ref{lem:res_pairing}, all terms $\ell_e^{I'(e)} = a_e^{I'(e)} t^{I'(e)}$ vanish for $I'(e)>1$ since $t^2=0$. These are exactly the contribution of the direct summand $\bigoplus_{e:I'(e)>1} \field$ above on which $b_{>1}$ is defined.

Thus it remains to compute the pairing with the image of $b_{L_v}$. Let $v \in V^1$ be a vertex, so $I'(e)=1$ for each edge $e:v_0 \to v$. Let $\bd a_v = (a_e)_{e:v_0 \to v} \in L_v$, then from \ref{lem:res_pairing} the value of the pairing of $\bd c$ with $b_{L_v}(\bd a_v)$ is given by
\begin{equation} \label{eqn:bLvpairing}
- c_v \sum_{e:v_0 \to v}a_{\gamma(e)}a_et\on{Res}_{q_e}(\phi_0^{1/\k}|_{C_{v_0}}/\bd 1_{v_0}^{1/\k}) \in J. 
\end{equation}
Evidently this vanishes when $c_v=0$. We claim that for $\bd a_v \neq 0$ (and recalling that $p$ was generically chosen in the boundary component), the sum appearing in the formula is nonzero. Thus $\bd c$ pairs to zero with the image of $b_{L_v}$ if and only if $c_v=0$. 

To show that the sum in \ref{eqn:bLvpairing} is nonzero, recall that we fixed an edge $e_v : v_0 \to v$ before. For any edge $e$ from $v_0$ to $v$ we have by definition of $L_v$ that 
\begin{equation*}
a_{\gamma(e)}a_{e}=a_{e_v}. 
\end{equation*}
Hence to prove that the element \ref{eqn:bLvpairing} is nonzero for $c_v, a_{e_v} \neq 0$, we need to show that 
\begin{equation} \label{eqn:finalressum}
\sum_{e:v_0 \to v}\on{Res}_{q_e}(\phi_0^{1/\k}|_{C_{v_0}}/\bd 1_{v_0}^{1/\k}) \neq 0.
\end{equation}
But the summands above are just some choice of $\k$th roots of $\k$-residues\footnote{This terminology is recalled in \ref{sec:generic_non-vanishing_residue}.} for the meromorphic $\k$-differential $\phi_0|_{C_{v_0}}/\bd 1_{v_0}$ on $C_{v_0}$. Since $p$ was generically chosen in its component, we have that $C_{v_0}$ is generic in a suitable stratum of meromorphic $\k$-differentials. Then the non-vanishing of \ref{eqn:finalressum} follows from \ref{thm:non_vanishing_root_sum}. 
\end{proof}

\section{The length of the double ramification cycle}\label{sec:computing_length}

In this section we will compute the length of the Artin local ring obtained by localising $\DRL^{1/\k}$ at the generic point of the component. In \ref{lem:length_equals_multiplicity} we checked that this coincides with the multiplicity of the cycle $\DRC^{1/\k}$ at the corresponding point. 

To compute this length we first choose a combinatorial chart $\Mbar \longleftarrow U \longrightarrow \mathbb{A}^E = \on{Spec}\field[\ell_e:e \in E]$ inducing a chart $\Mdk_{I,U} \to \Mdk$ containing the general point $p$ of the component of $\DRL^{1/\k} \subset \Mdk$ we want to consider. We can assume that $U$ is affine, say $U \subset \bb A^M$ for some $M$. Recall that in this situation we have that $\Mdk_{I,U} \subset U \times \bb A^\Upsilon \subset \bb A^M \times \bb A^\Upsilon = \bb A^N$. For a generic linear subspace $H$ in $\bb A^N$ through our chosen point $p \in \DRL^{1/\k}$, where $H$ has codimension $2g-3+n$, we denote $\DRL'=(\DRL^{1/\k} \cap H)_p$ the intersection of $\DRL^{1/\k}$ with $H$, localized at our point $p$. Since $\dim \DRL^{1/\k} = 2g-3+n$, this is an Artin-local $\field$ scheme with residue field $\field$, and its length is exactly the length of $\DRL^{1/\k}$ along the component containing $p$, assuming $p$ sufficiently generically chosen in that component.

The chosen combinatorial chart induces a map from $\Mdk_{I,U}$ to $\bb A^E$ sending $p$ to the origin; composing with the inclusion of $\DRL'$ into $\Mdk_{I,U}$ induces a map $\DRL' \to \bb A^E$. Our goal for this section is to prove the following

\begin{theorem}\label{thm:structure_of_DR}
The map $\DRL' \to \bb A^E$ is a closed immersion, with image cut out by the ideal $(\ell_e^{I'(e)}:e \in E)$. 
\end{theorem}
An immediate corollary of this theorem is that the length of $\DRL^{1/\k}$ at $p$ is given by $\prod_e I'(e)$.

We will deduce the theorem from the next lemma, for which we need a little notation. Set $R= \field[\ell_e:e\in E] = \ca O_{\bb A^E}(\bb A^E)$, and let $\frak b \triangleleft R$ be an ideal containing some power of $\frak m \coloneqq (\ell_e:e\in E)$. Let $B = R/\mathfrak b$ (we think of this ideal as a $B$-point of $\bb A^E$). 

\begin{lemma}\label{lem:lifting_B_point}
This $B$-point of $\bb A^E$ lifts to a $B$-point of $\DRL'$ if and only if $\frak b$ contains $\ell_e^{I'(e)}$ for every $e$. 
\end{lemma}
\begin{proof}[Proof of \ref{thm:structure_of_DR} assuming \ref{lem:lifting_B_point}]
Applying \ref{lem:lifting_B_point} with $\frak b =  (\ell_E^{I'(e)}:e \in E)$ we see that the map $\DRL' \to \on{Spec}R/(\ell_E^{I'(e)}:e \in E)$ has a section, so in particular it is surjective on tangent spaces.
Since the tangent spaces have the same dimension by \ref{thm:dim_T_p_DR}, the map is necessarily bijective on tangent spaces.

Now, since $\DRL'$ and $\bb A^E$ are evidently (locally) of finite presentation over $\field$, by  \cite[\href{https://stacks.math.columbia.edu/tag/00UV}{Tag 00UV (8)}]{stacks-project} injectivity of the tangent map implies that  $\DRL' \to \bb A^E$ is unramified. Since being a closed immersion is \'etale-local on the target, we conclude that $\DRL' \to \bb A^E$ is a closed immersion by applying  \cite[\href{https://stacks.math.columbia.edu/tag/00UY}{Tag 00UY}]{stacks-project} together with the fact that the source is Artin local.

%

It is then clear from another application of \ref{lem:lifting_B_point} that the image is cut out exactly by $(\ell_E^{I'(e)}:e \in E)$. 
\end{proof}

We want to apply deformation theory to prove \ref{lem:lifting_B_point}, but the kernel of $B \to \field; \ell_e \mapsto 0$ is not necessarily killed by $\frak m$. So we decompose it into steps. For every integer $r \ge 2$ we have a short exact sequence
\begin{equation*}
0 \to \frac{\frak m^{r-1} +  \frak b}{\frak m^r + \frak b} \to \frac{R}{\frak m^r + \frak b} \to \frac{R}{\frak m^{r-1} + \frak b} \to 0. 
\end{equation*}
Fixing some $r \ge 2$, denote the non-zero terms in the above sequence by $J$, $A'$ and $A$ respectively (remembering the $R$-algebra structures of the latter two), then we have a surjection $A' \to A$ of Artin local $\field$-algebras with residue field $\field$, and the kernel $J$ is killed by the maximal ideal of $A'$. 

Suppose we are given an $A$-point of $\DRL'$. We want to understand when this lifts to an $A'$-point of $\DRL'$ (again, as a map over $\on{Spec} R)$. We say an $R$-algebra $B$ is \emph{$I$-constrained} if for every edge $e$ the element $\ell_e^{I'(e)}$ maps to zero in $B$. Note that $\field = R/\frak m$ is automatically $I$-constrained. 

\begin{lemma}\label{lem:constrained_lifting}
Suppose $A$ is $I$-constrained. Then the given $A$-point of $\DRL'$ lifts to an $A'$-point of $\DRL'$ (over $\on{Spec} R$) if and only if $A'$ is $I$-constrained. 
\end{lemma}

\begin{proof}[Proof of \ref{lem:lifting_B_point} assuming \ref{lem:constrained_lifting}]
By induction on $r$, \ref{lem:constrained_lifting} shows that the given $\field$-point of $\DRL'$ lifts to an $R/\frak b$-point if and only if $R/\frak b$ is $I$-constrained, i.e. all $\ell_e^{I'(e)}$ lie in $\frak b$. 
\end{proof}


We fix now an $I$-constrained $A$-point of $\DRL'$ and a compatible $A'$-point of $\bb A^E$. We write $M(A, A')$ for the set of liftings of this $A'$-point to $\Mdk_{I,U}$ which are still compatible with the given $A$-point of $\DRL'$. By \ref{Lem:Apseudotors}, the set $M(A, A')$ is naturally a pseudotorsor under the group $H^1(\ca C_p, \Omega^\vee(-P)\otimes J) \oplus H^1(\Gamma, J)$; the $H^1(\ca C_p, \Omega^\vee(-P)\otimes J)$ term parametrises lifts from $\bb A^E$ to $U$, and $H^1(\Gamma, J)$ parametrises lifts along the map $\Mdk_{I,U} \to U$ (cf. \ref{sec:TpMd}).


\begin{lemma}\label{lem:MAA_not_empty}
If $A'$ is $I$-constrained then the pseudotorsor $M(A, A')$ is a torsor (i.e. is nonempty). 
\end{lemma}
\begin{proof}
Firstly, since by assumption the map $U \to \bb A^E$ from our combinatorial chart is smooth, we can always lift the $A'$-point of $\bb A^E$ to an $A'$-point of $U$ compatible with the $A$-point of $U$ induced by $\on{Spec} A \to \DRL'$ (\cite[\href{https://stacks.math.columbia.edu/tag/02H6}{Tag 02H6}]{stacks-project}).
 In the solid diagram
\begin{equation}
\begin{tikzcd}
  \on{Spec} A \arrow[r] \arrow[d] & \Mdk_{I,U} \arrow[d] \\
  \on{Spec} A' \arrow[r] \arrow[ur, dashed] & U\\
\end{tikzcd} 
\end{equation}
we need to show that a dashed arrow exists. To see this, recall from \ref{sect:affinepatches} the equations \ref{eqn:toriccoordinates} cutting out $\Mdk_{I,U} \subset U \times \mathbb{A}^{\Upsilon}$. Lifting the map $\on{Spec} A' \to U$ to $\Mdk_{I,U}$ requires specifying elements $(a_\gamma')_{\gamma \in \Upsilon}$ in $A'$ lifting the elements $(a_\gamma)_{\gamma \in \Upsilon}$ in $A$ coming from the map $\on{Spec} A \to \Mdk$ and still satisfying \ref{eqn:toriccoordinates}. But note that, as described in  \ref{sect:affinepatches}, in the equation \ref{eqn:toriccoordinates} we must substitute for $a_e$ the image of $l_e$ in $A'$. Since $A'$ is $I$-constrained and since all exponents $M_e$ are divisible by $I'(e)$ it turns out that all of the defining equations, except for $a_\gamma' a_{i(\gamma)}'=1$, become trivial. So indeed we can choose any lift $(a_\gamma')_{\gamma \in \Upsilon}$ of $(a_\gamma)_{\gamma \in \Upsilon}$ satisfying these equations. \end{proof}

\begin{lemma}\label{lem:lift_after_translate}
Assume that $M(A, A')$ is non-empty, and choose an element $\mu \in M(A, A')$. Then $A'$ is $I$-constrained if and only if there exists an element $\delta \in H^1(\ca C_p, \Omega^\vee(-P)\otimes J) \oplus H^1(\Gamma, J)$ such that $\mu + \delta\colon \on{Spec}A' \to \Mdk$ lands in $\DRL^{1/\k}$. 
\end{lemma}
\begin{proof}[Proof of \ref{lem:constrained_lifting} assuming \ref{lem:lift_after_translate}]
If $M(A, A')$ is empty then by \ref{lem:MAA_not_empty} we see that $A'$ is not $I$-constrained, and clearly no lift to an $A'$-point of $\DRL'$ exists. Hence we may as well assume $M(A, A')$ to be non-empty, hence a torsor under $H^1(\ca C_p, \Omega^\vee(-P)\otimes J) \oplus H^1(\Gamma, J)$. 

If $A'$ is not $I$-constrained then \ref{lem:lift_after_translate} shows that \emph{no} element of $M(A, A')$ lands in $\DRL^{1/\k}$. 

If $A'$ is $I$-constrained then by \ref{lem:lift_after_translate} there exists an element $\mu' \in M(A, A')$ which lands in $\DRL^{1/\k}$. To finish the proof we now need to show that given an element in $M(A,A')$ contained in $\DRL^{1/\k}$ we can construct another element of $M(A, A')$ which is also contained in $\DRL'$.


Recall that we have an \'etale coordinate chart $U \to \Mbar$ where $U \subset \bb A^M$ for some $M$. We have $\Mdk_{I,U} \subset U \times \bb A^\Upsilon \subset \bb A^M \times \bb A^\Upsilon = \bb A^N$ and by an affine linear transformation we can assume that our chosen point $p \in \DRL^{1/\k} \subset \Mdk_{I,U}$ maps to $0 \in \bb A^N$. Then we can obtain $\DRL'$ by intersecting $\DRL^{1/\k} \subset \Mdk_{I,U}$ with a generic linear subspace $H$ through the origin of codimension $2g-3+n$ (and localizing at $p$). Denote $W = T_0 \DRL_{\mathrm{red}}^{1/\k} \subset T_0 \bb A^N$ the tangent space to the reduced double ramification cycle, which we consider as a linear subspace of $\bb A^N$. Recall that since $\DRL^{1/\k}$ has dimension $2g-3+n$, the space $W$ also has dimension $2g-3+n$. As $H$ was assumed generic and the two linear subspaces $W,H$ are of complementary dimensions, there exists a linear projection $h:\bb A^N \to W$ with $h|_W = \on{id}_W$ and $h^{-1}(0)=H$. 

By assumption we have a lift  $\on{Spec}(A') \to \DRL^{1/\k} \subset \Mdk_{I,U}$ of the given $A$-point of $\DRL'$. For this to lie in $\DRL'=(\DRL^{1/\k} \cap H)_p$ we want that the composition $\phi: \on{Spec}(A') \to \Mdk_{I,U} \xrightarrow{h} W$ is zero. We know this is true on $\on{Spec}(A)$, since $\on{Spec}(A)$ factored through $\DRL'$. Thus the difference between $\phi$ and the zero map is an element $\epsilon \in (T_0 W) \otimes J$. But note we can shift our map $\on{Spec}(A') \to  \DRL^{1/\k} \subset \Mdk_{I,U}$ around by elements of $(T_0 W) \otimes J \subset (T_0 \Mdk_{I,U}) \otimes J$. Indeed, tangent vectors to the reduced $\DRL^{1/\k}$-component are locally trivial deformations (the reduced $\DRL^{1/\k}$-component is contained in the preimage of the boundary of $\Mbar$), so the shift by $W \otimes J$ does not change the map to $\bb A^E$ we want to lift. Also, clearly it does not change the composition with the Abel-Jacobi map, so we stay in $\DRL^{1/\k}$. But note that the tangent map $(T_0 W) \otimes J \subset (T_0 \Mdk_{I,U}) \otimes J \xrightarrow{Th} (T_0 W) \otimes J$ of $h$ is the identity, since $h$ was assumed to restrict to the identity of $W$. So indeed, we can shift our map $\on{Spec}(A') \to \DRL^{1/\k}$ by $-\epsilon \in (T_0 W) \otimes J$, to obtain an $A'$-point of $\DRL'$.
\end{proof}

\begin{proof}[Proof of \ref{lem:lift_after_translate}]
The sections $e$ and $\barsigma$ of the universal jacobian $\ca J$ induce a map $\Phi\colon M(A, A') \to T_e\ca J_p \otimes_\field J$, which is a pseudotorsor under the map
\begin{equation*}
\alpha = (b_\Omega \oplus b_\Gamma)\otimes \mathrm{id}_J \colon H^1(\ca C_p, \Omega^\vee(-p)\otimes J) \oplus H^1(\Gamma, J) \to T_e\ca J_{p}\otimes J
\end{equation*}
%
%
%

Choose an element $\mu \in M(A, A')$. We need to decide when there exists $\delta \in H^1(\ca C_p, \Omega^\vee(-p)\otimes J) \oplus H^1(\Gamma, J)$ such that $\Phi(\mu+\delta) = 0$, i.e. such that $\mu + \delta\colon \on{Spec} A' \to \Mdk_{I,U}$ lands in $\DRL^{1/\k}$. 

Now the existence of such a $\delta$ is equivalent to $\Phi(\mu)$ mapping to zero in the cokernel of $\alpha$, which is in turn equivalent to $\Phi(\mu)$ pairing to zero with the kernel of $\alpha^\vee$. Hence we are reduced to showing

\emph{(*) $A'$ is $I$-constrained if and only if $\Phi(\mu)$ pairs to zero with the kernel of $\alpha^\vee$. }

By \ref{lem:kernel_description} the kernel of $\alpha^\vee$ is given by the injection 
\begin{equation*}
\bigoplus_{v\in V^{out}} \field \to H^0(\ca C_p,  \omega)
\end{equation*}
sending $\bd c = (c_v)_v$ to the section given by $0$ on the smooth locus of $C_{v_0}$ for the central vertex $v_0$, and $c_v (\frac{\phi_0}{\bd 1_v})^{1/\k}$ on the smooth locus of $C_v$ for the outlying vertices $v$. By \ref{lem:res_pairing} (where the notation is also defined), the image of $\Phi(\mu) \bd c$ in $J$ under the residue pairing is given by
\begin{equation}\label{eqn:finalpairing}
 -\sum_{v \in V^{out}} c_v \sum_{e:v_0 \to v}a_{\gamma(e)}\ell_e^{I'(e)}\on{Res}_{q_e}(\phi_0^{1/\k}|_{C_{v_0}}/\bd 1_{v_0}^{1/\k}) \in J. 
\end{equation}
Note that the assumption $\ell_e^{I'(e)} \in J$ of \ref{lem:res_pairing} is exactly the fact that $A=A'/J$ is $I$-constrained.

Now if $A'$ is $I$-constrained, all terms $l_e^{I'(e)}$ are actually zero in $A'$, so the entire sum above vanishes, proving that $\Phi(\mu)$ pairs to zero with all $\bd c$.

Conversely assume that \ref{eqn:finalpairing} vanishes for all choices of $c_v$. From the fact that $\mu$ defines an $A'$-point of $\Mdk_{I,U}$, equation \ref{eqn:agammasimple} tells us that $a_{\gamma(e)} \ell_e^{I'(e)}=\ell_{e_v}^{I'(e_v)}$, so \ref{eqn:finalpairing} becomes
\begin{equation*}
 -\sum_{v \in V^{out}} c_v \ell_{e_v}^{I'(e_v)} \sum_{e:v_0 \to v}\on{Res}_{q_e}(\phi_0^{1/\k}|_{C_{v_0}}/\bd 1_{v_0}^{1/\k}). 
\end{equation*}
But the sums of residues appearing above are nonzero at a general point $p$ within its component by \ref{thm:non_vanishing_root_sum}. So for this sum to vanish for all choices of $c_w$, it is necessary that $l_{e_v}^{I'(e_v)}=0$ for all $v$ and since $a_{\gamma(e)} \ell_e^{I'(e)}=\ell_{e_v}^{I'(e_v)}$ with $a_{\gamma(e)}$ invertible in $A'$, it follows $\ell_e^{I'(e)}=0 \in A'$ for all $e$, finishing the proof.
\end{proof}

\subsection{Concluding the proof of the formula}\label{sec:concluding_the_proof}
Finally, we conclude the proof of the equality of the double ramification cycle $\overline\DRC$ and the cycle $H_{g,\m}^\k$ of Janda, Pandharipande, Pixton, Zvonkine and the second-named author (\ref{sec:FP_formula}) in the Chow ring $A^g(\Mbar)$.

Finally we can easily deduce the main result of this paper, that $\overline \DRC = H_{g,\m}^\k$. 
\begin{proof}[Proof of \ref{thm:main_intro}]

Let $\overline p$ be the general point of some component of $\widetilde{\mathcal{H}}_{g}^\k(\m)$. If $\overline p$ lies in the interior of the moduli space (i.e. $\ca C_{\overline p}$ is smooth), then $\overline{\DRC}$ has multiplicity $1$ at $\overline p$, agreeing with its multiplicity in $H_{g,\m}^\k$. This multiplicity follows from the computation in \cite[Proposition 1.2]{Schmitt2016Dimension-theor}.

Thus we can assume that $\ca C_{\overline p}$ lies in the boundary with associated simple star graph $\Gamma$ and positive twist $I$. Combining \ref{lem:generic_points_lie_over} and \ref{lem:compare_k_k_inverse}, it suffices to show that the sum of the lengths of the Artin local rings of $\DRL^{1/\k}$ at points lying over $\overline p$ is given by the formula $\frac{\prod_{e \in E(\Gamma)} I(e)}{\k^{\# V(\Gamma)+1}}$ (c.f. \ref{eq:weighted_fun_formula}). By \ref{pro:uniquegentwist} there are exactly $\k^{\# E - \# V - 1}$ of these points and by \ref{thm:structure_of_DR} the multiplicity of $\DRL^{1/\k}$ at each of them is $\prod_{e \in E} I'(e)$. Using $I'(e)=I(e)/\k$ the result is then immediate. 
\end{proof}

%

\subsection{Presentation of the local rings of the double ramification locus}

Our strategy for computing the multiplicities of the double ramification cycle was somewhat indirect, as we began \ref{sec:computing_length} by slicing with a generic hyperplane. In this section we do a little gentle bootstrapping to extract a presentation for the local rings of $\DRL^{1/\k}$ itself at generic points. We begin by resuming the notation from the start of \ref{sec:computing_length}, and write $\eta$ for the generic point of $\DRL^{1/\k}$ containing the point $p$ in its closure. The local ring $\ca O_{\DRL^{1/\k}, \eta}$ admits (by the Cohen Structure Theorem \cite[\href{https://stacks.math.columbia.edu/tag/0323}{Tag 0323}]{stacks-project}) a non-canonical structure as an algebra over $\kappa(\eta)$, the residue field of $\eta$ (compatible with the $K$-algebra structure). In this section we show 
\begin{theorem}\label{thm:local_ring_description}
 The $\kappa(\eta)$-algebra structure on $\ca O_{\DRL^{1/\k}, \eta}$ can be chosen so that 
\begin{equation}
\ca O_{\DRL^{1/\k}, \eta} \cong \kappa(\eta)\frac{[\ell_e : e \in E]}{(\ell_e^{I'(e)} : e \in E)}. 
\end{equation}
\end{theorem}
We are very close to proving this in \ref{thm:structure_of_DR}, except that latter concerns $\DRL'$, which is obtained from $\DRL^{1/\k}$ by cutting with a generic linear subspace (see the discussion at the beginning of \ref{sec:computing_length}). Noting that we can work over any field of characteristic zero, we can in particular base-change the whole setup to the residue field $\kappa(\eta)$. The gap is then filled by the following lemma. 
\begin{lemma}
Let $ Z \tra \bb A^N_K$ be an irreducible closed subscheme of dimension $d$, with generic point $\eta$. Write $\Delta\colon \eta \to \eta \times_K \eta$ for the diagonal (see the diagram below), and let $H \sub \bb A^N_\eta$ be a generic linear subspace of codimension $d$ through $\Delta$ (the latter viewed as a point of $\bb A^N_\eta$). Then there exists an isomorphism of complete $K$-algebras
\begin{equation}\label{eq:K-alg-iso}
\ca O_{Z, \eta} \cong \ca O_{Z_\eta \cap H, \Delta}. 
\end{equation}
\end{lemma}
Before the proof, we give a diagram (in which all squares are pullbacks) to illustrate the notation in the lemma:  
\begin{equation}
 \begin{tikzcd}
  \eta \arrow[d] & \eta \times_K \eta \arrow[l] \arrow[d] \\
    Z \arrow[d] & Z_\eta \arrow[l] \arrow[d] \\
\bb A^N_K \arrow[d] & \bb A^N_\eta \arrow[l] \arrow[d] \\
\on{Spec} K & \eta \arrow[l] \arrow[uuu, bend right, swap, "\Delta"]\\
\end{tikzcd}
\end{equation}
Since $\ca O_{Z_\eta \cap H, \Delta}$ is naturally a $\kappa(\eta)$-algebra, the isomorphism \ref{eq:K-alg-iso} naturally equips $\ca O_{Z, \eta}$ with the required $\kappa(\eta)$-algebra structure (dependent on the choice of $H$). 
\begin{proof}
The hyperplane $H$ is defined over the generic point $\eta$, and extends to a family of hyperplanes over some dense open subscheme $U \hra Z^{red}$. Writing $H_U \tra \bb A^N_U$ for this family, we note that $\dim_K H = N$, and we have a diagram of natural maps
\begin{equation}
 \begin{tikzcd}
  \bb A^N_K \arrow[d] &   \bb A^N_Z \arrow[d] \arrow[l]&   \bb A^N_U \arrow[d]\arrow[l] & H_U \arrow[d] \arrow[l]& H \arrow[l] \arrow[d]  \\
  K & Z \arrow[l] & U \arrow[l] & U \arrow[l, "id"]& \eta \arrow[u, bend right, swap, "\Delta"]\arrow[l] \\
\end{tikzcd}
\end{equation}
We claim that, since $H$ is generically chosen, the map $H_U \to \bb A^N_K$ is \'etale in a neighbourhood of $\Delta$. Since \'etaleness is an open condition it is enough to check this for some special $H$, and the claim is indeed clear if we take $H_U$ to be the normal bundle to $U$ in $\bb A^N_K$. 

We then have a diagram
\begin{equation}
 \begin{tikzcd}
 Z \arrow[d] & & Z_\eta \cap H \arrow[d]\arrow[ll] \\
  \bb A^N_K   & H_U \arrow[l] & H \arrow[l] & \eta \arrow[l, "\Delta"] \arrow[ul, "\Delta", swap]\\
\end{tikzcd}
\end{equation}
where the rectangle is a pullback, and the lower row is formally \'etale along $\Delta$, hence so is the upper row. Localising at the image of $\Delta$ thus yields a formally \'etale map of complete local $K$-algebras $\ca O_{Z_\eta \cap H, \Delta} \to \ca O_{Z, \eta}$ which is then necessarily an isomorphism. 
\end{proof}

\appendix

\section{Explicit deformation theory with \texorpdfstring{\v C}{C}ech cocycles}\label{sec:explicit_def_via_cech}
In this appendix we recall some standard results on deformation theory, to fix notation and to keep this paper reasonably self-contained. We emphasise explicit computations  with \v Cech covers. 

Let 
\begin{equation*}
0 \to J \to A' \to A \to 0
\end{equation*}
be a short exact sequence of $\field$-modules, where $A$ and $A'$ have the structure of (Artin local) $\field$-algebras, the map $A' \to A$ is a $\field$-algebra homomorphism, and $J\frak m_{A'} = 0$ where $\frak m_{A'}$ is the maximal ideal of $A'$. This generality will only be needed in \ref{sec:computing_length}; for computations of tangent spaces, it is enough to look at the special case
\begin{equation*}
0 \to t \field[t]/(t^2) \to \field[t]/(t^2) \to \field \to 0.
\end{equation*}

Before we start with deformation theory, we need to introduce some technical results, which we use later.
\begin{remark}\label{rem:J_as_k_module}
Now $J$ is an $A'$-module and $A'$ a $\field$-algebra, hence $J$ is also a $\field$-module. Writing $\bar a\in \field$ for the reduction of $a \in A'$ we see that $\bar a j = aj$ for all $j \in J$, since the difference between $a$ and the image of $\bar a$ in $A'$ lies in $\frak m_{A'}$, and $\frak m_{A'}J = 0$. 
\end{remark}
\begin{lemma}\label{lem:weird_iso}
Let $M$ be an $A'$-module, then the $\field$-bilinear map 
\begin{equation*}
f\colon M \times J \to \frac{M}{(\frak m_{A'} M)} \otimes_\field J; (m,j) \mapsto \bar m \otimes j
\end{equation*}
is in fact $A'$-bilinear, and the induced map 
\begin{equation*}
M \otimes_{A'} J \to \frac{M}{(\frak m_{A'} M)} \otimes_\field J
\end{equation*}
is an isomorphism. 
\end{lemma}
\begin{proof}
This is a special case of Bourbaki, Algebra I, Chapter 2, paragraph 3.6 (p.254). 
%
\end{proof}

We will apply the following well-known lemma to the jacobian of the universal curve. 
\begin{lemma} \label{Lem:Apseudotors}
Let $X$ be a $\field$-scheme and $\eta: \on{Spec} A \to X$ a map over $\field$ with image point $q \in X$. Denote 
\[M(\eta, A') = \{ \eta' : \on{Spec} A' \to X \text{ $\field$-morphism}: \eta'|_{\on{Spec}A} = \eta \}.\]
Then $M(\eta,A')$ is naturally a pseudotorsor under $T_{X,q} \otimes_\field J$. Moreover, for $g: X \to Y$ a morphism of $\field$-schemes, the natural map $M(\eta,A') \to M(g \circ \eta,A')$ is a pseudotorsor under the natural map 
\[Tg \otimes \on{id}_J: T_{X,q} \otimes_\field J \to T_{Y,g(q)} \otimes_\field J.\]
\end{lemma}

Let $\ca C_{A'} \to \on{Spec} A'$ be a family of stable curves (i.e. a map $\on{Spec} A' \to \Mbar)$ and $\mathcal{L}_{A'}$ a line bundle on $\ca C_{A'}$. Assume that the restriction $\mathcal{L}_{A}$ of $\mathcal{L}_{A'}$ to the fibre $\ca C_{A}=\ca C_{A'} \times_{A'} A$ is trivial, with a trivialising section $\phi_0 \in H^0(\ca C_A, \mathcal{L}_A)$. 

In particular, this implies that $\mathcal{L}_{A'}$ has multidegree zero, so the line bundles $\mathcal{L}_{A'}$ and $\mathcal{O}_{\ca C_{A'}}$ induce maps $\sigma_1, \sigma_2 : \on{Spec} A' \to \ca J$ into the universal Jacobian $\ca J \to \Mbar$. By the assumption $\mathcal{L}_A \cong \mathcal{O}_{\ca C_A}$, the restrictions  $\eta: \on{Spec} A \to \on{Spec} A' \to \ca J$ of these maps to $\on{Spec} A$ agree. Thus both give elements\footnote{$\ca J$ is only \'etale-locally a scheme, but this is enough for these infinitesimal considerations. } in $M(\eta,A')$, in the notation of \ref{Lem:Apseudotors}, with $X = \ca J$. Furthermore, their compositions with the projection $\ca J \to \Mbar$ to the moduli space of curves agree (on all of $\on{Spec} A'$), since for both $\sigma_1, \sigma_2$ the underlying family of curves is $\ca C_{A'}$.

Let $C=\ca C_{A'} \times_{A'} \field$, then by \ref{Lem:Apseudotors} the set $M(\eta,A')$ is a pseudotorsor under the group $T_{(C,\mathcal{O})} \ca J \otimes_\field J$, so the difference of $\sigma_1, \sigma_2$ gives a unique element $\delta \in T_{(C,\mathcal{O})} \ca J \otimes_\field J$. Furthermore, it must lie in the kernel of the map
\[T_{(C,\mathcal{O})} \ca J \otimes J \to T_{C} \Mbar \otimes J,\]
which is exactly $T_e \mathcal{J}_C \otimes J = H^1(C,\mathcal{O}_{C_\field}) \otimes J = H^1(C,\ca L_\field) \otimes J$, where $\mathcal{J}_C$ is the Jacobian of $C$, and the last isomorphism is via the restriction $\phi_0 |_{C}$ of $\phi_0$ to the fibre over $\field$. 

Our goal here is to describe how to obtain this element $\delta \in H^1(C,\ca L_\field) \otimes J$ using \v Cech cohomology for a suitable cover $\ca U=(U_i)_{i}$ of $\ca C_{A'}$. 

Suppose there exists $\ca U = \{U_i\}_{i \in I}$ an fpqc cover of $\ca C_{A'}$ by affines such that for every $i$ there exists a section $\phi_i \in H^0(U_i, \ca L_{A'}|_{U_i})$ with $\phi_i|_{U_i \times_{A'} A} = \phi_0|_{U_i \times_{A'} A}$. 


We fix a cover $\ca U$ and sections $\phi_i$ as above. For the overlaps $U_{i,j} = U_i \times_{\ca C_{A'}} U_j$ we see from the definition of the $\phi_i$ that 
\begin{equation} \label{eqn:samerestriction} \phi_i|_{U_{ij} \times_{A'} A} = \phi_0|_{U_{ij} \times_{A'} A} = \phi_j|_{U_{ij} \times_{A'} A}\end{equation}
for all $i,j$. The difference 
\[\psi_{i,j} = \phi_i - \phi_j  \in H^0(U_{ij}, \ca L_{A'}) \]
lies in the kernel of the `reduction mod $J$' map
\begin{equation} \label{eqn:redmodJmap}
H^0(U_{ij}, \ca L_{A'}) \to H^0(U_{ij}, \ca L_A). 
\end{equation}


We now want to identify this kernel with $H^0(U_{ij}, \ca L_\field) \otimes_\field J$. To see this note that we can interpret the sequence $0 \to J \to A' \to A \to 0$ as an exact sequence of sheaves on $\on{Spec} A'$. Since $\ca L_{A'}$ is flat over $A'$, we obtain an exact sequence 
\[0 \to \ca L_{A'} \otimes_{A'} J \to \ca L_{A'} \to \ca L_{A'} \otimes_{A} \to 0.\]
The map $\ca L_{A'} \to \ca L_{A'} \otimes_{A}$ induces \ref{eqn:redmodJmap} on global sections, so the kernel of that map is given by $H^0(U_{ij}, \ca L_{A'} \otimes_{A'} J) = H^0(U_{ij}, \ca L_{A'} )\otimes_{A'} J$. Here we use that $U_{ij}$ is affine. By \ref{lem:weird_iso} applied to $M = H^0(U_{ij}, \ca L_{A'} )$ this is naturally identified with 
\[(H^0(U_{ij}, \ca L_{A'} )\otimes_{A'} \field) \otimes_{\field} J = H^0(U_{ij}, \ca L_\field) \otimes_\field J.\]

\begin{lemma}\label{lem:class_of_deformation}
 For the cover $(U_i)_i$ of $C_\field=\ca C_{A'} \times_{A'} \field$, the element $(\psi_{i,j})_{i,j}$ defines a $1$-cocycle in $H^1(C, \ca L_\field) \otimes_\field J$, which represents the class $\delta \in T_e \mathcal{J}_C \otimes J$ we want to compute.
\end{lemma}

\section{Explicit Serre duality}\label{sec:serre_duality}
For the convenience of the reader, and to fix notation, we recall here the standard description of Serre duality on a curve in terms of \v Cech cocycles. 

We consider first the case of a smooth proper (possibly non-connected) curve $C/\field$. We choose a \v Cech cover $\ca U = \{U_i\}_i$ of $C$, where $i$ runs over some indexing set $I = \{0, \dots, n\}$. For a sheaf of abelian groups $\ca F$ on $C$ we write $C^i_\ca U(\ca F)$ for the group of \v Cech $i$-cochains, $Z^i_\ca U(\ca F)$ for the group of $i$-cocycles, and $\cecH^i(C, \ca F)$ for the $i$th \v Cech cohomology group. The point of Serre duality is that the `residue map' $\cecH^1(C, \omega) \to \field$ is an isomorphism of $\field$-vector spaces; our goal here is to make this residue map explicit. We approximately follow [Forster, section 17.2].

Write $\ca K$ for the sheaf of fractions of $\ca O_C$ on $C$. Fix an element $w = (w_{ij})_{i < j} \in Z^1_{\ca U}(\omega)$. Choose an element $\tilde w \in C^0_\ca U(\ca K \otimes_{\ca O_C} \omega)$ such that for all $i <j$ we have 
\[\tilde w_i - \tilde w_j=w_{ij}\in \omega(U_{ij})\sub \ca K \otimes \omega(U_{ij}).\]
For example, if all $U_0$ is dense in $C$ we could set $\tilde w_0 = 0$ and for $i \neq 0$ let $\tilde w_i$ be any meromorphic differential extending $w_{i0}$.
For a point $p \in C$, choose $i$ such that $p \in U_i$ and define $\on{res}_p \tilde w = \on{res}_p\tilde w_i$. To see that this is independent of the choice of $i$, note that if $p \in U_{ij}$ then $\tilde w_i - \tilde w_j$ is by assumption holomorphic around $p$, and so has zero residue. Finally, $\on{res}_p\tilde w$ is \emph{not} independent of the choice of $\tilde w$, but the \emph{global} residue $\sum_{p \in C} \on{res}_p\tilde w$ is independent of all choices, and we define this to be $\on{res} w$. This gives a well-defined residue map $H^1(C, \omega) \to \field$. 

Now for the case of nodal curves. We resume the notation from above, but we allow $C$ to have nodal singularities. We write $\pi\colon \tilde C \to C$ for the normalisation of $C$. Suppose again we are given $w = (w_{ij})_{i < j} \in Z^1_{\ca U}(\omega)$. Writing $\tilde{\ca U}$ for the cover of $\tilde C$ obtained by pulling back $\ca U$, we have a natural pullback map $Z^0_{\ca U}(\omega \otimes \ca K) \to Z^0_{\tilde{\ca U}}(\omega \otimes \ca K)$. We choose $\tilde w \in Z^0_{\ca U}(\omega \otimes \ca K)$ such that $\tilde w_i - \tilde w_j \in \omega(U_{ij})\sub \ca K \otimes \omega(U_{ij})$, and $\tilde w_i - \tilde w_j = w_{ij}$. Given a point $p \in C$ we choose $i$ with $p \in U_i$, and define $\on{res}_pw = \sum_{\pi(q) = p}\on{res}_q\pi^*\tilde w_i$. As before we should check that this is independent of the choice of $i$. For $p$ in the smooth locus of $C$ this proceeds exactly as before. If $p$ is a node, write $\pi^{-1}p = \{q, q'\}$. Assume\footnote{Generally in this article we assume that the intersections $U_{ij}$ do not contain any nodes, so we can ignore this case, but we treat it here from completeness. } $p \in U_{ij}$, then we need to show that 
\begin{equation}
\on{res}_q \tilde w_i + \on{res}_{q'} \tilde w_i = \on{res}_q \tilde w_j + \on{res}_{q'} \tilde w_j. 
\end{equation}
But since $\tilde w_i - \tilde w_j \in \omega(U_{ij})$ we have that the residues of $\tilde w_i - \tilde w_j$ at $q$ and $q'$ sum to zero, giving exactly the above equality. 

Thus for a given choice of $\tilde w$ we have a well-defined residue map at all points of $C$, and the \emph{global} residue $\sum_{p \in C} \on{res}_p\tilde w$ is independent of all choices, giving a well-defined residue map $\cecH^1(C, \omega) \to \field$.

\section{Generic non-vanishing of \texorpdfstring{$\k$}{k}-residues}\label{sec:generic_non-vanishing_residue}

In this section we are concerned with the vanishing of sums of $\k$th roots of $\k$-residues (defined just below) of $\k$-differentials on \emph{smooth} curves. We fix integers $g\ge 0$, $n \ge 2$ with $2g-2 + n > 0$. Let $\k >0$, and $\m$ a vector of $n$ integers $m_i$ summing to $\k(2g-2)$, and assume that some $m_i$ is either negative or not divisible by $\k$. Write $\DRL_g \tra \ca M_{g,n}$ for the locus of (smooth, marked) curves admitting a $\k$-differential with divisor $\sum_i m_i p_i$, and let $Y$ be an irreducible component of $\DRL_g$ (by \cite{Schmitt2016Dimension-theor}, such a $Y$ is necessarily smooth and of pure codimension $g$). Let $\eta$ denote a general $\field$-point of $Y$ (we assume in this section that $\field$ is algebraically closed, so that this exists; otherwise one simply works with the generic point, but the notation becomes slightly less convenient), and let $\xi_\eta$ be a differential on the curve $C_\eta$ with the prescribed divisor. 

We recall from \cite[prop 3.1]{Bainbridge2016Strata-of--k--d} the notion of $\k$-residue of a $\k$-differential $\xi$ on a smooth curve $C$. Assume that $\xi$ has multiplicity $m<0$ at a point $P \in C$ (i.e. a pole of order $|m|$ at $P$). Assume furthermore that $\k \mid m$. Then after suitable choice of local coordinate $z$ on $C$ (with $z=0$ at $P$) we can write 
\begin{equation*}
\xi = \left(\frac{s}{z}\right)^\k (dz)^\k,
\end{equation*}
for $m=-\k$ and
\begin{equation*}
\xi = \left(z^{m /\k} + \frac{s}{z}\right)^\k (dz)^\k,
\end{equation*}
for $m<-\k$, respectively. Here $s$ is an element of $\field$ whose $\k$-th power is well defined, denoted $\on{Res}^{\k}_P(\xi)=s^\k$, the \emph{$\k$-residue}. 

\begin{theorem}\label{thm:non_vanishing_root_sum}
Let $0 < n' < n$, and assume that $m_i < 0$ and $\k \mid m_i$ for all $1 \le i \le n'$. Also assume that $m_{n'+1}$ is either negative or not divisible by $\k$. For each $1 \le i \le n'$, let $r_i$ be \emph{any} $\k$th root of the $\k$-residue of $\xi_\eta$ at $p_i$. Then 
\begin{equation}\label{eq:residue_sum}
r_1 + \cdots + r_{n'} \neq 0,
\end{equation}
independent of the choices of $\k$th roots $r_i$.
\end{theorem}
For this result to hold, it is essential that the point $\eta$ be general in $Y$. In the case $\k=1$, we know that the sum of all the residues vanishes, and this result tells us that the sum of any proper subset of the residues is generically non-vanishing. Our result is closely related to those of Gendron and Tahar \cite{Gendron2017Differentielles}. The key difference is that on one hand we need to treat the connected components $Y$ of $\DRL_g$ separately but on the other hand, we are only interested at the behaviour at the generic point. Note that for $\k=1$ our result follows from \cite[Proposition 1.3]{Gendron2017Differentielles}. 


The proof will occupy the remainder of this section, and we break it into a number of steps. 

\textbf{Step 1: the case $\dim Y = 0$. }\\
In general we will argue by showing that the sum in \ref{eq:residue_sum} varies non-trivially in $Y$, and thus cannot vanish at a general point; this argument fails if $\dim Y = 0$, so we treat this case separately. Now $\dim Y = 2g-3 + n$ and $n \ge 2$, so we must have $g=0$, $n=3$. Note that in this case necessarily $n'=1$, since for $n'=2$ we have $m_1, m_2\leq -\k$ (since they are negative and divisible by $\k$). But $m_1 + m_2 + m_3=-2 \k$ forcing that $\k \mid m_3$ and  $m_3 = -2\k - m_1 - m_2 \geq 0$, a contradiction to the assumptions of the theorem.

So we are in the case $g=0, n'=1, n=3$. Then we have $\DRL_g=\Mbar_{0,3}$ is a single point and we can assume that $C=\mathbb{P}^1$ with $(p_1, p_2, p_3)=(0,1,\infty)$. Then $w$ is uniquely determined (up to scaling) as
\[w=z^{m_1} (1-z)^{m_2} (dz)^\k.\]
Let $g(z)$ be a $\k$th root of $(1-z)^{m_2}$ around $z=0$ with $g(0)=1$. Then for $\tilde m_1 = - m_1/\k$
we have
\[w=(z^{-\tilde m_1} g(z) dz)^\k\]
and 
\[\on{Res}^{\k}_0(w)=\left(\on{Res}_0(z^{-\tilde m_1} g(z) dz)) \right)^\k= \left(\frac{1}{(\tilde m_1 -1)!} \left(\frac{d}{dz}\right)^{\tilde m_1-1}
  g(z)\big|_{z=0} \right)^\k.\]
One verifies that to compute the derivative of $g$ we can just apply the usual rules for derivatives for the formula $g(z)=(1-z)^{m_2/\k}$ and obtain
\[\left(\frac{d}{dz}\right)^{b} g(z)\big|_{z=0} = (-1)^b \frac{m_2}{\k} \cdot \left(\frac{m_2}{\k}-1\right) \cdots \left(\frac{m_2}{\k}-b+1\right) (1-z)^{m_2/\k-b}\big|_{z=0}. \]
Since $m_2$ is either negative or not divisible by $\k$, this is a nonzero number for $b=-m_1/\k -1$ and thus the $\k$-residue of $w$ at $0$ does not vanish, as claimed.

%

\textbf{Step 2: Canonical covers of curves with $\k$-differentials}\\
We now move on to the general case, where $\dim Y >0$. Given a curve $C_\eta$ with the $\k$-differential $\xi_\eta$, we are going to use its \emph{canonical cover} $\pi:\widehat C_\eta \to C_\eta$ (see \cite{Bainbridge2016Strata-of--k--d}). This is a cyclic cover $\pi:\widehat C_\eta \to C_\eta$ of degree $\k$ obtained by extracting a $\k$th root of the section $\xi_\eta$ of the line bundle $\omega^{\otimes \k}(- \m P)$. 
This means that there exists a $1$-differential $\widehat \xi_\eta$ on $\widehat C_\eta$ with $(\widehat{\xi}_\eta)^\k = \pi^* \xi_\eta$.
Moreover, for  $\tau : \widehat C_\eta \to \widehat C_\eta$ an automorphism over $C_\eta$ generating the Galois group, it satisfies $\tau^* \widehat \xi_\eta = \rho_\k \widehat \xi_\eta$ where $\rho_\k$ is a primitive $\k$th root of unity. Note that the map $\pi$ is \'etale outside of the preimages of the points $p_i$ (since over points where the $\k$-differential $\xi_\eta$ is not zero, there are exactly $\k$ choices of a root).

There is a unique maximal $b \geq 1$ such that $\xi_\eta$ is a $b$th power of a $\k'=(\k/b)$-differential. The number $b$ is also the number of connected components of the cover $\widehat C_\eta$, and each such component is the canonical cover for the suitable $\k'$-differential on $C_\eta$. The component $Y$ of $\DRL_g$ is then just a component of a space of $\k'$-differentials and the $\k$-th roots $r_i$ of the $\k$-residues are exactly $\k'$-th roots of the corresponding $\k'$-residues. Since the canonical cover of the $\k'$-differential is connected, it suffices to show the statement of the theorem for connected canonical covers if we show it for all $\k \geq 1$. So from now on we assume that $\widehat C_\eta$ is connected.

Let $\widehat g$ be the genus of $\widehat C_\eta$. There are $\gcd(m_i, \k)$ preimages of each $p_i$ (and $\widehat g$ is determined by Riemann-Hurwitz). Then we write $\ca H$ for the stack of
`cyclic covers with the same degree and ramification data as $\widehat C_\eta\to C_\eta$'; more precisely,  the objects of $\ca H$ consist of
\begin{itemize}
\item
A (smooth, connected, proper) curve $C$ of genus $g$ with $n$ marked points $p_1, \dots, p_n$;
\item 
A (smooth, connected, proper)  curve $\widehat C$ of genus $\widehat g$ with $\sum_i \gcd(m_i, \k)$ marked points $q_{i,j}: 1 \le i \le n, 1 \le j \le \gcd(m_i, \k)$;
\item
A cyclic cover $\pi\colon \widehat C \to C$ of degree $\k$ mapping the $q_{i,j}$ to $p_i$. 
\end{itemize}
For a full definition and the properties of the stacks $\ca H$ that we will use, we refer the reader to \cite{SchmittvanZelm} and the references therein.

The stack $\ca H$ comes with a map 
\[\delta: {\ca H} \to \ca M_{g,n}\]
remembering the target curve $(C,(p_i)_i)$ and a map 
\[\phiorig: {\ca H} \to \ca M_{\widehat g, r}\]
remembering the domain curve $(\widehat C, (q_{i,j})_{i,j})$; here $r = \sum_{i=1}^n \gcd(m_i, \k)$. 
The map $\delta$ is \'etale 
and $\phiorig$ is unramified. 

Recall that we write $\DRL_{g} \subset \ca M_{g,n}$ for the `double ramification' locus where there exists a $\k$-differential $\xi$ with divisor $\m P$. 
If $q_{i,j}$ is a marked point on $\widehat C$ mapping to a marking $p_i$ on $C$, then the canonical cover $\widehat C \to C$ has multiplicity $f_i=\k/\gcd(m_i,\k)$ at $q_{i,j}$ and the $1$-differential $\widehat\xi$ has multiplicity $m_i' \coloneqq (m_i+\k)/\gcd(m_i,\k)-1$.
We write $\DRL_{\widehat g} \subset \ca M_{\widehat g,r}$ for the locus where there exists a $1$-differential $\widehat\xi$ with divisor $\m'Q = \sum_{i,j} m_i' q_{i,j}$. 

\begin{lemma}
Let $\pi:(\widehat C, (q_{i,j})_{i,j}) \to (C, (p_i)_i)$ be a point of ${\ca H}$ given by the canonical cover of a curve $(C, (p_i)_i) \in \DRL_g$. Then   $(\widehat C, (q_{i,j})_{i,j}) \in \DRL_{\widehat g}$, and inside ${\ca H}$, in a neighbourhood of the point $\pi:(\widehat C, (q_{i,j})_{i,j}) \to (C, (p_i)_i)$ of ${\ca H}$ we have
\begin{equation} \label{eqn:preimkdiffhurwitz} \phiorig^{-1}(\DRL_{\widehat g}) = \delta^{-1}(\DRL_{g}). \end{equation}
\end{lemma}
\begin{proof}
For the inclusion $\phiorig^{-1}(\DRL_{\widehat g}) \sub \delta^{-1}(\DRL_{g})$ let $\pi:(\widehat C, (q_{i,j})_{i,j}) \to (C, (p_i)_i)$ be a point of $\ca H$ such that $(\widehat C, (q_{i,j})_{i,j}) \in \DRL_{\widehat g}$, i.e. such that there exists a 1-differential $w$ on $\widehat C$ with multiplicity $m_i'$ at the points $q_{i,j}$. Then for the cyclic automorphism $\tau$ of the cover $\pi$ we have $\tau^* w = \lambda w$ for some $\lambda \in \field$, since $\tau^* w$ has the same pattern of zeros and poles as $w$. Since $\tau$ has order $\k$, it follows that $\lambda$ is a $\k$th root of unity. But then the $\k$th power $w^{\otimes \k}$ of $w$ is invariant under $\tau$, and hence descends to a $\k$-differential on $C$ with suitable zeros and poles. This shows $\phiorig^{-1}(\DRL_{\widehat g}) \sub \delta^{-1}(\DRL_{g})$. 

The other inclusion is not true globally, but we only need it on a neighbourhood of our point $\pi$ which already lies in $\phiorig^{-1}(\DRL_{\widehat g})$. If we can show that every infinitesimal deformation of $\pi$ which lies in $\delta^{-1}(\DRL_{g})$ also lies in $\phiorig^{-1}(\DRL_{\widehat g})$ then we are done, since all these moduli stacks are of finite presentation. 
A deformation $(\widehat C_t, (q_{i,j;t})_{i,j})$ of $\widehat C$ lying in $\delta^{-1}(\DRL_g)$ implies that the line bundle $\omega_{\widehat C_t}^{\otimes \k}(-\sum_{i,j} \k m_i' q_{i,j;t})$ is trivial, i.e. $\omega_{\widehat C_t}(-\sum_{i,j} m_i' q_{i,j;t})$ is $\k$-torsion. Since the $\k$-torsion points are discrete in the relative Picard of the family $\widehat C_t$ and since at $\widehat C=\widehat C_0$ this bundle is trivial (since $\widehat C \in \DRL_{\widehat g}$), it stays trivial in the deformation $\widehat C_t$, so $\widehat C_t \in \DRL_{\widehat g}$. 
%
%
\end{proof}

\textbf{Step 3: Tangent space computations}\\
We know that
\[T_{(\widehat C,(q_{i,j})_{i,j})}\ca M_{\widehat g,r} = H^1(\widehat C, \Omega^\vee_{\widehat C}(-\sum q_{i,j})) \]
and we have an action of $\mathbb{Z}/\k\mathbb{Z}$ on $\widehat C$ induced by the automorphism $\tau$ of $\widehat C$.  This in turn induces an action of $\mathbb{Z}/\k\mathbb{Z}$ on $T_{(\widehat C,(q_{i,j})_{i,j})}\ca M_{\widehat g,r}$, and
\begin{equation}\label{eqn:invarHurwitztangent} T_{(\widehat C,(q_{i,j})_{i,j})}{\ca H} = (T_{(\widehat C,(q_{i,j})_{i,j})}\ca M_{\widehat g,r})^{\mathbb{Z}/\k\mathbb{Z}},\end{equation}
where we see the tangent space to ${\ca H}$ as a subspace of the tangent space to $\ca M_{\widehat g,r}$ via the unramified map $\phiorig$. 

One also checks that the tangent space to  $\DRL_{\widehat g}$ ( contained in $T_{(\widehat C,(q_{i,j})_{i,j})}\ca M_{\widehat g,r}$) is stable under the $\mathbb{Z}/\k\mathbb{Z}$-action (for fixed $i$, all markings $q_{i,j}$, which form a $\mathbb{Z}/\k\mathbb{Z}$-orbit, have the same weight $m_i'$ in the definition of the double ramification locus $\DRL_{\widehat g}$). 

\textbf{Step 4: Residue maps on the tangent space}\\
From now on we focus on a curve $C$ coming from a general $\field$-point of the component $Y$ in the statement of the theorem, and take the cover $\pi\colon \widehat C \to C$ by extracting a $\k$th root of the given differential $\xi$, as in the previous step. Write $\widehat \xi$ for canonical $\k$th root of $\xi$ on $\widehat C$. Suppose that $\k \mid m_i$. Then there are exactly $\k$ markings $q_{i,j}$ of $\widehat C$ lying over $p_i$ (the cover $\pi$ is unramified there), and the residues of $\widehat \xi$ at the markings $q_{i,j}$ lying over $p_i$ are exactly the $\k$th roots of the $\k$-residue of $\xi$ at $p_i$. The chosen roots $r_1, \dots, r_{n'}$ in the statement of the theorem thus correspond uniquely to certain markings $q_{i,0}$ lying over the $p_i$; they are given exactly by the residue of $\widehat \xi$ at the $q_{i,0}$. 

Write $\widehat f \colon \widehat{\ca C} \to \DRL_{\widehat g}$  for the universal curve over $\DRL_{\widehat g}$. 
After perhaps restricting to an open subset of $Y$, on $\DRL_{\widehat g}$ the coherent sheaf $\widehat f_*\omega_{\widehat{\ca C}}(-\sum_{i,j} m_i' q_{i,j})$ is invertible. We write $\widehat n$ for the number of markings $q_{i,j}$ with negative weight, and we let $H \sub \ca O_{\DRL_{\widehat g}}^{\oplus \widehat n}$ be the subspace defined by the vanishing of the sum of all the coordinates. Then the (usual) residue gives a map 
\begin{equation}
R\colon \widehat f_*\omega_{\widehat{\ca C}}(-\sum_{i,j} m_i' q_{i,j}) \to H. 
\end{equation}

Let $\widehat C \to C$ in $\ca H$ lie over a general point in $Y$, and choose a non-zero section $\xi$ of $\omega^\k(-\m P)$ over $C$, leading to a differential $\widehat \xi$ over $\widehat C$. This gives a point in the total space of $\widehat f_*\omega_{\widehat{\ca C}}(-\sum_{i,j} m_i' q_{i,j})$, and we can consider the tangent map $TR$ at such a point. 

\textbf{Step 5: Concluding the proof with a lemma of Sauvaget}\\
%
%
Recall that we have reduced to the case where the curve $\widehat C$ is connected. 
We can then apply \cite[Corollary 3.8]{Sauvaget2017Cohomology-clas} to see that $TR$ is a \emph{surjection}. 

The morphism $\tau$ induces an automorphism of the pair $(\widehat C, \widehat\xi)$, where it acts on the differential by pulling back and dividing by $\rho_\k$. This induces an action of the group $\bb Z / \k \bb Z$ on the deformations of the pair, in other words on the source of $TR$. This group also acts on the target (by permutation of markings and multiplication by suitable roots of unity), and the map $TR$ is then equivariant for the action. Since $\bb Z / \k \bb Z$ is linearly reductive, the induced map on the invariant subspaces is also surjective.


On the left, the invariant subspace is the fibre of the source of $TR$ over 
\[T_{(\widehat C,(q_{i,j})_{i,j})}\overline{\ca H} \cap T_{(\widehat C,(q_{i,j})_{i,j})}\DRL_{\widehat g} = T_{(C,(p_i)_i)} \DRL_{g} \]
using the combination of \ref{eqn:invarHurwitztangent} and \ref{eqn:preimkdiffhurwitz}, and the fact that for a deformation of $(\widehat C, \widehat\xi)$ leaving the underlying curve $\widehat C$ fixed, the group $\bb Z / \k \bb Z$ also fixes the deformation of the differential. 
 On the right, the invariant subspace $H^{\bb Z/\k\bb Z}$ is the tangent space to the subspace of $H$ with coordinates of the form $[\ldots, s, \rho_\k s, \ldots, \rho_\k^{\k-1} s, \ldots]$ and the desired $\k$th roots $r_i$ of the $\k$-residues are just some of these values. We claim that the corresponding projection 
\[\pi_{n'}: T_{R(\widehat C,(q_{i,j})_{i,j}))}H^{\bb Z / \k \bb Z} \to \field^{n'}\] 
is surjective. If $\k >1$, the numbers $s, \rho_\k s, \ldots, \rho_\k^{\k-1}s$ automatically sum to zero, so the coordinates of $H$ summing to zero places no additional restriction on $s$. Thus every tuple $(r_, \ldots, r_{n'})\in \field^{n'}$ can be obtained from an element of $H^{\bb Z / \k \bb Z}$ under the map $\pi_{n'}$. On the other hand, for $\k=1$ we have $H^{\bb Z/\k\bb Z}=H$ and necessarily $m_{n'+1}<0$, so one of the coordinates of $H$ is the residue at $p_{n'+1}$ which is forgotten under $\pi_{n'}$. Thus we can use this coordinate to balance the sum of coordinates in $H$ to be zero and for any choice $(r_1, \ldots, r_{n'}) \in \field^{n'}$ find a preimage under $\pi_{n'}$. 

Since the invariant part of the source of $TR$ is the tangent space to $Y$, and the differential of the map taking the residues at the $n'$ points is surjective, and $\dim Y >0$, we see that at a general point of $Y$ the sum of these residues cannot be zero. This concludes the proof of \ref{thm:non_vanishing_root_sum}.

\bibliographystyle{alpha} 
\bibliography{../prebib}
%
%
%
%

\end{document}